\documentclass[11pt]{article}

\usepackage[leqno]{amsmath}
\usepackage{amssymb,amsthm,url}
\usepackage[letterpaper,margin=1in]{geometry}

\theoremstyle{plain}      \newtheorem{thm}{Theorem}[section]
                          \newtheorem{prop}[thm]{Proposition}
                          \newtheorem{lemma}[thm]{Lemma}
                          \newtheorem{cor}[thm]{Corollary}
\theoremstyle{definition} \newtheorem{example}[thm]{Example}

\theoremstyle{remark}	    \newtheorem*{rem}{Note}
                          \newtheorem*{rems}{Remarks}

\date{\today}\numberwithin{equation}{section}

\usepackage{graphicx}
\usepackage{enumerate}

\usepackage{varioref}
\labelformat{section}{Section~#1}
\labelformat{subsection}{Section~#1}
\labelformat{subsubsection}{Section~#1}
\labelformat{figure}{Figure~#1}

\date{}\usepackage{setspace}

\DeclareMathOperator\inv{inv}
\DeclareMathOperator\sep{sep}
\DeclareMathOperator\sym{sym}
\DeclareMathOperator\tv{TV}
\DeclareMathOperator\wt{wt}

\newcommand\real{\mathbb{R}}

\newcommand\calb{\mathcal{B}}
\newcommand\calc{\mathcal{C}}

\newcommand\calg{\mathcal{G}}
\newcommand\calh{\mathcal{H}}
\newcommand\calu{\mathcal{U}}
\newcommand\calx{\mathcal{X}}

\newcommand\hs{x^{[2]}}
\newcommand\hate{\hat{e}}
\newcommand\hatb{\hat{b}}
\newcommand\hatk{\hat{K}}
\newcommand\hatx{\hat{x}}

\newcommand\hatc{\hat{c}}

\newcommand\baralpha{\bar\alpha}
\newcommand\bard{\bar\Delta}

\newcommand\nk{\left[\begin{smallmatrix}n\\k\end{smallmatrix}\right]_q}
\newcommand\nj{\left[\begin{smallmatrix}n\\j\end{smallmatrix}\right]_q}

\newcommand\tref[1]{Theorem~\ref{#1}}
\newcommand\trefs[2]{Theorems~\ref{#1} and \ref{#2}}
\newcommand\pref[1]{Proposition~\ref{#1}}
\newcommand\lref[1]{Lemma~\ref{#1}}
\newcommand\coref[1]{Corollary~\ref{#1}}
\newcommand\exref[1]{Example~\ref{#1}}
\newcommand\exrefs[2]{Examples~\ref{#1} and \ref{#2}}

\begin{document}

\title{Hopf algebras and Markov chains: Two examples and a theory}

\author{\textsc{Persi Diaconis}\footnote{Supported in part by NSF grant DMS 0804324.}\ %
\\
\textit{Departments of}\\\textit{Mathematics and Statistics}\\\textit{Stanford University}\and
       \textsc{C. Y. Amy Pang}\footnote{Supported in part by NSF grant DMS 0652817.}
\footnote{Corresponding author: 450 Serra Mall, Stanford, CA 94305-4065, USA 
}\\
\textit{Department of}\\\textit{Mathematics}\\\textit{Stanford University}\and
       \textsc{Arun Ram}\footnote{Supported in part by ARC grant DP0986774.}\\
\textit{Department of}\\\textit{Mathematics and Statistics}\\\textit{University of Melbourne}}

\maketitle

\begin{abstract}
The operation of \textit{squaring}  (coproduct followed by product) in a combinatorial Hopf algebra is shown to induce a Markov chain in natural bases. Chains constructed in this way include widely studied methods of card shuffling, a natural ``rock-breaking'' process, and Markov chains on simplicial complexes. Many of these chains can be explictly diagonalized using the primitive elements of the algebra and the combinatorics of the free Lie algebra. For card shuffling, this gives an explicit description of the eigenvectors. For rock-breaking, an explicit description of the quasi-stationary distribution and sharp rates to absorption follow.
\end{abstract}

\section{Introduction}\label{sec1}

A Hopf algebra is an algebra $\calh$ with a coproduct $\Delta:\calh\to\calh\otimes\calh$ which fits together with the product $m:\calh\otimes\calh \rightarrow \calh$. Background on Hopf algebras is in \ref{sec22}. The map $m\Delta:\calh\to\calh$ is called the \textit{Hopf-square} (often denoted $\Psi^2$ or $\hs$). Our first discovery is that the coefficients of $\hs$ in natural bases can often be interpreted as a Markov chain. Specializing to familiar Hopf algebras can give interesting Markov chains: the free associative algebra gives the Gilbert--Shannon--Reeds model of riffle shuffling. Symmetric functions give a rock-breaking model of Kolmogoroff \cite{kolmogorov}. These two examples are developed first for motivation.
\begin{example}[Free associative algebra and riffle shuffling]
Let $x_1,x_2,\dots,x_n$ be noncommuting variables and $\calh=k\langle x_1,\dots,x_n\rangle$ be the free associative algebra. Thus $\calh$ consists of finite linear combinations of words $x_{i_1}x_{i_2}\cdots x_{i_k}$ in the generators with the concatenation product. The coproduct $\Delta$ is an algebra map defined by $\Delta(x_i)=1\otimes x_i+x_i\otimes1$ and extended linearly. Consider
\begin{equation*}
\Delta(x_{i_1}\cdots x_{i_k})=(1\otimes x_{i_1}+x_{i_1}\otimes1)(1\otimes x_{i_2}+x_{i_2}\otimes1)\cdots(1\otimes x_{i_k}+x_{i_k}\otimes1).
\end{equation*}
A term in this product results from a choice of left or right from each factor. Equivalently, for each subset $S\subseteq\{1,2,\dots,k\}$, there corresponds the term
\begin{equation*}
\prod_{j\in S}x_{i_j}\otimes\prod_{j\in S^\calc}x_{i_j}.
\end{equation*}
Thus $m\Delta$ is a sum of $2^k$ terms resulting from removing $\{x_{i_j}\}_{j\in S}$ and moving them to the front. For example,
\begin{equation*}
m\Delta(x_1x_2x_3)=x_1x_2x_3+x_1x_2x_3+x_2x_1x_3+x_3x_1x_2+x_1x_2x_3+x_1x_3x_2+x_2x_3x_1+x_1x_2x_3.
\end{equation*}
Dividing $m\Delta$ by $2^k$, the coefficient of a word on the right is exactly the chance that this word appears in a Gilbert--Shannon--Reeds inverse shuffle of a deck of cards labeled by $x_i$ in initial order $x_{i_1}x_{i_2}\dots x_{i_k}$. Applying $\frac{1}{2^k}m\Delta$ in the dual algebra gives the usual model for riffle shuffling. Background on these models is in \ref{sec5}. As shown there, this connection between Hopf algebras and shuffling gives interesting new theorems about shuffling.
\label{ex11}
\end{example}
\begin{example}[Symmetric functions and rock-breaking]
Let us begin with the rock-breaking description. Consider a rock of total mass $n$. Break it into two pieces according to the symmetric binomial distribution:
\begin{equation*}
P\{\text{left piece has mass $j$}\}=\binom{n}{j}\bigg/2^n,\qquad0\leq j\leq n.
\end{equation*}
Continue, at the next stage breaking each piece into $\{j_1,j-j_1\},\{j_2,n-j-j_2\}$ by independent binomial splits. The process continues until all pieces are of mass one when it stops. This description gives a Markov chain on partitions of $n$, absorbing at $1^n$.

This process arises from the Hopf-square map applied to the algebra $\Lambda=\Lambda(x_1,x_2,\dots,x_n)$ of symmetric functions, in the basis of elementary symmetric functions $\{e_{\lambda}\}$. This is an algebra under the usual product. The coproduct, following \cite{geissinger} is defined by
\begin{equation*}
\Delta(e_i)=e_0\otimes e_i+e_1\otimes e_{i-1}+\dots+e_i\otimes e_0,
\end{equation*}
extended multiplicatively and linearly. This gives a Hopf algebra structure on $\Lambda$ which is a central object of study in algebraic combinatorics. It is discussed in \ref{symfns}. Rescaling the basis elements to $\{\hate_i:=i!e_i\}$, a direct computation shows that $m\Delta$ in the $\{e_\lambda\}$ basis gives the rock-breaking process; see \ref{sec41}.
\label{ex12}
\end{example}

A similar development works for any Hopf algebra which is either a polynomial algebra as an algebra (for instance, the algebra of symmetric functions, with generators $e_n$), or is cocommutative and a free associative algebra as an algebra (e.g., the free associative algebra), provided each object of degree greater than one can be broken non-trivially. These results are described in \tref{thm31}.

Our second main discovery is that this class of Markov chains can be explicitly diagonalized using the Eulerian idempotent and some combinatorics of the free associative algebra. This combinatorics is reviewed in \ref{sec23}. It leads to a description of the left eigenvectors (\trefs{gefnthm1}{gefnthm2}) which is often interpretable and allows exact and asymptotic answers to natural probability questions. For a polynomial algebra, we are also able to describe the right eigenvectors completely (\tref{fefnthm1}).
\begin{example}[Shuffling]
For a deck of $n$ distinct cards, the eigenvalues of the Markov chain induced by repeated riffle shuffling are $1,1/2,\dots,1/2^{n-1}$ \cite{hanlon}. The multiplicity of the eigenvalue $1/2^{n-i}$ equals the number of permutations in $S_n$ with $i$ cycles. For example, the second eigenvalue, 1/2, has multiplicity $\binom{n}2$. For $1\leq i<j\leq n$, results from \ref{sec5} show that a right eigenvector $f_{ij}$ is given by
\begin{equation*}
f_{ij}(w)=\begin{cases}1,&\text{if $i$ and $j$ are adjacent in $w$ in order $ij$,}\\
-1,&\text{if $i$ and $j$ are adjacent in $w$ in order $ji$,}\\
0,&\text{otherwise.}\end{cases}
\end{equation*}
Summing in $i<j$ shows that $d(w)-\frac{n-1}2$ is an eigenvector with eigenvalue 1/2 ($d(w)=\#$ descents in $w$). Similarly $p(w)-\frac{n-2}3$ is an eigenvector with eigenvalue 1/4 ($p(w)=\#$ peaks in $w$). These eigenvectors are used to determine the mean and variance of the number of carries when large integers are added.

Our results work for decks with repeated values allowing us to treat cases when e.g., the suits do not matter and all picture cards are equivalent to tens. Here, fewer shuffles are required to achieve stationarity. For decks of essentially any composition we show that all eigenvalues $1/2^i,\ 0\leq i\leq n-1$, occur and determine multiplicities and eigenvectors.
\label{ex13}
\end{example}
\begin{example}[Rock-breaking]
Consider the rock-breaking process of \exref{ex12} started at $(n)$, the partition with a single part of size $n$. This is absorbing at the partition $1^n$. In \ref{sec4}, this process is shown to have eigenvalues $1,1/2,\dots,1/2^{n-1}$ with the multiplicity of $1/2^{n-l}$ the number of partitions of $n$ into $l$ parts. Thus, the second eigenvalue is 1/2 taken on uniquely at the partition $1^{n-2}2$. The corresponding eigenfunction is
\begin{equation*}
f_{1^{n-2}2}(\lambda)=\sum_i\binom{\lambda_i}2.
\end{equation*}
This is a monotone function in the usual partial order on partitions and equals zero if and only if $\lambda=1^n$. If $X_0=(n),X_1,X_2,\dots$ are the successive partitions generated by the Markov chain then
\begin{equation*}
E_{(n)}\left\{f_{1^{n-2}2}(X_k)\right\}=\frac1{2^k}f_{1^{n-2}2}(X_0)=\binom{n}2\bigg/2^k.
\end{equation*}
Using Markov's inequality,
\begin{equation*}
P_{(n)}\{\text{$X_k$ is not absorbed}\}\leq\binom{n}2\bigg/2^k.
\end{equation*}
This shows that for $k=2\log_2n+c$, the chance of absorption is asymptotic to $1-1/2^{c+1}$ when $n$ is large. \ref{sec4} derives all of the eigenvectors and gives further applications.
\label{ex14}
\end{example}

\ref{sec2} reviews Markov chains (including uses for eigenvectors), Hopf algebras, and some combinatorics of the free associative algebra. \ref{sec3} gives our basic theorems, generalizing the two examples to polynomial Hopf algebras and cocommutative, free associative Hopf algebras. \ref{sec4} treats rock-breaking; \ref{sec5} treats shuffling. \ref{sec6} briefly describes other examples (e.g., graphs and simplicial complexes), counter-examples (e.g., the Steenrod algebra), and questions (e.g., quantum groups).

Two historical notes: The material in the present paper has roots in work of Patras \cite{patras91,patras93,patras94}, whose notation we are following, and Drinfeld \cite{drinfeld}. Patras studied shuffling in a purely geometric fashion, making a ring out of polytopes in $\real^n$. This study led to natural Hopf structures, Eulerian idempotents, and generalization of Solomon's descent algebra in a Hopf context. His Eulerian idempotent maps decompose a graded commutative or cocommutative Hopf algebra into eigenspaces of the $a$th Hopf-powers; we improve upon this result, in the case of polynomial algebras or cocommutative, free associative algebras, by algorithmically producing a full eigenbasis. While there is no hint of probability in the work of Patras, it deserves to be much better known. More detailed references are given elsewhere in this paper.

We first became aware of Drinfeld's ideas through their mention in Shnider--Sternberg \cite{ss}. Consider the Hopf-square, acting on a Hopf algebra $\calh$. Suppose that $x\in\calh$ is primitive, $\Delta(x)=1\otimes x+x\otimes1$. Then $m\Delta(x)=2x$ so $x$ is an eigenvector of $m\Delta$ with eigenvalue 2. If $x$ and $y$ are primitive then $m\Delta(xy+yx)=4(xy+yx)$ and, similarly, if $x_1,\dots, x_k$ are primitive then the sum of symmetrized products is an eigenvector of $m\Delta$ with eigenvector $2^k$. Drinfeld \cite[~Prop 3.7]{drinfeld} used these facts without comment in his proof that any formal deformation of the cocommutative universal enveloping algebra $\calu(\mathfrak{g})$ results already from deformation of the underlying Lie algebra $\mathfrak{g}$. See \cite[Sect.~3.8]{ss} and \ref{sec333} below for an expanded argument and discussion. For us, a description of the primitive elements and their products gives the eigenvectors of our various Markov chains. This is developed in \ref{sec3}.

Acknowledgements: We thank Marcelo Aguiar, Federico Ardila, Nantel Bergeron, Megan Bernstein, Dan Bump, Gunner Carlsson, Ralph Cohen, David Hill, Peter J. McNamara, Susan Montgomery and Servando Pineda for their help and suggestions, and the anonymous reviewer for a wonderfully detailed review.

\section{Background}\label{sec2}

This section gives notation and background for Markov chains (including uses for eigenvectors), Hopf algebras, the combinatorics of the free associative algebra and symmetric functions. All of these are large subjects and pointers to accessible literature are provided.

\subsection{Markov chains}\label{sec21}

Let $\calx$ be a finite set. A Markov chain on $\calx$ may be specified by a transition matrix $K(x,y)$ $(x,y\in \calx)$ with $K(x,y)\geq0,\ \sum_yK(x,y)=1$. This is interpreted as the chance that the chain moves from $x$ to $y$ in one step. If the chain is denoted $X_0,X_1,X_2,\dots$ and $X_0=x_{0}$ is a fixed starting state then
\begin{equation*}
P\{X_1=x_1, \cdots, X_n=x_n\}=\prod_{i=0}^{n-1}K(x_i,x_{i+1}).
\end{equation*}
Background and basic theory can be found in \cite{karlin} or \cite{bremaud}. The readable introduction \cite{levin} is recommended as close in spirit to the present paper. The analytic theory is developed in \cite{saloff}.

Let $K^2(x,y)=\sum_zK(x,z)K(z,y)$ denote the probability of moving from $x$ to $y$ in two steps. Similarly, $K^l$ is defined. Under mild conditions \cite[Sec ~1.5]{levin} Markov chains have unique stationary distributions $\pi(x)$: thus $\pi(x)\geq0,\ \sum_x\pi(x)=1$, $\sum_x\pi(x)K(x,y)=\pi(y)$, so $\pi$ is a left eigenvector of $K$ with eigenvalue 1. Set
\begin{equation*}
L^2(\pi)=\{f:\calx\to\real\}\quad\text{with}\quad\langle f_1|f_2\rangle=\sum f_1(x)f_2(x)\pi(x).
\end{equation*}
Then $K$ operates as a contraction on $L^2$ with $Kf(x)=\sum_yK(x,y)f(y)$. The Markov chains considered in this paper are usually not self-adjoint (equivalently reversible), nonetheless, they are diagonalizable over the rationals with eigenvalues $1=\beta_0\geq\beta_1\geq\dots\geq\beta_{|\calx|-1}>0$. We have a basis of left eigenfunctions $\{g_i\}_{i=0}^{|\calx|-1}$ with $g_0(x)=\pi(x)$ and $\sum_xg_i(x)K(x,y)=\beta_ig_i(y)$, and, in some cases, a dual basis of right eigenfunctions $\{f_i\}_{i=0}^{|\calx|-1}$ with $f_0(x)\equiv1$, $Kf_i(x)=\beta_if_i(x)$, and $\sum_x f_i(x) g_j(x)=\delta_{ij}$. As is customary in discussions of random walks on algebraic structures, we will abuse notation and think of the eigenfunctions $f_i$ both as functions on the state space and as linear combinations of the states - in other words, $\sum_x f_i(x) x$ will also be denoted $f_i$.

Throughout, we are in the unusual position of knowing $\beta_i,\ g_i$ and possibly $f_i$ explicitly. This is rare enough that some indication of the use of eigenfunctions is indicated.

\paragraph*{Use A} For any function $f:\calx\to\real$, expressed in the basis of right eigenfunctions $\{f_i\}$ as
\begin{equation}
f=\sum_{i=0}^{|\calx|-1}a_if_i,
\label{21}
\end{equation}
the expectation of $f$ after $k$ steps, having started at $x_0$, is given by
\begin{equation}
E_{x_0}\left\{f(X_k)\right\}=\sum_{i=0}^{|\calx|-1}a_i\beta_i^kf_i(x_0).
\label{22}
\end{equation}
For example, for shuffling, the normalized number of descents $d(\pi)-(n-1)/2$ is the sum of the 1/2-eigenfunctions for riffle shuffling; see \exref{ex55}. Thus, with $x_0=id$ and all $k$, $0\leq k<\infty$,
\begin{equation*}
E_{id}\left\{d(X_k)\right\}=\frac{n-1}2\left(1-\frac1{2^k}\right).
\end{equation*}
In \cite{diacfulman09a,diacfulman09b} it is shown that the number of descents in repeated riffle shuffles has the same distribution as the number of carries when $n$ integers are added. Further, the square of this eigenfunction has a simple eigenfunction expansion leading to simple formulae for the variance and covariance of the number of carries.

\paragraph*{Use B} If $f$ is a right eigenfunction with eigenvalue $\beta$, then the self-correlation after $k$ steps (starting in stationarity) is
\begin{equation*}
E_\pi\left\{f(X_0)f(X_k)\right\}=E_\pi\left\{E\left\{f(X_0)f(X_k)|X_0
=x_0\right\}\right\}=\beta^kE_\pi\left\{f(X_0)^2\right\}.
\end{equation*}
This indicates how certain correlations fall off and gives an interpretation of the eigenvalues.

\paragraph*{Use C} For $f$ a right eigenfunction with eigenvalue $\beta$, let $Y_i=f\left(X_i\right) / \beta^i,\ 0 \leq i < \infty$. Then $Y_i$ is an $\mathcal{F}_i$ martingale with $\mathcal{F}_i=\sigma\left(X_0,X_1, \dots , X_i\right)$. One may try to use optional stopping, maximal and concentration inequalities and the martingale central limit theorem to study the behavior of the original $X_i$ chain.

\paragraph*{Use D} One standard use of right eigenfunctions is to prove lower bounds for mixing times for Markov chains. The earliest use of this is the second moment method \cite{diaconis}. Here, one uses the second eigenfunction as a test function and expands its square in the eigenbasis to get concentration bounds. An important variation is Wilson's method \cite{wilson} which only uses the first eigenfunction but needs a careful understanding of the variation of this eigenfunction. A readable overview of both methods and many examples is in \cite{saloffcoste}.

\paragraph*{Use E} The left eigenfunctions come into computations since $\sum_xg_i(x)f_j(x)=\delta_{ij}$. Thus in \eqref{21}, $a_i=\langle g_i|f/\pi\rangle$. (Here $f/\pi$ is just the density of $f$ with respect to $\pi$.)

\paragraph*{Use F} A second prevalent use of left eigenfunctions throughout this paper: the dual of a Hopf algebra is a Hopf algebra and left eigenfunctions of the dual chain correspond to right eigenfunctions of the original chain. This is similar to the situation for time reversal. If $K^*(x,y)=\frac{\pi(y)}{\pi(x)}K(y,x)$ is the time-reversed chain (note $K^*(x,y)$ is a Markov chain with stationary distribution $\pi$), then $g_i/\pi$ is a right eigenfunction of $K^*$.

\paragraph*{Use G} The left eigenfunctions also come into determining the quasi-stationary distribution of absorbing chains such as the rock-breaking chain. A useful, brief introduction to quasi-stationarity is in \cite{karlin}. The comprehensive survey \cite{vandoorn} and annotated bibliography \cite{pollett} are also useful. Consider the case where there is a unique absorbing state $x_{\bullet}$ and the second eigenvalue $\beta_1$ of the chain satisfies $1=\beta_0>\beta_1>\beta_2\geq\dots$. This holds for rock-breaking. There are two standard notions of ``the limiting distribution of the chain given that it has not been absorbed'':
\begin{subequations}\begin{align}
\pi^1(x)&=\lim_{k\to\infty}P\{X_k=x|X_k\neq x_{\bullet}\};\label{23a}\\
\pi^2(x)&=\lim_{k\to\infty}\lim_{l\to\infty}P\{X_k=x|X_l\neq x_{\bullet}\}.\label{23b}
\end{align}\label{23}
\end{subequations}
In words, $\pi^1(x)$ is the limiting distribution of the chain given that it has not been absorbed up to time $k$ and $\pi^2(x)$ is the limiting distribution of the chain given that it is never absorbed. These quasi-stationary distributions can be expressed in terms of the eigenfunctions:
\begin{equation}
\pi^1(x)=\frac{g_1(x)}{\sum_{y\neq \bullet} g_1(y)},\qquad\pi^2(x)=\frac{g_1(x)f_1(x)}{\sum_{y\neq \bullet} g_1(y)f_1(y)}.
\label{24}
\end{equation}
These results follow from simple linear algebra and are proved in the references above. For rock-breaking, results in \ref{sec4} show that $\pi^1=\pi^2$ is point mass at the partition $1^{n-2}2$.

\paragraph*{Use H} Both sets of eigenfunctions appear in the formula
\begin{equation}
K^l(x,y)=\sum_{i=0}^{|\calx|-1}\beta_i^lf_i(x)g_i(y).
\label{25}
\end{equation}
This permits the possibility of determining convergence rates. It can be difficult to do for chains with large state spaces. See the examples and discussion in \cite{diacfulman11}.

To conclude this discussion of Markov chains we mention that convergence is customarily measured by a few standard distances:
\begin{alignat}3
\text{Total variation}&\qquad&\|K_{x_0}^l-\pi\|_{\tv}&=\max_{A\subseteq\calx}|K_{x_0}^l(A)-\pi(A)|
     =\frac12\sum_y|K^l(x_0,y)-\pi(y)|\label{26}\\
\text{Separation}&\qquad&\sep_{x_0}(l)&=\max_y1-\frac{K^l(x_0,y)}{\pi(y)}\label{27}\\
\text{Sup}&\qquad&l_\infty(l)&=\max_y\left|\frac{K^l(x_0,y)-\pi(y)}{\pi(y)}\right|\label{28}
\end{alignat}
Here $\|K_{x_0}^l-\pi\|_{\tv}\leq\sep_{x_0}(l)\leq l_\infty(l)$ and all distances are computable by determining the maximizing or minimizing values of $A$ or $y$ and using \eqref{25}-\eqref{28}. See \cite[Lemma 6.13]{levin} for further discussion of these distances.

\subsection{Hopf algebras}\label{sec22}

A Hopf algebra is an algebra $\calh$ over a field $k$ (usually the real numbers in the present paper). It is associative with unit 1, but not necessarily commutative. Let us write $m$ for the multiplication in $\calh$, so $m(x\otimes y)=xy$. Then $m^{[a]}:\calh^{\otimes a}\to\calh$ will denote $a$-fold products (so $m=m^{[2]}$), formally $m^{[a]}=m(\iota \otimes m^{[a-1]})$ where $\iota$ denotes the identity map.

$\calh$ comes equipped with a \textit{coproduct} $\Delta:\calh\to\calh\otimes\calh$, written $\Delta(x)=\sum_{(x)} x_{(1)}\otimes x_{(2)}$ in Sweedler notation \cite{sweedler}. The coproduct is \textit{coassociative} in that 
\begin{equation*}
(\Delta \otimes \iota) \Delta(x) = \sum_{(x),\left(x_{(1)}\right)} x_{(1)_{(1)}}\otimes x_{(1)_{(2)}}\otimes x_{(2)}= \sum_{(x), (x_{(2)})} x_{(1)}\otimes x_{(2)_{(1)}}\otimes x_{(2)_{(2)}} = (\iota \otimes \Delta) \Delta(x)
\end{equation*}
so there is no ambiguity in writing $\Delta^{[3]} (x)=\sum_{(x)}x_{(1)} \otimes x_{(2)} \otimes x_{(3)}$. Similarly, $\Delta^{[a]}:\calh\to\calh ^{\otimes a}$ denotes the $a$-fold coproduct, where $\Delta$ is applied $a-1$ times, to any one tensor-factor at each stage; formally $\Delta^{[a]}=(\iota \otimes \dots \otimes \iota \otimes \Delta) \Delta^{[a-1]}$. The Hopf algebra $\calh$ is \textit{cocommutative} if $\sum_{(x)} x_{(1)} \otimes x_{(2)} = \sum_{(x)}  x_{(2)} \otimes x_{(1)}$; in other words, an expression in Sweedler notation is unchanged when the indices permute. An element $x$ of $\calh$ is \textit{primitive} if $\Delta(x)=1\otimes x + x \otimes 1$.

The product and coproduct have to be compatible so $\Delta$ is an algebra homomorphism, where multiplication on $\calh \otimes \calh$ is componentwise; in Sweedler notation this says $\Delta(xy)=\sum_{(x),(y)}  x_{(1)}y_{(1)}\otimes x_{(2)}y_{(2)}$. All of the algebras considered here are graded and connected, i.e., $\calh=\bigoplus_{i=0}^\infty\calh_i$ with $\calh_0=k$ and $\calh_n$ finite-dimensional. The product and coproduct must respect the grading so $\calh_i\calh_j\subseteq\calh_{i+j}$, and $x\in\calh_n$ implies $\Delta(x)\in\bigoplus_{j=0}^n\calh_j\otimes\calh_{n-j}$. There are a few more axioms concerning a counit map and an antipode (automatic in the graded case); for the present paper, the most important is that the counit is zero on elements of positive degree, so, by the coalgebra axioms, $\bard(x):=\Delta(x) - 1 \otimes x - x \otimes 1 \in\bigoplus_{j=1}^{n-1}\calh_j\otimes\calh_{n-j}$, for $x\in \calh_n$. The free associative algebra and the algebra of symmetric functions, discussed in \ref{sec1}, are examples of graded Hopf algebras.

The subject begins in topology when H.\ Hopf realized that the presence of the coproduct leads to nice classification theorems which allowed him to compute the cohomology of the classical groups in a unified manner. Topological aspects are still a basic topic \cite{hatcher} with many examples which may provide grist for the present mill. For example, the cohomology groups of the loops on a topological space form a Hopf algebra, and the homology of the loops on the suspension of a wedge of circles forms a Hopf algebra isomorphic to the free associative algebra of \exref{ex11} \cite{bott}. 

Joni and Rota \cite{joni} realized that many combinatorial objects have a natural breaking structure which gives a coalgebra structure to the graded vector space on such objects. Often there is a compatible way of putting pieces together, extending this to a Hopf algebra structure. Often, either the assembling or the breaking process is symmetric, leading to commutative or cocommutative Hopf algebras respectively. For example, the symmetric function algebra is commutative and cocommutative while the free associative algebra is just cocommutative. 

The theory developed here is for graded commutative or cocommutative Hopf algebras with one extra condition, that there is a unique way to assemble any given collection of objects. This amounts to the requirement that the Hopf algebra is either a polynomial algebra as an algebra (and therefore commutative) or a free associative algebra as an algebra and cocommutative (and therefore noncommutative). (We write \textit{a} free associate algebra to refer to the algebra structure only, as opposed to \textit{the} free associative algebra which has a specified coalgebra structure - namely, the generating elements are primitive.)

Increasingly sophisticated developments of combinatorial Hopf algebras are described by \cite{aguiar06,schmitt93, schmitt94, schmitt87,schmitt95} and \cite{aguiar10}. This last is an expansive extension which unifies many common examples. Below are two examples that are prototypes for their Bosonic Fock functor and Full Fock functor constructions respectively \cite[Ch.~15]{aguiar10}; they are also typical of constructions detailed in other sources.

\begin{example}[The Hopf algebra of unlabeled graphs] \cite[Sec.~12]{schmitt94} \cite[Sec.~3.2]{fisher}
Let $\bar{\calg}$ be the vector space spanned by unlabeled simple graphs (no loops or multiple edges). This becomes a Hopf algebra with product disjoint union and coproduct
\begin{equation*}
\Delta(G)=\sum G_S \otimes G_{S^{\calc}}
\end{equation*}
where the sum is over subsets of vertices $S$ with $G_S,\ G_{S^{\calc}}$ the induced subgraphs. Graded by number of vertices, $\bar{\calg}$ is both commutative and cocommutative, and is a polynomial algebra as an algebra. The associated random walk is described in \exref{ex31} below.
\label{ex21}
\end{example}

\begin{example}[The noncommutative Hopf algebra of labeled graphs] \cite[Sec.~13]{schmitt94} \cite[Sec.~3.3]{fisher}
Let $\calg$ be the vector space spanned by the set of simple graphs where vertices are labeled $\{1,2,\dots,n\}$, for some $n$. The product of two graphs $G_1G_2$ is their disjoint union, where the vertices of $G_1$ keep their labels, and the labels in $G_2$ are increased by the number of vertices in $G_1$. The coproduct is \begin{equation*}
\Delta(G)=\sum G_S \otimes G_{S^{\calc}}
\end{equation*}
where we again sum over all subsets $S$ of vertices of $G$, and $G_S,\ G_{S^{\calc}}$ are relabeled so the vertices in each keep the same relative order. For example,
\begin{center}
\includegraphics[scale=0.4, clip]{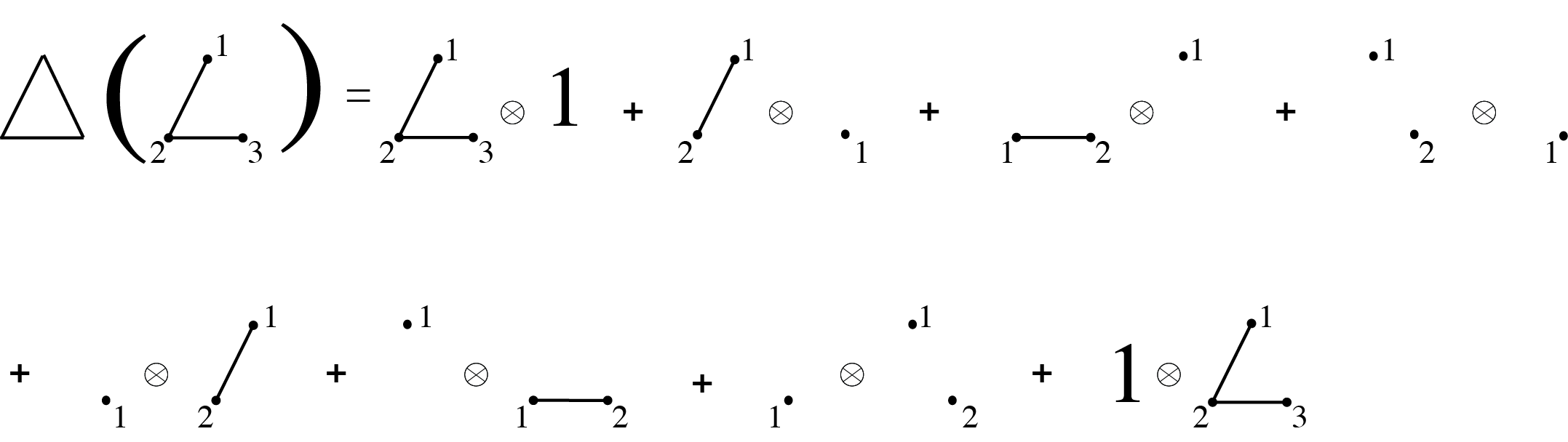}
\end{center}
\noindent
where $1$ denotes the empty graph. $\calg$ is noncommutative and cocommutative and a free associative algebra as an algebra; the associated random walk is detailed in \exref{exncgraphs2}. As the notation suggests, $\bar{\calg}$ is a quotient of $\calg$, obtained by forgetting the labels on the vertices.
\label{exncgraphs}
\end{example}

Aguiar--Bergeron--Sottile \cite{aguiar06} define a combinatorial Hopf algebra as a Hopf algebra $\calh$ with a character $\zeta:\calh\to k$ which is both additive and multiplicative. They prove a universality theorem: any combinatorial Hopf algebra has a unique character-preserving Hopf morphism into the algebra of quasisymmetric functions. They show that this unifies many ways of building generating functions. When applied to the Hopf algebra of graphs, their map gives the chromatic polynomial. In \ref{sec35} we find that their map gives the probability of absorption for several of our Markov chains. See also the examples in \ref{sec6}.

A good introduction to Hopf algebras is in \cite{ss}. A useful standard reference is in \cite{montgomery} and our development does not use much outside of her Chapter 1. The broad-ranging text \cite{majid} is aimed towards quantum groups but contains many examples useful here. Quantum groups are neither commutative nor cocommutative and need special treatment; see \exref{ex62}.

A key ingredient in our work is the Hopf-square map $\Psi^2 =m\Delta$; $\Psi^2 (x)$ is also written $x^{[2]}$. In Sweedler notation, $\Psi^2(x)= \sum_{(x)}x_{(1)}x_{(2)}$; in our combinatorial setting, it is useful to think of  ``pulling apart'' $x$ according to $\Delta$, then using the product to put the pieces together. On graded Hopf algebras, $\Psi^2$ preserves the grading and, appropriately normalized, gives a Markov chain on appropriate bases. See \ref{sec32} for assumptions and details. The higher power maps $\Psi^a=m^{[a]}\Delta^{[a]}$ will also be studied, since under our hypothesis, they present no extra difficulty. For example, $\Psi^3(x)=\sum_{(x)} x_{(1)}x_{(2)}x_{(3)}$. In the shuffling example, $\Psi^a$ corresponds to the ``$a$-shuffles'' of \cite{bayer}. A theorem of \cite{tate} shows that, for commutative or cocommutative Hopf algebras, the power rule holds: $(x^{[a]})^{[b]}=x^{[ab]}$, or $\Psi^a \Psi^b = \Psi^{ab}$. See also the discussion in \cite{landers}. In shuffling language this becomes ``an $a$-shuffle followed by a $b$-shuffle is an $ab$-shuffle'' \cite{bayer}. In general Hopf algebras this power law often fails \cite{kashina}. Power maps are actively studied as part of a program to carry over to Hopf algebras some of the rich theory of groups. See \cite{guralnick,linchenko} and their references.

\subsection{Structure theory of a free associative algebra}\label{sec23}

The eigenvectors of our Markov chains are described using combinatorics related to the free associative algebra, as described in the self-contained \cite[Chap.~5]{lothaire}. 

A word in an ordered alphabet is \textit{Lyndon} if it is strictly smaller (in lexicographic order) than its cyclic rearrangements. So $1122$ is Lyndon but $21$ or $1212$ are not. A basic fact \cite[Th.~5.1.5]{lothaire} is that any word $w$ has a unique \textit{Lyndon factorization}, that is, $w=l_1l_2\cdots l_k$ with each $l_i$ a Lyndon word and $l_1\geq l_2\geq\dots\geq l_k$. Further, each Lyndon word $l$ has a \textit{standard factorization}: if $l$ is not a single letter, then $l=l_1l_2$ where $l_i$ is non-trivial Lyndon and $l_2$ is the longest right Lyndon factor of $l$. (The standard factorization of a letter is just that letter by definition.) Thus $13245=13\cdot245$. Using this, define, for Lyndon $l$, its \textit{standard bracketing} $\lambda(l)$ recursively by $\lambda(a)=a$ for a letter and $\lambda(l)=[\lambda(l_1),\lambda(l_2)]$ for $l=l_1l_2$ in standard factorization. As usual, $[x,y]=xy-yx$ for words $x,y$. Thus
\begin{align*}
\lambda(13245)&=[\lambda(13),\lambda(245)]=\left[[1,3],\left[2,[4,5]\right]\right]\\
&=13245-13254-13452+13542-31245+31254+31452-31542 \\
& \quad -24513+25413+45213-54213+24531-25431-45231+54231 
\end{align*}
and
\begin{align*}
\lambda(1122)&=[1,\lambda(122)]=\left[1,[\lambda(12),2]\right]\\
&=1122-2(1212)+2(2121)-2211.
\end{align*}

\cite[Sec.~2]{garsia89} describes how to visualize the standard bracketing of a Lyndon word as a rooted binary tree: given a Lyndon word $l$ with standard factorization $l=l_1l_2$, inductively set $T_l$ to be the tree with $T_{l_1}$ as its left branch and $T_{l_2}$ as its right branch. $T_{13245}$ and $T_{1122}$ are shown below.
\begin{center}
\includegraphics[scale=1.0, clip]{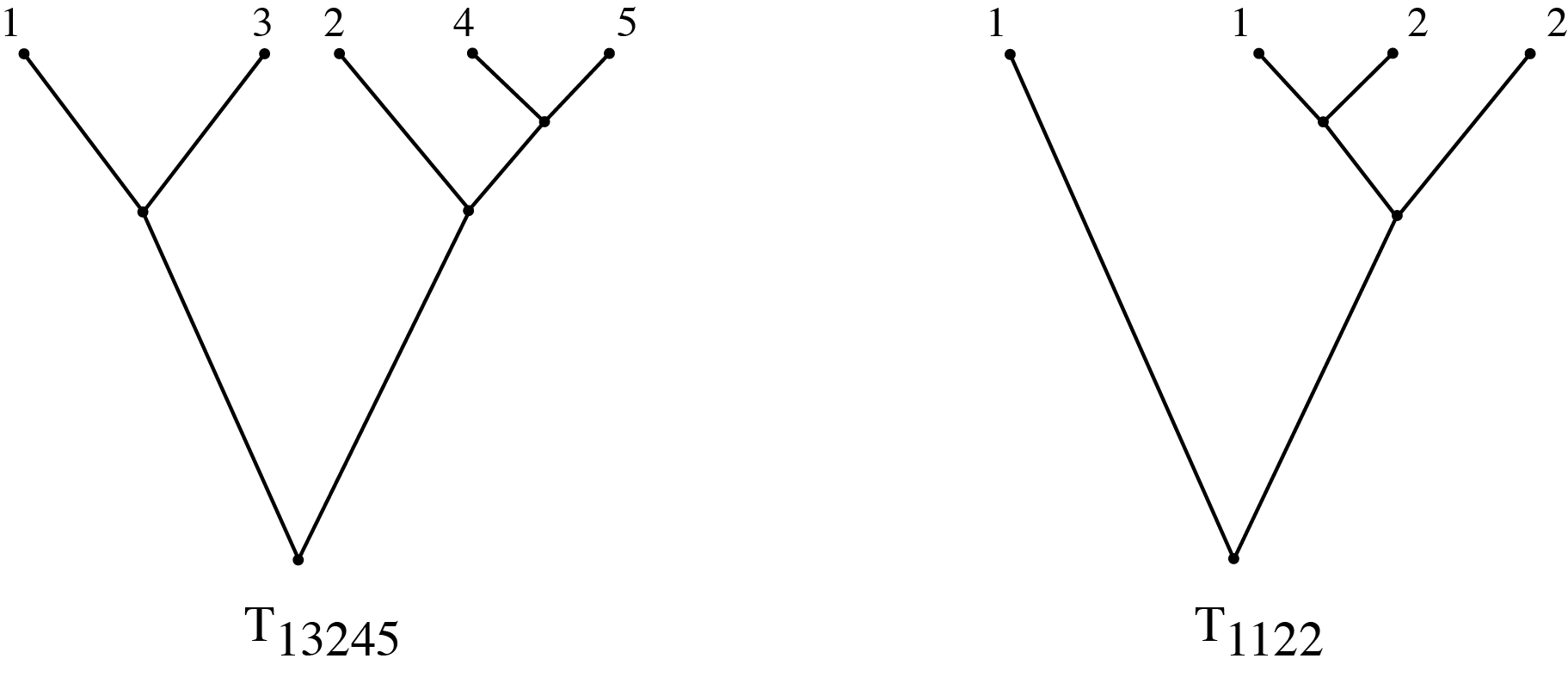}
\end{center}
\noindent

Observe that a word $w$ appears in the expansion of $\lambda(l)$ only if, after exchanging the left and right branches at some vertices of $T_l$, the leaves of $T_l$, when read from left to right, spell out $w$. The coefficient of $w$ in $\lambda(l)$ is then the signed number of ways to do this (the sign is the parity of the number of exchanges required). For example,
\begin{itemize}
\item 25413 has coefficient 1 in $\lambda(13245)$ since the unique way to rearrange $T_{13245}$ so the leaves spell 25413 is to exchange the branches at the root and the highest interior vertex;
\item 21345 does not appear in $\lambda(13245)$ since whenever the branches of $T_{13245}$ switch, 2 must appear adjacent to either 4 or 5, which does not hold for 21345;
\item 1221 has coefficient 0 in $\lambda(1122)$ as, to make the leaves of $T_{1122}$ spell 1221, we can either exchange branches at the root, or exchange branches at both of the other interior vertices. These two rearrangements have opposite signs, so the signed count of rearrangments is 0.
\end{itemize}

A final piece of notation is the following symmetrized product: let $w=l_1l_2\cdots l_k$ in Lyndon factorization. Then set
\begin{equation*}
\sym(w)=\sum_{\sigma\in S_k}\lambda(l_{\sigma(1)})\lambda(l_{\sigma(2)})\cdots\lambda(l_{\sigma(k)}).
\end{equation*}
Viewing $\sym(w)$ as a polynomial in the letters $w_1, w_2, \dots, w_l$ will be useful for \tref{gefnthm2}. 

Garsia and Reutenauer's tree construction can be extended to visualize $\sym(w)$, using what Barcelo and Bergeron \cite{barcelobergeron} call \textit{decreasing Lyndon hedgerows}, which simply consist of $T_{l_1},T_{l_2},\dots,T_{l_k}$ placed in a row. Denote this as $T_w$ also. The example $T_{35142}$ is shown below.
\begin{center}
\includegraphics[scale=1.0, clip]{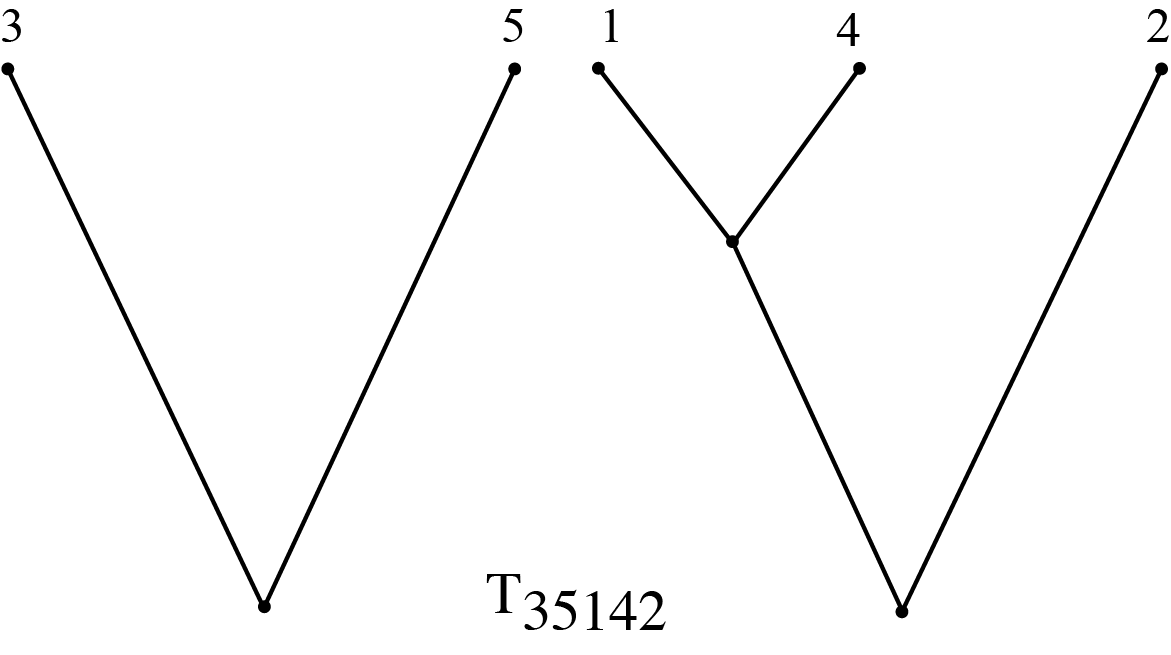}
\end{center}
\noindent

We can again express the coefficient of $w'$ in $\sym(w)$ as the signed number of ways to rearrange $T_w$ so the leaves spell $w'$. Now there are two types of allowed moves: exchanging the left and right branches at a vertex, and permuting the trees of the hedgerow. The latter move does not come with a sign. Thus 14253 has coefficient $-1$ in $\sym(35142)$, as the unique rearrangement of $T_{35142}$ which spells 14253 requires transposing the trees and permuting the branches labeled 3 and 5.

It is clear from this pictorial description that every term appearing in $\sym(w)$ is a permutation of the letters in $w$. \cite[Th.~5.2]{garsia89} shows that $\{\sym(w)\}$ form a basis for a free associative algebra. This will turn out to be a left eigenbasis for inverse riffle shuffling, and similar theorems hold for other Hopf algebras.

\subsection{Symmetric functions and beyond}\label{symfns}

A basic object of study is the vector space $\Lambda_k^n$ of homogeneous symmetric polynomials in $k$ variables of degree $n$. The direct sum $\Lambda_k=\bigoplus_{n=0}^\infty\Lambda_k^n$ forms a graded algebra with familiar bases: the monomial ($m_\lambda$), elementary ($e_\lambda$), homogeneous ($h_\lambda$), and power sums ($p_\lambda$). For example, $e_2(x_1,\cdots, x_k)=\sum_{1\leq i<j\leq k}x_ix_j$ and for a partition  $\lambda=\lambda_1\geq\lambda_2\geq\dots\geq\lambda_l>0$ with $\lambda_1+\dots+\lambda_l=n,\ e_\lambda=e_{\lambda_1}e_{\lambda_2}\cdots e_{\lambda_l}$. As $\lambda$ ranges over partitions of $n,\ \{e_\lambda\}$ form a basis for $\Lambda_k^n$, from which we construct the rock-breaking chain of \exref{ex12}. Splendid accounts of symmetric function theory appear in \cite{macdonald} and \cite{stanley99}. A variety of Hopf algebra techniques are woven into these topics, as emphasized by \cite{geissinger} and \cite{zelevinsky}. The comprehensive account of noncommutative symmetric functions \cite{gelfand} and its follow-ups furthers the deep connection between combinatorics and Hopf algebras. However, this paper will only involve its dual, the algebra of quasisymmetric functions, as they encode informations about absorption rates of our chains, see \ref{sec35}. A basis of this algebra is given by the monomial quasisymmetric functions: for a composition $\alpha=(\alpha_1,\dots,\alpha_k)$, define $M_{\alpha}=\sum_{i_1 < i_2 < \dots <i_k} x_{i_1}^{\alpha_1} \cdots x_{i_k}^{\alpha_k}$. Further details are in \cite[Sec.~7.19]{stanley99}.

\section{Theory}\label{sec3}
\subsection{Introduction}\label{sec31}

This section states and proves our main theorems. This introduction sets out definitions. \ref{sec32} develops the reweighting schemes needed to have the Hopf-square maps give rise to Markov chains. \ref{sec322} explains that these chains are often acyclic. \ref{sec333} addresses a symmetrization lemma that we will use in \ref{sec33} and \ref{sec34} to find descriptions of some left and right eigenvectors respectively for such chains. \ref{sec35} determines the stationary distributions and gives expressions for the chance of absorption in terms of generalized chromatic polynomials. Applications of these theorems are in the last three sections of this paper.

As mentioned at the end of \ref{sec22}, we will be concerned with connected, graded (by positive integers) Hopf algebras $\calh$ with a distinguished basis $\calb$ satisfying one of two ``freeness" conditions (in both cases, the number of generators may be finite or infinite):
\begin{enumerate}
\item $\calh=\mathbb{R}\left[c_1,c_2,\dots\right]$ as an algebra (i.e., $\calh$ is a polynomial algebra) and $\calb=\left\{ c_1^{n_1}c_2^{n_2}\dots|n_i\in \mathbb{N}\right\}$, the basis of monomials. The $c_i$ may have any degree, and there is no constraint on the coalgebra structure. This will give rise to a Markov chain on combinatorial objects where assembling is symmetric and deterministic.
\item $\calh$ is cocommutative, $\calh=\mathbb{R}\left<c_1,c_2,\dots\right>$ as an algebra, (i.e., $\calh$ is a free associative algebra) and $\calb=\left\{ c_{i_1}c_{i_2}\dots|i_j\in \mathbb{N}\right\}$, the basis of words. The $c_i$ may have any degree, and do not need to be primitive. This will give rise to a Markov chain on combinatorial objects where pulling apart is symmetric, assembling is non-symmetric and deterministic.
\end{enumerate}
By the Cartier-Milnor-Moore theorem \cite{milnor,cartier}, any graded connected commutative Hopf algebra has a basis which satisfies the first condition. However, we will not make use of this, since the two conditions above are reasonable properties for many combinatorial Hopf algebras and their canonical bases. For example, the Hopf algebra of symmetric functions, with the basis of elementary symmetric functions $e_{\lambda}$, satisfies the first condition.

Write $\calh_n$ for the subspace of degree $n$ in $\calh$, and $\calb_n$ for the degree $n$ basis elements. The \textit{generators} $c_i$ can be identified as those basis elements which are not the non-trivial product of basis elements; in other words, generators cannot be obtained by assembling objects of lower degree. Thus, all basis elements of degree one are generators, but there are usually generators of higher degree; see \exrefs{ex31}{exncgraphs2} below. One can view the conditions 1 and 2 above as requiring the basis elements to have unique factorization into generators, allowing the convenient view of $b \in \calb$ as a word $b=c_1c_2\cdots c_l$. Its \textit{length} $l(b)$ is then well-defined - it is the number of generators one needs to assemble together to produce $b$. Some properties of the length are developed in \ref{sec322}. For a noncommutative Hopf algebra, it is useful to choose a linear order on the set of generators refining the ordering by degree: i.e. if $\deg(c)<\deg(c')$, then $c<c'$. This allows the construction of the Lyndon factorization and standard bracketing of a basis element, as in \ref{sec23}. \exref{gefnex} demonstrates such calculations.

The $a$th Hopf-power map is $\Psi^a:=m^{[a]}\Delta^{[a]}$, the $a$-fold coproduct followed by the $a$-fold product. These power maps are the central object of study of \cite{patras91,patras93,patras94}. Intuitively, $\Psi^a$ corresponds to breaking an object into $a$ pieces (some possibly empty) in all possible ways and then reassembling them. The $\Psi^a$ preserve degree, thus mapping $\calh_n$ to $\calh_n$. 

As noted in \cite{patras93}, the power-map $\Psi^a$ is an algebra homomorphism if $\calh$ is commutative:
\begin{align*}
\Psi^a(xy)&=m^{[a]}\sum_{(x),(y)}x_{(1)}y_{(1)} \otimes \cdots \otimes x_{(a)}y_{(a)} \\
&=\sum_{(x),(y)}x_{(1)}y_{(1)} \dots x_{(a)}y_{(a)} = \sum_{(x),(y)}x_{(1)}\dots x_{(a)} y_{(1)} \dots y_{(a)} =\Psi^a(x) \Psi^a(y);
\end{align*}
and a coalgebra homomorphism if $\calh$ is cocommutative:
\begin{align*}
\left(\Psi^a\otimes \Psi^a \right)(\Delta x)&=\sum_{(x)}\Psi^a \left(x_{(1)}\right)\otimes \Psi^a \left(x_{(2)}\right)\\
&=\sum_{(x)}x_{(1)}\dots x_{(a)} \otimes x_{(a+1)} \dots x_{(2a)} = \sum_{(x)}x_{(1)} x_{(3)}\dots x_{(2a-1)}\otimes  x_{(2)} x_{(4)} \dots x_{(2a)} =\Delta\Psi^a(x).
\end{align*}
Only the former will be necessary for the rest of this section.

\subsection{The Markov chain connection}\label{sec32}

The power maps can sometimes be interpreted as a natural Markov chain on the basis elements $\calb_n$ of $\calh_n$.
\begin{example}[The Hopf algebra of unlabeled graphs, continuing from \exref{ex21}]
The set of all unlabeled simple graphs gives rise to a Hopf algebra $\bar{\calg}$ with disjoint union as product and
\begin{equation*}
\Delta(G)=\sum G_S \otimes G_{S^{\calc}}
\end{equation*}
where the sum is over subsets of vertices $S$ with $G_S,\ G_{S^{\calc}}$ the induced subgraphs. Graded by the size of the vertex set, $\bar{\calg}$ is a commutative and cocommutative polynomial Hopf algebra with basis $\calb$ consisting of all graphs. The generators are precisely the connected graphs, and the length of a graph is its number of connected components.

The resulting Markov chain on graphs with $n$ vertices evolves as follows: from $G$, color the vertices of $G$ red or blue, independently with probability 1/2. Erase any edge with opposite colored vertices. This gives one step of the chain; the process terminates when there are no edges. Observe that each connected component breaks independently; that $\Delta$ is an algebra homomorphism ensures that, for any Hopf algebra, the generators break independently. The analogous Hopf algebra of simplicial complexes is discussed in \ref{sec6}.
\label{ex31}
\end{example}

\begin{example}[The noncommutative Hopf algebra of labeled graphs, continuing from \exref{exncgraphs}]
Let $\calg$ be the linear span of the simple graphs whose vertices are labeled $\{1,2,\dots,n\}$ for some $n$. The product of two graphs $G_1G_2$ is their disjoint union, where the vertices of $G_1$ keep their labels, and the labels in $G_2$ are increased by the number of vertices in $G_1$. The coproduct is
\begin{equation*}
\Delta(G)=\sum G_S \otimes G_{S^{\calc}}
\end{equation*}
where the sum again runs over all subsets $S$ of vertices of $G$, and $G_S,\ G_{S^{\calc}}$ are relabeled so the vertices in each keep the same relative order. An example of a coproduct calculation is in \exref{exncgraphs}. $\calg$ is cocommutative and a free associative algebra; its distinguished basis $\calb$ is the set of all graphs. A graph in $\calg$ is a product if and only if there is an $i$ such that no edge connects a vertex with label $\leq i$ to a vertex with label $>i$. Thus, all connected graphs are generators, but there are non-connected generators such as 
\begin{center}
\includegraphics[scale=1.0, clip]{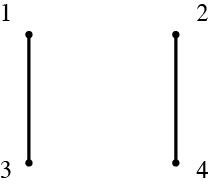}
\end{center}
\noindent

Each step of the associated random walk on $\calb_n$, the graphs with $n$ vertices, has this description:  from $G$, color the vertices of $G$ red or blue, independently with probability 1/2. Suppose $r$ vertices received the color red; now erase any edge with opposite colored vertices, and relabel so the red vertices are ${1,2,\dots,r}$ and the blue vertices are ${r+1,r+2,\dots,n}$, keeping their relative orders. For example, starting at the complete graph on three vertices, the chain reaches each of the graphs shown below with probability $1/8$:
\begin{center}
\includegraphics[scale=1.0, clip]{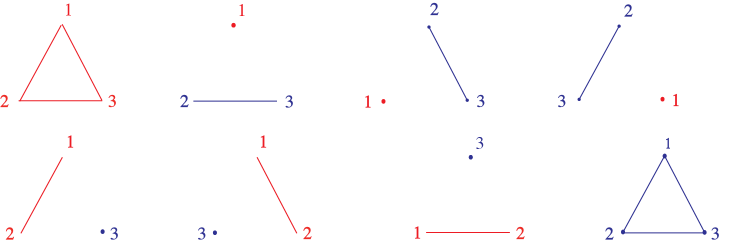}
\end{center}
\noindent
So, forgetting the colors of the vertices,
\begin{center}
\includegraphics[scale=1.0, clip]{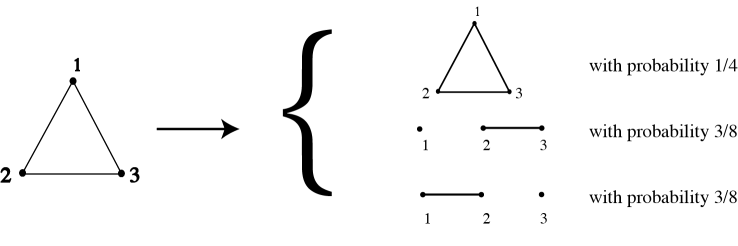}
\end{center}
\noindent
As with $\bar{\calg}$, the chain on $\calg_n$ stops when all edges have been removed.
\label{exncgraphs2}
\end{example}

When is such a probabilistic interpretation possible? To begin, the coefficients of $m\Delta(b)$ must be non-negative real numbers for $b\in\calb$. This usually holds for combinatorial Hopf algebras, but the free associative algebra and the above algebras of graphs have an additional desirable property: for any $b\in\calb$, the coefficients of $\Psi^2(b)$ sum to $2^{\deg(b)}$, regardless of $b$. Thus the operator $\frac1{2^n}\Psi^2(b)=\sum_{b'}K(b,b')b'$ forms a Markov transition matrix on basis elements of degree $n$. Indeed, the coefficients of $\Psi^a(b)$ sum to $a^{\deg(b)}$ for all $a$, so  $\frac1{a^n}\Psi^a(b)=\sum_{b'}K_a(b,b')b'$  defines a transition matrix $K_a$. For other Hopf algebras, the sum of the coefficients in $\Psi^2(b)$ may depend on $b$, so simply scaling $\Psi^2$ does not always yield a transition matrix. 

Zhou's rephrasing \cite[Lemma 4.4.1.1]{zhou} of the Doob transform \cite[Sec.17.6.1]{levin} provides a solution: if $K$ is a matrix with non-negative entries and $\phi$ is a strictly positive right eigenfunction of $K$ with eigenvalue 1, then $\hatk(b,b'):=\phi(b)^{-1}K(b,b')\phi(b')$ is a transition matrix. Here $\hatk$ is the conjugate of $K$ by the diagonal matrix whose entries are $\phi(b)$. \tref{thm31} below gives conditions for such $\phi$ to exist, and explicitly constructs $\phi$ recursively; \coref{scaling} then specifies a non-recursive definition of $\phi$ when there is a sole basis element of degree 1. The following example explains why this construction is natural.

\begin{example}[Symmetric functions and rock-breaking]
Consider the algebra of symmetric functions with basis $\{e_\lambda\}$, the elementary symmetric functions. The length $l(e_\lambda)$ is the number of parts in the partition $\lambda$, and the generators are the partitions with a single part. The coproduct is defined by
\begin{align*}
\Delta(e_i)&=\sum_{j=0}^ie_j\otimes e_{i-j}\qquad\text{so}\\
\Delta(e_\lambda)&=\Delta(e_{\lambda_1})\cdots\Delta(e_{\lambda_l})
=\left(\sum_{i_1=0}^{\lambda_1}e_{i_1}\otimes e_{\lambda_1-i_1}\right)\cdots
\left(\sum_{i_l=0}^{\lambda_l}e_{i_l}\otimes e_{\lambda_l-i_l}\right)\\
&=\sum_{\lambda'\leq\lambda}e_{\lambda'}\otimes e_{\lambda-\lambda'}
\end{align*}
with the sum over all compositions $\lambda'=\lambda_1',\lambda_2',\dots,\lambda_l'$ with $0\leq\lambda'_i\leq\lambda_i$, and $\lambda-\lambda'$ is the composition $\lambda_1-\lambda_1',\dots,\lambda_l-\lambda_l'$. When reordered, some parts may be empty and some parts may occur several times. There are $(\lambda_1 +1) \cdots (\lambda_l+1)$ possible choices of $\lambda '$, so the coefficients of $\Psi^2 (e_\lambda)$ sum to $(\lambda_1 +1) \cdots (\lambda_l+1)$, which depends on $\lambda$. 

Consider degree 2, where the basis elements are $e_{1^{2}}$ and $e_{2}$. For $K$ such that $\frac{1}{2^{2}}\Psi^{2}(b)=\sum_{b'}K(b,b')b'$,
 \[
K=\begin{bmatrix}
1 & 0\\
\frac{1}{4} & \frac{1}{2}\end{bmatrix},\]
which is not a transition matrix as the second row does not sum to
1. Resolve this by performing a diagonal change of basis: set $\hate_{1^{2}}=\phi(e_{1^{2}})^{-1}e_{1^{2}}$,
$\hate_{2}=\phi(e_{2})^{-1}e_{2}$ for some non-negative function $\phi:\calb \rightarrow \real$,
and consider $\hatk$ with $\frac{1}{2^{2}}\Psi^{2}(\hatb)=\sum_{\hatb'}\hatk(\hatb,\hatb')\hatb'$.
Since the first row of $K$, corresponding to $e_{1^{2}}$, pose no
problems, set $\phi(e_{1^{2}})=1$. In view of the upcoming theorem,
it is better to think of this as $\phi(e_{1^{2}})=\left(\phi(e_{1})\right)^{2}$
with $\phi(e_{1})=1$. Equivalently, $\hate_{1^2} = \hate_1 ^2$ with $\hate_1 =e_1$. Turning attention to the second row, observe that
$\Delta(\hate_{2})=\phi(e_{2})^{-1}(e_{2}\otimes1+e_{1}\otimes e_{1}+1\otimes e_{2})$,
so $\Psi^{2}(\hate_{2})=\hate_{2}+\phi(e_{2})^{-1}\hate_{1^{2}}+\hate_{2}$,
which means \[
\hatk=\begin{bmatrix}
1 & 0\\
\frac{1}{4}\phi(e_{2})^{-1} & \frac{1}{2}\end{bmatrix},\]
so $\hatk$ is a transition matrix if $\frac{1}{4}\phi(e_{2})^{-1}+\frac{1}{2}=1$,
i.e. if $\phi(e_{2})= \frac{1}{2}$.

Continue to degree 3, where the basis elements are $e_{1^{3}}$, $e_{12}$
and $e_{3}$. Now define $K$ such that $\frac{1}{2^{3}}\Psi^{2}(b)=\sum_{b'}K(b,b')b'$;
\[
K=\begin{bmatrix}
1 & 0 & 0\\
\frac{1}{4} & \frac{1}{2} & 0\\
0 & \frac{1}{4} & \frac{1}{4}\end{bmatrix}.\]
Again, look for $\phi(e_{1^{3}}),\phi(e_{12})$ and $\phi(e_{3})$ so
that $\hatk$, defined by $\frac{1}{2^{3}}\Psi^{2}(\hatb)=\sum_{\hatb'}\hatk(\hatb,\hatb')\hatb'$,
is a transition matrix, where $\hate_{1^3}=\phi(e_{1^{3}})^{-1}e_{1^3}$,
$\hate_{12}=\phi(e_{12})^{-1}e_{12}$, $\hate_{3}=\phi(e_{3})^{-1}e_{3}$. Note
that, taking $\phi(e_{1^{3}})=\left(\phi(e_1)\right)^{3}=1$ and $\phi(e_{12})=\phi(e_{2})\phi(e_{1})=\frac{1}{2}$,
the first two rows of $\hatk$ sum to 1. View this as $\hate_{1^3}=\hate_1 ^3$ and $\hate_{12}=\hate_2 \hate_1$. Then, as $\Psi^{2}(\hate_{3})=\phi(e_{3})^{-1}(e_{3}+e_{2}e_{1}+e_{1}e_{2}+e_{3})=\hate_{3}+\frac{1}{2}\phi(e_{3})^{-1}\hate_{2,1}+\frac{1}{2}\phi(e_{3})^{-1}\hate_{2,1}+\hate_{3}$, the transition matrix is given by \[
\hatk=\begin{bmatrix}
1 & 0 & 0\\
\frac{1}{2} & \frac{1}{2} & 0\\
0 & \frac{1}{8}\phi(e_{3})^{-1} & \frac{1}{4}\end{bmatrix}\]
and choosing $\phi(e_{3})=\frac{1}{6}$ makes the third row sum to 1.

Continuing, we find that $\phi(e_i)=\frac{1}{i!}$, so $\hate_i=i!e_i$, more generally, $\hate_\lambda=\prod(i!)^{a_i(\lambda)}e_\lambda$ with $i$ appearing $a_i(\lambda)$ times in $\lambda$. Then, for example,
\begin{equation*}
m\Delta\left(\hate_n\right)=n!m\Delta(e_n)=n!m\sum_{i=0}^ne_i\otimes e_{n-i}=\sum_{i=0}^n\binom{n}{i}\hate_i\hate_{n-i}.
\end{equation*}
So, for any partition $\lambda$ of $n$, 
\begin{eqnarray*}
m\Delta\left(\hate_{\lambda}\right) & = & m\Delta\left(\hate_{\lambda_{1}}\right)\cdots m\Delta\left(\hate_{\lambda_{n}}\right)\\
 & = & \sum_{\lambda'\leq\lambda}\binom{\lambda_{1}}{\lambda'_{1}}\binom{\lambda_{2}}{\lambda'_{2}}\cdots\binom{\lambda_{l}}{\lambda'_{l}}\hate_{\lambda'}\hate_{\lambda-\lambda'}
\end{eqnarray*}
and the coefficients of $m\Delta(\hate_\lambda)$ sum to $\sum_{\lambda'\leq\lambda}\binom{\lambda_{1}}{\lambda'_{1}}\cdots\binom{\lambda_{l}}{\lambda'_{l}}=2^{\lambda_1}\cdots 2^{\lambda_n}=2^n$, irrespective of $\lambda$. Thus $\frac1{2^n}m\Delta$ describes a transition matrix, which has the rock-breaking interpretation of \ref{sec1}.
\label{ex32}
\end{example}

The following theorem shows that this algorithm works in many cases. Observe that, in the above example, it is the non-zero off-diagonal entries that change; the diagonal entries cannot be changed by rescaling the basis. Hence the algorithm would fail if some row had all off-diagonal entries equal to 0, and diagonal entry not equal to 1. This corresponds to the existence of $b \in \calb_n$ with $\frac{1}{2^n}\Psi^2 (b)=\alpha b$ for some $\alpha \neq 1$; the condition $\bard (c):=\Delta(c) - 1 \otimes c - c\otimes 1 \neq 0$ below precisely prevents this. Intuitively, we are requiring that each generator of degree greater than one can be broken non-trivially. For an example where this condition fails, see \exref{quotsym}.
\begin{thm}[Basis rescaling]
Let $\calh$ be a graded Hopf algebra over $\real$ which is either a polynomial algebra or a free associative algebra that is cocommutative. Let $\calb$ denote the basis of monomials in the generators. Suppose that, for all generators $c$ with $\deg(c)>1$, all coefficients of $\Delta(c)$ (in the $\calb\otimes\calb$ basis) are non-negative and $\bard(c)\neq0$. Let $K_a$ be the transpose of the matrix of $a^{-n}\Psi^a$ with respect to the basis $\calb_n$; in other words, $a^{-n}\Psi^a(b)=\sum_{b'}K_a(b,b')b'$ (suppressing the dependence of $K_a$ on $n$). Define, by induction on degree,
\begin{align*}
\hatc & = c & \text{if $\deg(c)=1$} \\
\hatc & = \frac{1-2^{1-\deg(c)}}{\sum_{b \neq c} \phi(b) K_2 (c,b)} c & \text{for a generator $c$ with $\deg(c)>1$} \\
\hatb & = \hatc_1 \dots \hatc_l & \text{for $b \in \calb$ with factorization into generators $c_1 \dots c_l$}
\end{align*}
where $\phi(b)$ satisfies $b=\phi(b)\hatb$. Write $\hat{\calb}:=\{\hatb|b \in \calb\}$ and $\hat{\calb}_n:=\{\hatb|b \in \calb_n\}$. Then the matrix of the $a$th power map with respect to the $\hat{\calb}_n$ basis, when transposed and multiplied by $a^{-n}$, is a transition matrix. In other words, the operator $\hatk_a$ on $\calh_n$, defined by $a^{-n}\Psi^a(\hatb)=\sum_{b'}\hatk_a(\hatb,\hatb')\hatb'=\sum_{b'}\phi(b)^{-1}K_a(b,b')\phi(b')b'$, has $\hatk_a(\hatb,\hatb')\geq0$ and $\sum_{b'}\hatk_a(\hatb,\hatb')=1$ for all $b\in\calb_n$, and all $a\geq0$ and $n\geq0$ (the same scaling works simultaneously for all $a$). 
\label{thm31}
\end{thm}

\begin{rems}
\begin{enumerate}
\item Observe that, if $b=xy$, then the definition of $\hatb$ ensures $\hatb=\hatx \hat{y}$. Equivalently, $\phi$ is a multiplicative function.
\item The definition of $\hatc$ is not circular: since $\calh$ is graded with
$\calh_0=\real$, the counit is zero on elements of positive degree so that $\bard(c)\in \bigoplus_{j=1}^{\deg(c)-1} \calh_j\otimes
\calh_{\deg(c)-j}$. Hence $K_2(c,b)$ is non-zero only if $b=c$ or $l(b)>1$, so the denominator in the expression for $\hatc$ only involves $\phi(b)$ for $b$ with $l(b)>1$. Such $b$ can be factorized as $b=xy$ with $\deg(x), \deg(y)< \deg(b)$, whence $\phi(b) = \phi(x) \phi(y)$, so $\hatc$ only depends on $\phi(x)$ with $\deg(x)<\deg(c)$.
\end{enumerate}
\end{rems}

\begin{proof}
First note that $\hatk_2(c,c)=\phi(c)^{-1}K_2(c,c)\phi(c)=K_2(c,c)=2^{1-\deg(c)}$, since $m\Delta(c)=2c+m\bard(c)$ and $\bard(c) \in \bigoplus_{j=1}^{\deg(c)-1}\calh_j\otimes\calh_{\deg(c)-j}$ means no $c$ terms can occur in $m\bard(c)$.  So 
\begin{eqnarray*}
\sum_{b'}\hatk_{2}(\hatc,\hatb') & = & 2^{1-\deg(c)}+\phi(c)^{-1}\sum_{b'\neq c}K_{2}(c,b')\phi(b')\\
 & = & 2^{1-\deg(c)}+\frac{1-2^{1-\deg(c)}}{\sum_{b'\neq c}K_{2}(c,b')\phi(b)}\sum_{b'\neq c}K_{2}(c,b')\phi(b')\\
 & = & 1,\end{eqnarray*}
as desired.

Let $\eta_{c}^{xy}$ denote the coefficients of $\Delta(c)$ in the
$\calb\otimes\calb$ basis, so $\Delta(c)=\sum_{x,y\in\calb}\eta_{c}^{xy}x\otimes y$.
Then $K_{2}(c,b)=2^{-\deg(c)}\sum_{xy=b}\eta_{c}^{xy}$, and \begin{eqnarray*}
\hatk_{2}(\hatc,\hatb) & = & 2^{-\deg(c)}\sum_{xy=b}\phi(c)^{-1}\eta_{c}^{xy}\phi(b)\\
 & = & 2^{-\deg(c)}\sum_{xy=b}\phi(c)^{-1}\eta_{c}^{xy}\phi(x)\phi(y).\end{eqnarray*}
So, if $b$ has factorization into generators $b=c_{1}\dots c_{l}$,
then \begin{eqnarray*}
\Delta(b) & = & \Delta(c_{1})\dots\Delta(c_{l})\\
 & = & \sum_{\substack{x_{1},\dots,x_{l}\\ y_{1},\dots,y_{l}}} \eta_{c_{1}}^{x_{1}y_{1}}\dots\eta_{c_{l}}^{x_{l}y_{l}}x_{1}\dots x_{l}\otimes y_{1}\dots y_{l},\end{eqnarray*}
so \[
K_{2}(b,b')=2^{-\deg(b)}\sum_{x_{1}\dots x_{l}y_{1}\dots y_{l}=b'}\eta_{c_{1}}^{x_{1}y_{1}}\dots\eta_{c_{l}}^{x_{l}y_{l}}.\]
Thus\begin{eqnarray*}
\sum_{b'}\hatk_{2}(\hatb,\hatb') & = & 2^{-\deg(b)}\sum_{b'}\phi(b)^{-1}K_{2}(b,b')\phi(b')\\
 & = & 2^{-\deg(b)}\sum_{\substack{x_{1},\dots,x_{l}\\ y_{1},\dots,y_{l}}}\phi(b)^{-1}\eta_{c_{1}}^{x_{1}y_{1}}\dots\eta_{c_{l}}^{x_{l}y_{l}}\phi(x_{1}\dots x_{l}y_{1}\dots y_{l})\\
 & = & \prod_{i=1}^{l}2^{-\deg(c_{i})}\sum_{x_{i},y_{i}}\phi(c_{i})^{-1}\eta_{c_i}^{x_{i}y_{i}}\phi(x_{i})\phi(y_{i})\\
 & = & \prod_{i=1}^{l}\sum_{b_i}\hatk_{2}(\hatc_i,\hatb_i)\\
 & = & 1\end{eqnarray*}
as desired, where the third equality is due to multiplicativity of
$\phi$.

The above showed each row of $\hatk_2$ sums to 1, which means $(1,1,\dots ,1)$ is a right eigenvector of $\hatk_2$ of eigenvalue 1. The matrix $\hatk_a$ describes $\Psi^a$ in the $\hat{\calb}$ basis, which is also a basis of monomials/words, in a rescaled set of generators $\hatc$, so, by \trefs{fefnthm1}{fefnthm2}, the eigenspaces of $\hatk_a$ do not depend on $a$. Hence $(1,1,\dots ,1)$ is a right eigenvector of $\hatk_a$ of eigenvalue 1 for all $a$, thus each row of $\hatk_a$ sums to 1 also.

Finally, to see that the entries of $\hatk_{a}$ are non-negative,
first extend the notation $\eta_{c}^{xy}$ so $\Delta^{[a]}(c)=\sum_{b_{1},\dots b_{a}}\eta_{c}^{b_{1},\dots,b_{a}}b_{1}\otimes\dots\otimes b_{a}$.
As $\Delta^{[a]}=(\iota\otimes\dots\otimes\iota\otimes\Delta)\Delta^{[a-1]}$,
it follows that $\eta_{c}^{b_{1},\dots,b_{a}}=\sum_{x}\eta_{c}^{b_{1},\dots,b_{a-2},x}\eta_{x}^{b_{a-1},b_{a}}$,
which inductively shows that $\eta_{c}^{b_{1},\dots,b_{a}}\ge0$ for
all generators $c$ and all $b_{i}\in\calb$. So, if $b$ has factorization
into generators $b=c_{1}\dots c_{l}$, then \[
K_{a}(b,b')=\sum\eta_{c_{1}}^{b_{1,1},\dots,b_{1,a}}\dots\eta_{c_{l}}^{b_{l,1},\dots,b_{l,a}}\geq0,\]
where the sum is over all sets $\{b_{i,j}\}_{i=1,j=1}^{i=l,j=a}$
such that the product $b_{1,1}b_{2,1}\dots b_{l,1}b_{1,2}\dots b_{l,2}\dots b_{1,a}\dots b_{l,a}=b'$.
Finally, $\hatk_{a}(\hatb,\hatb')=\phi(b)^{-1}K_{a}(b,b')\phi(b')\geq0$.
\end{proof}

Combinatorial Hopf algebras often have a single basis element of degree 1 - for the algebra of symmetric functions, this is the unique partition of 1; for the Hopf algebra $\calg$  of graphs, this is the discrete graph with one vertex. After the latter example, denote this basis element by $\bullet$. Then there is a simpler definition of the eigenfunction $\phi$, and hence $\hatb$ and $\hatk$, in terms of $\eta_{b}^{b_{1},\dots,b_{r}}$, the coefficient of $b_1 \otimes \dots \otimes b_r$ in $\Delta^{[r]}(b)$:

\begin{cor}
Suppose that, in addition to the hypotheses of \tref{thm31}, $\calb_1 =\{\bullet\}$. Then $\hatb=\frac{(\deg b)!}{\eta_{b}^{\bullet,\dots,\bullet}}b$, so $\hatk_a$ is defined by 
\begin{equation*}
\hatk_a(\hatb,\hatb')=\frac{\eta_{b'}^{\bullet,\dots,\bullet}}{\eta_{b}^{\bullet,\dots,\bullet}}K_a(b,b')
\end{equation*}
\label{scaling}
\end{cor}

\begin{proof}
Work on $\calh_n$ for a fixed degree $n$. Recall that $\phi$ is a right eigenvector of $\hatk_a$ of eigenvalue 1, and hence, by the notation of \ref{sec34}, an eigenvector of $\Psi^{*a}$  of eigenvalue $a^n$. By \trefs{fefnthm1}{fefnthm2}, this eigenspace is spanned by $f_b$ for $b$ with length $n$. Then $\calb_1 =\{\bullet\}$ forces $b=\bullet^n$, so $f_{\bullet^n}(b')=\frac1{n!}\eta_{b'}^{\bullet,\dots,\bullet}$ spans the $a^n$-eigenspace of $\Psi^{*a}$. Consequently, $\phi$  is a multiple of $f_{\bullet^n}$. To determine this multiplicative factor, observe that \tref{thm31} defines $\phi(\bullet)$ to be 1, so $\phi(\bullet^n)=1$, and $f_{\bullet^n}(\bullet^n)=1$ also, so $\phi=f_{\bullet^n}$.
\end{proof}

\subsection{Acyclicity} \label{sec322}
Observe that the rock-breaking chain (\exrefs{ex12}{ex32}) is \textit{acyclic} - it can never return to a state it has left, because the only way to leave a state is to break the rocks into more pieces. More specifically, at each step the chain either stays at the same partition or moves to a partition which refines the current state; as refinement of partitions is a partial order, the chain cannot return to a state it has left. The same is true for the chain on unlabeled graphs (\exref{ex31}) - the number of connected components increases over time, and the chain never returns to a previous state. Such behavior can be explained by the way the length changes under the product and coproduct. (Recall that the length $l(b)$ is the number of factors in the unique factorization of $b$ into generators.) Define a relation on $\calb$ by $b\rightarrow b'$ if $b'$ appears in $\Psi^a(b)$ for some $a$. If $\Psi^a$ induces a Markov chain on $\calb_n$, then this precisely says that $b'$ is accessible from $b$.

\begin{lemma}
Let $b, b_i, b_{(i)}$ be monomials/words in a Hopf algebra which is either a polynomial algebra or a free associative algebra that is cocommutative. Then 
\begin{enumerate}[(i)]
\item $l\left( b_1 \dots b_a \right)=l \left( b_1 \right) + \dots + l\left(b_a \right)$;
\item For any summand $b_{(1)} \otimes \dots \otimes  b_{(a)}$ in $\Delta^{[a]}(b)$, $l \left( b_{(1)} \right) + \dots + l\left(b_{(a)} \right) \geq l(b)$;
\item if $b\rightarrow b'$, then $l(b') \geq l(b)$.
\end{enumerate}
\label{lengthlemma}
\end{lemma}

\begin{proof}
(i) is clear from the definition of length.

Prove (ii) by induction on $l(b)$. Note that the claim is vacuously true if $b$ is a generator, as each $l\left(b_{(i)}\right)\geq 0$, and not all $l\left(b_{(i)}\right)$ may be zero. If $b$ factorizes non-trivially as $b=xy$, then, as $\Delta^{[a]}(b)=\Delta^{[a]}(x)\Delta^{[a]}(y)$, it must be the case that $b_{(i)}=x_{(i)}y_{(i)}$, for some $x_{(1)} \otimes \dots \otimes  x_{(a)}$ in $\Delta^{[a]}(x)$, $y_{(1)} \otimes \dots \otimes  y_{(a)}$ in $\Delta^{[a]}(y)$. So $l \left( b_{(1)} \right) + \dots + l\left(b_{(a)} \right)=l \left( x_{(1)} \right) + \dots + l\left(x_{(a)} \right)+l \left( y_{(1)} \right) + \dots + l\left(y_{(a)} \right)$ by (i), and by inductive hypothesis, this is at least $l(x)+l(y)=l(b)$.

(iii) follows trivially from (i) and (ii): if $b\rightarrow b'$, then $b'= b_{(1)} \dots b_{(a)}$ for a term $b_{(1)} \otimes \dots \otimes  b_{(a)}$ in $\Delta^{[a]}(b)$. So $l(b')= l \left( b_{(1)} \right) + \dots + l\left(b_{(a)} \right) \geq l(b)$.
\end{proof}

If $\calh$ is a polynomial algebra, more is true. The following proposition explains why chains built from polynomial algebras (i.e., with deterministic and symmetric assembling) are always acyclic; in probability language, it says that, if the current state is built from $l$ generators, then, with probability $a^{l-n}$, the chain stays at this state, otherwise, it moves to a state built from more generators. Hence, if the states are totally ordered to refine the partial ordering by length, then the transition matrices are upper-triangular with $a^{l-n}$ on the main diagonal.

\begin{prop}(Acyclicity)
Let $\calh$ be a Hopf algebra which is a polynomial algebra as an algebra, and $\calb$ its monomial basis. Then the relation $\rightarrow$ defines a partial order on $\calb$, and the ordering by length refines this order: if $b\rightarrow b'$ and $b\neq b'$, then $l(b)<l(b')$. Furthermore, for any integer $a$ and any $b \in \calb$ with length $l(b)$,
\begin{equation*}
\Psi^a(b)=a^{l(b)}b +\sum_{l(b') > l(b)} \alpha_{bb'} b'
\end{equation*}
for some $\alpha_{bb'}$.
\label{acyclicity}
\end{prop}

\begin{proof}
It is easier to first prove the expression for $\Psi^a (b)$. Suppose $b$ has factorization into generators $b=c_1 c_2 \dots c_{l(b)}$. As $\calh$ is commutative, $\Psi^a$ is an algebra homomorphism, so $\Psi^a(b)=\Psi^a\left(c_1\right) \dots \Psi^a\left(c_{l(b)}\right)$. Recall from \ref{sec22} that $\bard(c)=\Delta(c)-1 \otimes c - c\otimes 1 \in \bigoplus_{i=1}^{deg(c)-1}\calh_i \otimes \calh_{deg(c)-i}$, in other words, $1 \otimes c$ and $c\otimes 1$ are the only terms in $\Delta(c)$ which have a tensor-factor of degree 0. As $\Delta^{[3]}=(\iota \otimes \Delta)\Delta$, the only terms in $\Delta^{[3]}(c)$ with two tensor-factors of degree 0 are $1\otimes 1 \otimes c$, $1 \otimes c \otimes 1$ and $c \otimes 1 \otimes 1$. Inductively, we see that the only terms in $\Delta^{[a]}(c)$ with all but one tensor-factor having degree 0 are $1 \otimes \dots \otimes 1 \otimes c, 1 \otimes \dots \otimes 1 \otimes c \otimes 1, \dots , c \otimes 1 \otimes \dots \otimes 1$. So $\Psi^a(c)=ac +\sum_{l(b') > 1} \alpha_{cb'} b'$ for generators $c$. As $\Psi^a(b)=\Psi^a\left(c_1\right) ... \Psi^a\left(c_l\right)$, and length is multiplicative (\lref{lengthlemma} (i)), the expression for $\Psi^a (b)$ follows.

It is then clear that $\rightarrow$ is reflexive and antisymmetric. Transitivity follows from the power rule: if $b\rightarrow b'$ and $b' \rightarrow b''$, then $b'$ appears in $\Psi^a (b)$ for some $a$ and $b''$ appears in $\Psi^{a'} (b')$ for some $a'$. So $b''$ appears in $\Psi^{a'} \Psi^{a} (b) = \Psi^{a'a} (b)$.
\end{proof}

The same argument applied to a cocommutative free associative algebra shows that all terms in $\Psi^a(b)$ are either a permutation of the factors of $b$, or have length greater than that of $b$. The relation $\rightarrow$ is only a preorder; the associated chains are not acyclic, as they may oscillate between such permutations of factors. For example, in the noncommutative Hopf algebra of labeled graphs, the following transition probabilities can occur:
\begin{center}
\includegraphics[scale=0.8, clip]{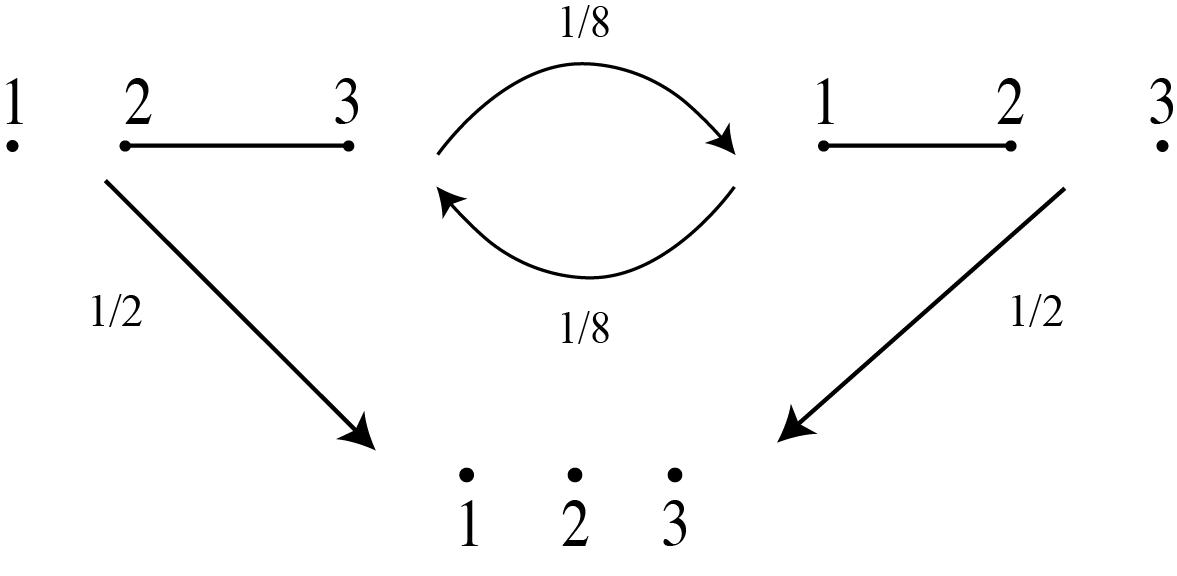}
\end{center}
\noindent
(the bottom state is absorbing). The probability of going from $b$ to some permutation of its factors (as opposed to a state of greater length, from which there is no return to $b$) is $a^{l(b)-n}$.

Here is one more result in this spirit, necessary in \ref{sec33} to show that the eigenvectors constructed there have good triangularity properties and hence form an eigenbasis:

\begin{lemma}
Let $b,b_i, b_i'$ be monomials/words in a Hopf algebra which is either a polynomial algebra or a free associative algebra that is cocommutative. If $b=b_1 \dots b_k$ and $b_i \rightarrow b'_i$ for each $i$, then $b \rightarrow b'_{\sigma (1)} \dots b'_{\sigma (k)}$ for any $\sigma \in S_k$.
\label{acyclicproduct}
\end{lemma}

\begin{proof}
For readability, take $k=2$ and write $b=xy,\ x\rightarrow x',\ y \rightarrow y'$. By definition of the relation $\rightarrow$, it must be that $x'=x_{(1)} \dots x_{(a)}$ for some summand $x_{(1)} \otimes \dots \otimes x_{(a)}$ of $\bard^{[a]}(x)$. Likewise $y'=y_{(1)} \dots y_{(a')}$ for some $a'$. Suppose $a>a'$. Coassociativity implies that $\Delta^{[a]}(y)=(\iota \otimes \dots \otimes \iota \otimes \Delta^{[a-a']})\Delta^{[a']}(y)$, and $y_{(a')}\otimes 1 \otimes \dots \otimes 1$ is certainly a summand of $\Delta^{[a-a']}(y_{(a')})$, so $y_{(1)} \otimes  \dots \otimes y_{(a')} \otimes 1 \otimes \dots \otimes 1$ occurs in $\Delta^{[a]}(y)$. So, taking $y_{(a'+1)}= \dots = y_{(a)}=1$, we can assume $a=a'$. Then $\Delta^{[a]}(b) = \Delta^{[a]}(x) \Delta^{[a]}(y)$ contains the term $x_{(1)} y_{(1)} \otimes \dots \otimes x_{(a)} y_{(a)}$. Hence $\Psi^a(b)$ contains the term $x_{(1)}y_{(1)} \dots x_{(a)} y_{(a)}$, and this product is $x'y'$ if $\calh$ is a polynomial algebra.

If $\calh$ is a cocommutative, free associative algebra, the factors in $x_{(1)} y_{(1)} \otimes \dots \otimes x_{(a)} y_{(a)}$ must be rearranged to conclude that $b \rightarrow x'y'$ and $b \rightarrow y'x'$. Coassociativity implies $\Delta^{[2a]}=(\Delta \otimes \dots \otimes \Delta) \Delta^{[a]}$, and $\Delta\left(x_{(i)}y_{(i)}\right)=\Delta\left(x_{(i)}\right)\Delta\left(y_{(i)}\right)$ contains $\left(x_{(i)} \otimes 1\right) \left(1 \otimes y_{(i)}\right)=x_{(i)} \otimes y_{(i)}$, so $\Delta^{[2a]}(b)$ contains the term $x_{(1)} \otimes y_{(1)} \otimes x_{(2)} \otimes y_{(2)} \otimes \dots \otimes x_{(a)} \otimes y_{(a)} $. As $\calh$ is cocommutative, any permutation of the tensor-factors, in particular, $x_{(1)} \otimes x_{(2)} \otimes \dots \otimes x_{(a)} \otimes y_{(1)} \otimes \dots \otimes y_{(a)}$ and $y_{(1)} \otimes y_{(2)} \otimes \dots \otimes y_{(a)} \otimes x_{(1)} \otimes \dots \otimes x_{(a)}$,  must also be summands of $\Delta^{[2a]}(b)$, and multiplying these tensor-factors together shows that both $x'y'$ and $y'x'$ appear in $\Psi^{[2a]}(b)$.
\end{proof}

\begin{example}[Symmetric functions and rock-breaking]
Recall from \exref{ex32} the algebra of symmetric functions with basis $\{e_\lambda \}$, which induces the rock-breaking process. Here, $e_\lambda \rightarrow e_{\lambda'}$ if and only if $\lambda'$ refines $\lambda$. \lref{acyclicproduct} for the case $k=2$ is the statement that, if $\lambda$ is the union of two partitions $\mu$ and $\nu$, and $\mu'$ refines $\mu$, $\nu'$ refines $\nu$, then $\mu' \amalg \nu'$ refines $\mu \amalg \nu = \lambda$.
\end{example}

\subsection{The symmetrization lemma} \label{sec333}

The algorithmic construction of left and right eigenbases for the chains created in \ref{sec32} will go as follows:
\begin{enumerate}[(i)]
\item Make an eigenvector of smallest eigenvalue for each generator $c$;
\item For each basis element $b$ with factorization $c_1 c_2 ... c_l$, build an eigenvector of larger eigenvalue out of the eigenvectors corresponding to the factors $c_i$, produced in the previous step.
\end{enumerate}

Concentrate on the left eigenvectors for the moment. Recall that the transition matrix $K_a$ is defined by $a^{-n}\Psi^a(b)=\sum_{b'} K_a(b,b')b'$, so the left eigenvectors for our Markov chain are the usual eigenvectors of $\Psi^a$ on $\calh$. Step (ii) is simple if $\calh$ is a polynomial algebra, because then $\calh$ is commutative so $\Psi^a$ is an algebra homomorphism. Consequently, the product of two eigenvectors is an eigenvector with the product eigenvalue. This fails for cocommutative, free associative algebras $\calh$, but can be fixed by taking symmetrized products:

\begin{thm}[Symmetrization lemma]
Let $x_1,x_2, \dots ,x_k$ be primitive elements of any Hopf algebra $\calh$, then $\sum_{\sigma \in S_k} x_{\sigma(1)} x_{\sigma(2)} \dots x_{\sigma(k)}$ is an eigenvector of $\Psi^a$ with eigenvalue $a^k$.
\label{symlemma}
\end{thm}

\begin{proof}
For concreteness, take $a=2$. Then
\begin{eqnarray*}
m\Delta\left(\sum_{\sigma\in S_{k}}x_{\sigma(1)}x_{\sigma(2)} \dots x_{\sigma(k)}\right) & = & m\left(\sum_{\sigma\in S_{k}}\left(\Delta x_{\sigma(1)}\right)\left(\Delta x_{\sigma(2)}\right) \dots \left(\Delta x_{\sigma(k)}\right)\right)\\
 & = & m\left(\sum_{\sigma\in S_{k}}\left(x_{\sigma(1)}\otimes1+1\otimes x_{\sigma(1)}\right) \dots \left(x_{\sigma(k)}\otimes1+1\otimes x_{\sigma(k)}\right)\right)\\
 & = & m\sum_{A_{1}\amalg A_{2}=\{1,2, \dots, k\}}\sum_{\sigma\in S_{k}}\prod_{i\in A_{1}}x_{\sigma(i)}\otimes\prod_{j\in A_{2}}x_{\sigma(j)}\\
 & = & \left|\left\{ \left(A_{1},A_{2}\right)|A_{1}\amalg A_{2}=\{1,2, \dots, k\}\right\} \right|\sum_{\sigma\in S_{k}}x_{\sigma(1)} \dots x_{\sigma(k)}\\
 & = & 2^{k}\sum_{\sigma\in S_{k}}x_{\sigma(1)} \dots x_{\sigma(k)}\end{eqnarray*}
\end{proof}

In \ref{sec33} and \ref{sec34}, the fact that the eigenvectors constructed give a basis will follow from triangularity arguments based on \ref{sec322}. These rely heavily on the explicit structure of a polynomial algebra or a free associative algebra. Hence it is natural to look for alternatives that will generalize this eigenbasis construction plan to Hopf algebras with more complicated structures. For example, one may ask whether some good choice of $x_i$ exists with which the symmetrization lemma will automatically generate a full eigenbasis. When $\calh$ is cocommutative, an elegant answer stems from the following two well-known structure theorems:

\begin{thm}[Cartier-Milnor-Moore]
 \cite{milnor,cartier} If $\calh$ is graded, cocommutative and connected, then $\calh$ is Hopf isomorphic to $\calu(\mathfrak{g})$, the universal enveloping algebra of a Lie algebra $\mathfrak{g}$, where $\mathfrak{g}$ is the Lie algebra of primitive elements of $\calh$.
\end{thm}

\begin{thm}[Poincar\'e--Birkoff--Witt]
\cite{humphreys,lothaire}. If $\{x_1,x_2,...\}$ is a basis for a Lie algebra $\mathfrak{g}$, then the symmetrized products $\sum_{\sigma \in S_k} x_{i_{\sigma(1)}} x_{i_{\sigma(2)}} ...x_{i_{\sigma(k)}}$, for $1\leq i_1\leq i_2\leq\dots\leq i_k$, form a basis for $\calu(\mathfrak{g})$. 
\end{thm}

Putting these together reduces the diagonalization of $\Psi^a$ on a cocommutative Hopf algebra to determining a basis of primitive elements:

\begin{thm} [Strong symmetrization lemma]
Let $\calh$ be a graded, cocommutative, connected Hopf algebra, and let $\{x_1,x_2,...\}$ be a basis for the subspace of primitive elements in $\calh$. Then, for each $k \in \mathbb{N}$, 
\begin{equation*}
\left\{\sum_{\sigma \in S_k} x_{i_{\sigma(1)}} x_{i_{\sigma(2)}} ...x_{i_{\sigma(k)}} | 1\leq i_1\leq i_2\leq\dots\leq i_k\right\}
\end{equation*}
is a basis of the $a^k$-eigenspace of $\Psi^a$.
\label{strongkeylemma}
\end{thm}

Much work \cite{fisher, aguiarsottile05, aguiarsottile06} has been done on computing a basis for the subspace of the primitives of particular Hopf algebras, their formulas are in general more efficient than our universal method here, and using these will be the subject of future work. Alternatively, the theory of good Lyndon words \cite{lalonderam} gives a Grobner basis argument to further reduce the problem to finding elements which generate the Lie algebra of primitives, and understanding the relations between them. This is the motivation behind our construction of the eigenvectors in \tref{gefnthm2}, although the proof is independent of this theorem, more analogous to that of \tref{gefnthm1}, the case of a polynomial algebra.

\subsection{Left eigenfunctions}\label{sec33}
This section gives an algorithmic construction of an eigenbasis for the Hopf power maps $\Psi^a$ on the Hopf algebras of interest. If $K_a$ as defined by $a^{-n}\Psi^a(b)=\sum_{b'} K_a(b,b')b'$ is a transition matrix, then this eigenbasis is precisely a left eigenbasis of the associated chain, though the results below stand whether or not such a chain may be defined (e.g., the construction works when some coefficients of $\Delta(c)$ are negative, and when there are primitive generators of degree $>1$). The first step is to associate each generator to an eigenvector of smallest eigenvalue, this is achieved using the \textit{(first) Eulerian idempotent} map
\begin{equation*}
e(x)=\sum_{a\geq1}\frac{(-1)^{a-1}}{a}m^{[a]}\bar{\Delta}^{[a]}(x)
\end{equation*} 
Here $\bard(x)=\Delta(x) - 1 \otimes x - x \otimes 1 \in\bigoplus_{j=1}^{n-1}\calh_j\otimes\calh_{n-j}$, as explained in \ref{sec22}. Then inductively define $\bard^{[a]}=(\iota \otimes \cdots \otimes \iota \otimes \bard)\bard^{[a-1]}$, which picks out the terms in $\Delta^{[a]} (x)$ where each tensor-factor has strictly positive degree. This captures the notion of breaking into $a$ non-trivial pieces. Observe that, if $x\in\calh_n$, then $\bar{\Delta}^{[a]}(x)=0$
whenever $a>n$, so $e(x)$ is a finite sum for all $x$. (By convention, $e \equiv 0$ on $\calh_0$.)

This map $e$ is the first of a series of Eulerian idempotents $e_i$ defined by Patras \cite{patras93}; he proves that, in a commutative or cocommutative Hopf algebra of characteristic zero where $\bard$ is \textit{locally nilpotent} (i.e. for each $x$, there is some $a$ with $\bard^{[a]} x=0$), the Hopf-powers are diagonalizable, and these $e_i$ are orthogonal projections onto the eigenspaces. In particular, this \textit{weight decomposition} holds for graded commutative or cocommutative Hopf algebras. We will not need the full series of Eulerian idempotents, although \exref{highereulerianidempotent} makes the connection between them and our eigenbasis.

To deduce that the eigenvectors we construct are triangular with respect to $\calb$, one needs the following crucial observation (recall from \ref{sec322} that $b \rightarrow b'$ if $b'$ occurs in $\Psi^a(b)$ for some $a$):
\begin{prop}
For any generator $c$, 
\begin{equation*}
e(c)=c +\sum_{\substack{c \rightarrow b' \\ b' \neq c}} \alpha_{cb'} b' = c +\sum_{l(b') >1} \alpha_{cb'} b',
\end{equation*}
for some real $\alpha_{cb'}$.
\label{ec}
\end{prop}

\begin{proof}
The summand $\frac{(-1)^{a-1}}{a}m^{[a]}\bar{\Delta}^{[a]}(c)$ involves terms of length at least $a$, from which the second expression of $e(c)$ is immediate. Each term $b'$ of $e(c)$ appears in $\Psi^a(c)$ for some $a$, hence $c\rightarrow b'$. Combine this with the knowledge from the second expression that $c$ occurs with coefficient 1 to deduce the first expression.
\end{proof}

The two theorems below detail the construction of an eigenbasis for $\Psi^a$ in a polynomial algebra and in a cocommutative free associative algebra respectively. These are left eigenvectors for the corresponding transition matrices. A worked example will follow immediately; it may help to read these together.

\begin{thm}
Let $\mathcal{H}$ be a Hopf algebra (over a field of characteristic zero)
that is a polynomial algebra as an algebra, with monomial basis $\mathcal{B}$. For $b\in\calb$ with factorization into generators $b=c_1 c_2...c_l$,
set 
\begin{equation*}
g_{b}:=e\left(c_1\right) e\left(c_2\right) ...e\left(c_l\right).
\end{equation*}
Then $g_{b}$
is an eigenvector of $\Psi^{a}$ of eigenvalue $a^{l}$ satisfying the triangularity condition
\begin{equation*}
g_b=b +\sum_{\substack{b \rightarrow b' \\ b' \neq b}} g_b (b') b' = b +\sum_{l(b') > l(b)} g_b (b') b'.
\end{equation*}
Hence $\left\{ g_{b}|b\in\mathcal{B}_{n}\right\} $ is an eigenbasis for
the action of $\Psi^{a}$ on $\mathcal{H}_{n}$, and the multiplicity
of the eigenvalue $a^{l}$ in $\mathcal{H}_{n}$ is the coefficient
of $x^{n}y^{l}$ in $\prod_{i}\left(1-yx^{i}\right)^{-d_{i}}$, where
$d_{i}$ is the number of generators of degree $i$.
\label{gefnthm1}
\end{thm}

\begin{thm}
Let $\mathcal{H}$ be a cocommutative Hopf algebra (over a field of characteristic
zero) that is a free associative algebra with word basis $\mathcal{B}$.
For $b\in\mathcal{B}$ with factorization into generators $b=c_1 c_2...c_l$,
set $g_{b}$ to be the polynomial $\sym (b)$ evaluated at $\left(e\left(c_1\right),e\left(c_2\right), \dots , e\left(c_l\right)\right)$. In other words, in the terminology of \ref{sec23},
\begin{itemize}
\item for $c$ a generator, set $g_c:=e(c)$;
\item for $b$ a Lyndon word, inductively define $g_b: = \left[g_{b_1}, g_{b_2}\right]$ where $b=b_1 b_2$ is the standard factorization of $b$;
\item for $b$ with Lyndon factorization $b=b_1 \dots b_k$, set $g_b:=\sum_{\sigma\in S_k} g_{b_{\sigma(1)}} g_{b_{\sigma(2)}}\dots g_{b_{\sigma(k)}}$.
\end{itemize}
Then $g_{b}$
is an eigenvector of $\Psi^{a}$ of eigenvalue $a^{k}$ ($k$ the number of Lyndon factors in $b$) satisfying the triangularity condition
\begin{equation*}
g_b=\sum_{b \rightarrow b'} g_b (b') b' = \sym (b) +\sum_{l(b') > l(b)} g_b (b') b'.
\end{equation*}
Hence $\left\{ g_{b}|b\in\mathcal{B}_{n}\right\} $ is an eigenbasis for
the action of $\Psi^{a}$ on $\mathcal{H}_{n}$, and the multiplicity
of the eigenvalue $a^{k}$ in $\mathcal{H}_{n}$ is the coefficient
of $x^{n}y^{k}$ in $\prod_{i}\left(1-yx^{i}\right)^{-d_{i}}$, where
$d_{i}$ is the number of Lyndon words of degree $i$ in the alphabet of generators.
\label{gefnthm2}
\end{thm}

\begin{rems}
\begin{enumerate}
\item If $\Psi^a$ defines a Markov chain, then the triangularity of $g_b$ (in both theorems) has the following interpretation: the left eigenfunction $g_b$ takes non-zero values only on states that are reachable from $b$.
\item The expression of the multiplicity of the eigenvalues (in both theorems) holds for Hopf algebras that are multigraded, if we replace all $x$s, $n$s and $i$s by tuples, and read the formula as multi-index notation. For example, for a bigraded polynomial algebra $\calh$, the multiplicity of the $a^l$-eigenspace in $\calh_{m,n}$ is the coefficient of $x_1^m x_2^n y^l$ in $\prod_{i,j}\left(1-yx_1^i x_2^j\right)^{-d_{i,j}}$, where
$d_{i,j}$ is the number of generators of bidegree $(i,j)$. This idea will be useful in \ref{sec5}.
\item \tref{gefnthm2} essentially states that any cocommutative free associative algebra is in fact isomorphic to \textit{the} free associative algebra, generated by $e(c)$. But there is no analogous interpretation for \tref{gefnthm1}; being a polynomial algebra is not a strong enough condition to force all Hopf algebras with this condition to be isomorphic. A polynomial algebra $\calh$ is isomorphic to the usual polynomial Hopf algebra (i.e. with primitive generators) only if $\calh$ is cocommutative; then $e(c)$ gives a set of primitive generators.
\end{enumerate}
\end{rems}

\begin{example}
As promised, here is a worked example of this calculation, in the
noncommutative Hopf algebra of labeled graphs, as defined in \exref{exncgraphs2}. Let $b$ be the graph
\begin{center}
\includegraphics[scale=0.5, clip]{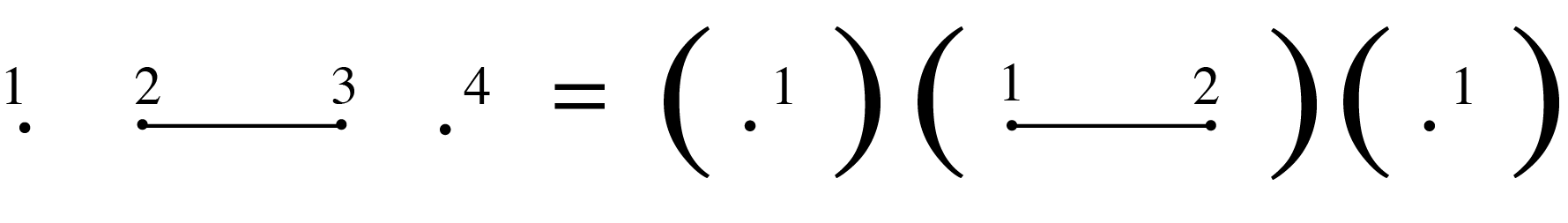}
\end{center}
\noindent
which is the product of three generators as shown. (Its factors happen to be its connected components, but that's not always true). Since the ordering of generators refines the ordering by degree, a vertex (degree 1) comes before an edge (degree 2), so the Lyndon factorization of $b$ is
\begin{center}
\includegraphics[scale=0.5, clip]{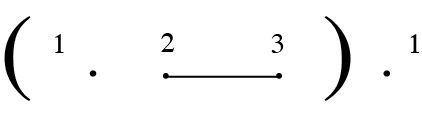}
\end{center}
\noindent
So $g_b$ is defined to be
\begin{center}
\includegraphics[scale=0.5, clip]{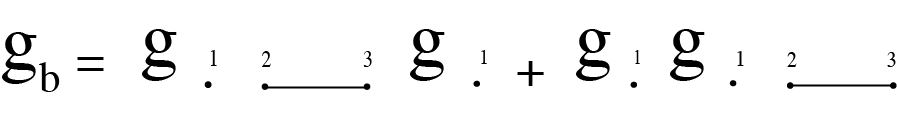}
\end{center}
\noindent
The first Lyndon factor of $b$ has standard factorization 
\begin{center}
\includegraphics[scale=0.5, clip]{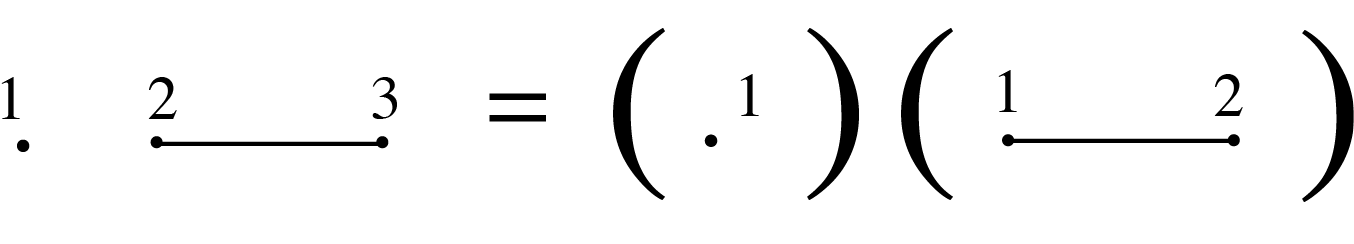}
\end{center}
\noindent
so
\begin{center}
\includegraphics[scale=0.5, clip]{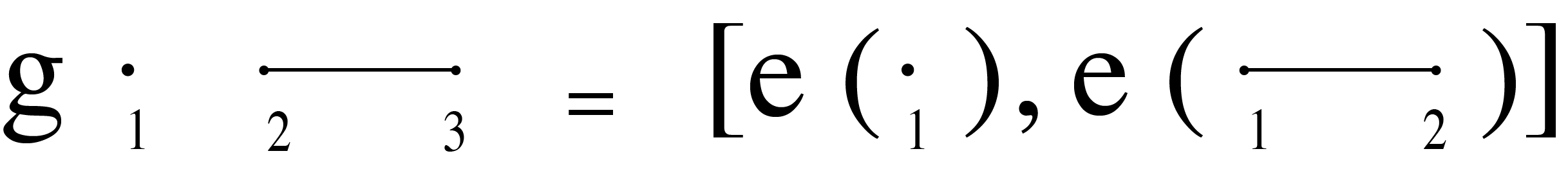}
\end{center}
\noindent
The Eulerian idempotent map fixes the single vertex, and 
\begin{center}
\includegraphics[scale=0.5, clip]{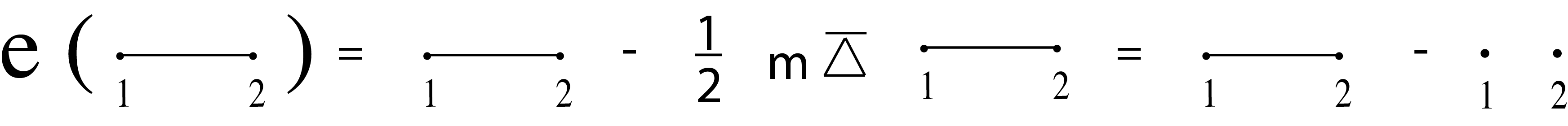}
\end{center}
\noindent
thus substituting into the previous equation gives
\begin{center}
\includegraphics[scale=0.5, clip]{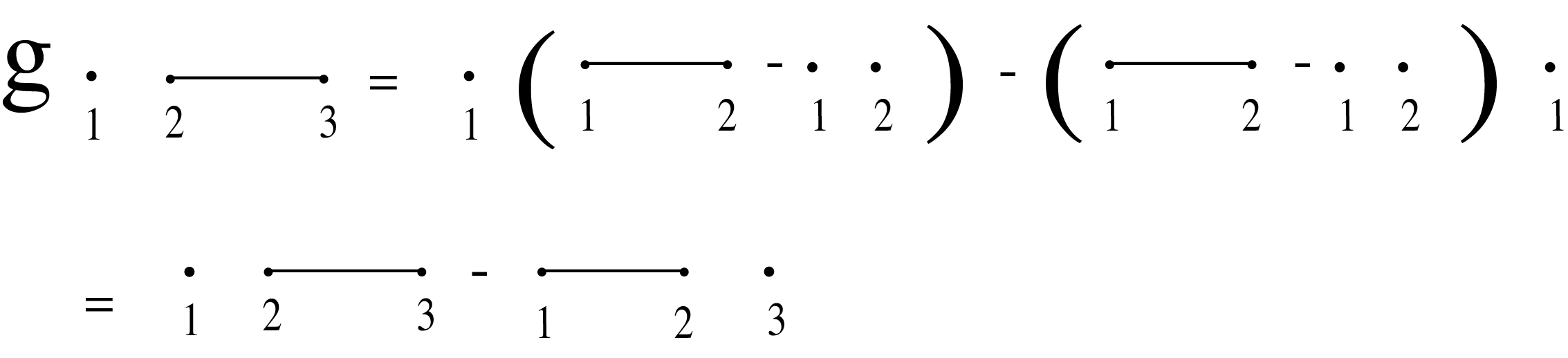}
\end{center}
\noindent
Since 
\begin{center}
\includegraphics[scale=0.5, clip]{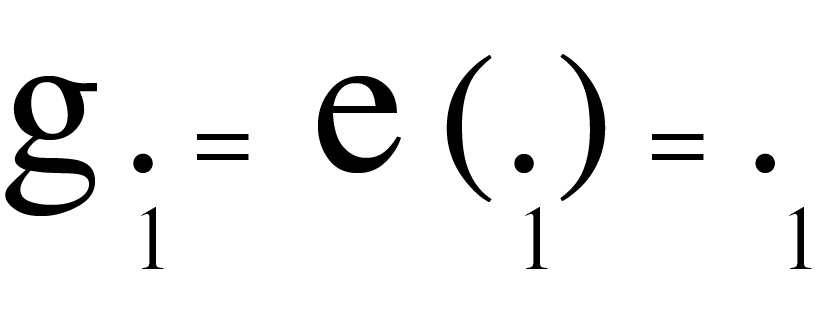}
\end{center}
\noindent
returning to the first expression for $g_b$ gives the following eigenvector of eigenvalue $a^2$:
\begin{center}
\includegraphics[scale=0.5, clip]{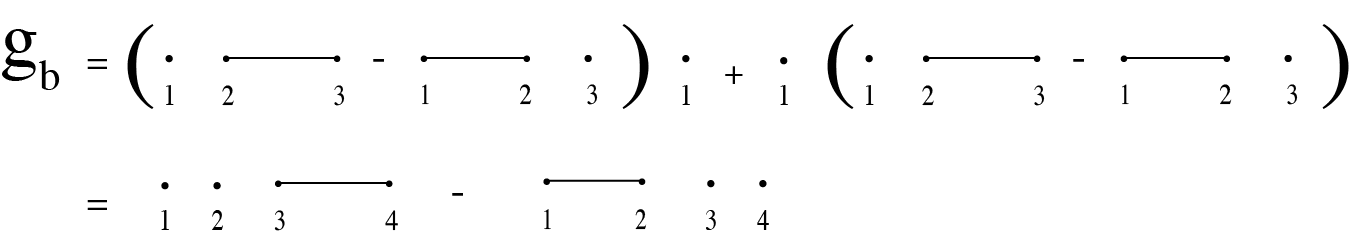}
\end{center}
\noindent
\label{gefnex}
\end{example}

\begin{proof}[Proof of \tref{gefnthm1} (polynomial algebra)]
By Patras \cite{patras93}, the Eulerian idempotent map is a projection onto the $a$-eigenspace of $\Psi^a$, so, for each generator $c$, $ e(c)$ is an eigenvector of eigenvalue $a$. As $\mathcal{H}$ is commutative, $\Psi^{a}$ is an algebra homomorphism,
so the product of two eigenvectors is another eigenvector with the
product eigenvalue. Hence $g_{b}:=e(c_1)e(c_2)...e(c_l)$ is an
eigenvector of eigenvalue $a^{l}$.

To see triangularity, note that, by \pref{ec}, 
\begin{eqnarray*}
g_b & = & \left(c_1 +\sum_{\substack{c_1 \rightarrow c'_1 \\ c_1' \neq c_1}} g_{c_1} (c_1') c_1'\right)\dots \left(c_l +\sum_{\substack{c_l \rightarrow c'_l \\ c_l' \neq c_l}} g_{c_l} (c_l') c_l'\right) \\
& = & b + \sum_{\substack{c_i \rightarrow c'_i \\ c_i' \neq c_i\text{ for some }i}} g_{c_1} (c_1') \dots g_{c_l} (c_l') c_1' \dots c_l'.
\end{eqnarray*}
\lref{acyclicproduct} shows that $b \rightarrow c_1 ' \dots c_l'$ in each summand, and the condition $c_i' \neq c_i$ for some $i$ means precisely that $ c_1' \dots c_l' \neq b$. Also, by \pref{ec}, 
\begin{eqnarray*}
g_b & = & \left(c_1 +\sum_{l(c_1')>1} g_{c_1} (c_1') c_1'\right)\dots \left(c_l +\sum_{l(c_l')>1}g_{c_l} (c_l') c_l'\right) \\
& = & b + \sum_{l(c_i')>1 \text{ for some }i} g_{c_1} (c_1') \dots g_{c_l} (c_l') c_1' \dots c_l'
\end{eqnarray*}
and thus $l(c_1' \dots c_l')>l$ as length is multiplicative.

The multiplicity of the eigenvalue $a^l$ is the number of basis elements $b$ with length $l$. The last assertion of the theorem is then immediate from \cite[Th.~3.14.1]{wilf}. 
\end{proof}

\begin{example}
We  show
that $g_{b}=e_{l(b)}(b)$, where the higher Eulerian idempotents are
defined by 
\begin{equation*}
e_{i}=\frac 1{i!} m^{[i]}(e \otimes e \otimes \cdots \otimes e) \Delta^{[i]}.
\end{equation*}
By Patras \cite{patras93}, $e_i$ is a projection to the $a^i$-eigenspace of $\Psi^a$, so, given the triangularity condition of the eigenbasis $\{g_b\}$, it suffices to show that $b$ is the only term of length $l(b)$ in $e_{l(b)}(b)$. Note that $e_{l(b)}(b)$ is a sum of terms of the form $e(b_{(1)})e(b_{(2)})...e(b_{(l)})$ for some $b_{(i)}$ with 
$b_{(1)}\otimes\cdots\otimes b_{(l)}$ a summand of $\Delta^{[l]}(b)$. As
$e \equiv 0$ on $\mathcal{H}_{0}$, the $b_{i}$s
must be non-trivial. Hence each term $b'$ of $e_{l(b)}(b)$ has the form $b'=b_{(1)}'\dots b_{(l)}'$, with $b_{(i)}\rightarrow b_{(i)}'$ and $b\rightarrow b_{(1)} \dots b_{(l)}$. It follows from \lref{acyclicproduct} that $b_{(1)} \dots b_{(l)} \rightarrow b_{(1)}' \dots b_{(l)}'$, so $b \rightarrow b'$ by transitivity, which, by \lref{acyclicity} means $l(b')>l(b)$ unless $b'=b$. 

It remains to show that the coefficient of $b$ in $e_{l(b)}(b)$ is 1. Let $b=c_1 \dots c_l$ be the factorization of $b$ into generators. With notation from the previous paragraph, taking $b'=b$ results in $b \rightarrow b_{(1)} \dots b_{(l)} \rightarrow b_{(1)}' \dots b_{(l)}' =b$, so $b=b_{(1)} \dots b_{(l)}$. This forces the $b_{(i)}=c_{\sigma(i)}$ for some $\sigma \in S_l$. As $b_{(i)}$ occurs with coefficient 1 in $e(b_{(i)})$, the coefficient of $b_{(1)} \otimes \dots \otimes b_{(l)}$ in $(e \otimes  \cdots \otimes e) \Delta^{[l]}(b)$ is the coefficient of $c_{\sigma(1)} \otimes \dots \otimes c_{\sigma(l)}$ in $\Delta^{[l]}(b)=\Delta^{[l]}(c_1) \dots \Delta^{[l]}(c_l)$, which is 1 for each $\sigma \in S_l$. Each occurrence of $c_{\sigma(1)} \otimes \dots \otimes c_{\sigma(l)}$ in $(e \otimes  \cdots \otimes e) \Delta^{[l]}(b)$ gives rise to a $b$ term in $m^{[l]}(e \otimes \cdots \otimes e) \Delta^{[l]}(b)$ with the same coefficient, for each $\sigma \in S_l$, hence $b$ has coefficient $l!$ in $m^{[l]}(e \otimes \cdots \otimes e) \Delta^{[l]}(b)=l!e_l(b)$. 

The same argument also shows that, if $i<l(b)$, then $e_{i}(b)=0$,
as there is no term of length $i$ in $e_{i}(b)$. In particular, $e(b)=0$ if $b$ is not a generator.
\label{highereulerianidempotent}
\end{example}

\begin{proof}[Proof of \tref{gefnthm2} (cocommutative and free associative algebra)]
Schmitt \cite[~Thm 9.4]{schmitt94} shows that the Eulerian idempotent map $e$ projects a graded cocommutative algebra onto its subspace of primitive elements, so $g_c:=e(c)$ is primitive. A straightforward calculation shows that, if $x,y\in\calh$ are primitive, then so is $[x,y]$. Iterating this implies that, if $b$ is a Lyndon word, then $g_b$ (which is the standard bracketing of $e(c)$s) is primitive. Now apply the symmetrization lemma (\lref{symlemma}) to deduce that, if $b\in \calb$ has $k$ Lyndon factors, $g_b$ is an eigenvector of eigenvalue $a^k$.

To see triangularity, first recall that $\sym$ is a linear combination of the permutations of its arguments, hence $g_b$ is a linear combination of products of the form $e(c_{\sigma(1)}) \dots e(c_{\sigma(l)})$ for some $\sigma \in S_l$. Hence, by \pref{ec}, each term in $g_b$ has the form $c'_{\sigma(1)} \dots c'_{\sigma(l)}$ with $c_i \rightarrow c_i'$, and by \lref{acyclicproduct}, we have $b \rightarrow c'_{\sigma(1)} \dots c'_{\sigma(l)}$. Also, by \pref{ec},
\begin{eqnarray*}
g_b & = & \sym \left( \left(c_1 +\sum_{l(c_1')>1} g_{c_1} (c_1') c_1'\right), \dots ,\left(c_l +\sum_{l(c_l')>1}g_{c_l} (c_l') c_l'\right) \right) \\
& = & \sym(b) + \sum_{l(c_i')>1 \text{ for some }i} \sym \left( g_{c_1}(c_1')  c_1' , \dots, g_{c_l} (c_l') c_l'\right),
\end{eqnarray*}
and all terms of the sum have length greater than $l$, as length is multiplicative, and $\sym$ is a linear combination of the permutations of its arguments.

The multiplicity of the eigenvalue $a^k$ is the number of basis elements with $k$ Lyndon factors. The last assertion of the theorem is then immediate from \cite[Th.~3.14.1]{wilf}.
\end{proof}

\subsection{Right eigenvectors}\label{sec34}

To obtain the right eigenvectors for our Markov chains, consider the graded
dual $\calh^*$ of the algebras examined above. The multiplication $\Delta^*$ and comultiplication $m^*$ on $\calh^*$ are given by:
\begin{alignat*}{4}
\left(x^{*}y^{*}\right)(z) & = & \left(\Delta^{*}\left(x^{*}\otimes y^{*}\right)\right)(z) & = & \left(x^{*}\otimes y^{*}\right)(\Delta z)\\
 &  & \left(m^{*}x^{*}\right)(z\otimes w) & = & x^{*}m(z\otimes w) & = & x^{*}(zw)
\end{alignat*}
for any $x^*,y^*\in \calh^*$, $z,w\in \calh$. Then $\Psi^{*a}:=\Delta^{*[a]}m^{*[a]}$  is the dual map to $\Psi^a$. So, if $K_a$, defined by $a^{-n}\Psi^a(b)=\sum_{b'} K_a(b,b')b'$, is a transition matrix, then its right eigenvectors are the eigenvectors of $\Psi ^{*a}$. The theorems below express these eigenvectors in terms of $\left\{b^{*}\right\}$, the dual basis to $\calb$. Dualizing a commutative Hopf algebra creates a cocommutative Hopf algebra, and vice versa, so \tref{fefnthm1} below, which diagonalizes $\Psi^{*a}$ on a polynomial algebra, will share features with \tref{gefnthm2}, which diagonalizes $\Psi^a$ on a cocommutative free associative algebra. Similarly, \tref{fefnthm2} and \tref{gefnthm1} will involve common ideas. However, \trefs{fefnthm1}{fefnthm2} are not direct applications of \trefs{gefnthm2}{gefnthm1} to $\mathcal{H}^{*}$ as $\calh^*$ is not a polynomial or free associative algebra - a breaking and recombining chain with a deterministic recombination does not dualize to one with a deterministic recombination. For example, the recombination step is deterministic for inverse shuffling (place the left pile on top of the right pile), but not for forward riffle shuffling (shuffle the two piles together). 

The two theorems below give the eigenvectors of $\Psi^{*a}$; exemplar computations are in \ref{sec42}. \tref{fefnthm1} gives a complete description of these for $\calh$ a polynomial algebra, and \tref{fefnthm2} yields a partial description for $\calh$ a cocommutative free associative algebra. Recall that $\eta_{b}^{b_{1},\dots,b_{a}}$ is the coefficient of $b_1 \otimes \dots \otimes b_a$ in $\Delta^{[a]}(b)$.

\begin{thm}
Let $\mathcal{H}$ be a Hopf algebra (over a field of characteristic zero)
that is a polynomial algebra as an algebra, with monomial basis $\mathcal{B}$. For $b\in\calb$ with factorization into generators $b=c_1 c_2...c_l$,
set 
\begin{equation*}
f_{b}:=\frac{1}{A(b)l!}\sum_{\sigma\in S_{l}}c_{\sigma(1)}^{*}c_{\sigma(2)}^{*} \dots c_{\sigma(l)}^{*},
\end{equation*}
where the normalizing constant $A(b)$ is calculated as follows: for each generator $c$, let $a_{c}(b)$ be the power of $c$ in the factorization of $b$, and set $A(b)=\prod_{c}a_{c}(b)!$. Then $f_{b}$
is an eigenvector of $\Psi^{*a}$ of eigenvalue $a^{l}$, and 
\begin{eqnarray*}
f_b(b') & = & \frac{1}{A(b)l!} \sum_{\sigma\in S_{l}} \eta_{b'}^{c_{\sigma(1)}, c_{\sigma(2)}, \dots, c_{\sigma(l)}} \\
& = & \frac{1}{l!} \sum_{\sigma} \eta_{b'}^{c_{\sigma(1)}, c_{\sigma(2)}, \dots, c_{\sigma(l)}},
\end{eqnarray*}
where the sum on the second line runs over all $\sigma$ with $\left(c_{\sigma(1)}, \dots, c_{\sigma(l)}\right)$ distinct (i.e., sum over all coset representatives of the stabilizer of $\left(c_1, \dots, c_l \right)$). The eigenvector $f_b$ satisfies the triangularity condition
\begin{equation*}
f_b=b^{*} +\sum_{\substack{b' \rightarrow b \\ b' \neq b}} f_b (b') b'^{*} = b^{*} +\sum_{l(b') < l(b)} f_b (b') b'^{*}.
\end{equation*}
Furthermore, $\{f_{b}\}$ is the dual basis to $\{g_{b}\}$. In other words, $f_{b}\left(g_{b'}\right)=0$
if $b\neq b'$, and $f_b (g_b)=1$.
Yet another formulation: the change of basis matrix from $\{f_{b}\}$ to $\calb$, which has $f_{b}$ as its columns, is
the inverse of the matrix with $g_{b}$ as its rows.  
\label{fefnthm1}
\end{thm}

\begin{rems}
\begin{enumerate}
\item If $\mathcal{H}$ is also cocommutative, then it is
unnecessary to symmetrize - just define $f_{b}=\frac{1}{A(b)}c_1^{*}c_2^{*}...c_l^{*}$.
\item If $\Psi^a$ defines a Markov chain on $\calb_n$, then the theorem says $f_b (b')$ may be interpreted as the number of ways to break $b'$ into $l$ pieces so that the result is some permutation of the $l$ generators that are factors of $b$. In particular, $f_b$ takes only non-negative values, and $f_b$ is non-zero only on states which can reach $b$. Thus $f_b$ may be used to estimate the probability of being in states that can reach $b$, see \coref{cor44} for an example.
\end{enumerate}
\end{rems}

\begin{thm}
Let $\mathcal{H}$ be a cocommutative Hopf algebra (over a field of characteristic zero) which is a free associative algebra as an algebra, with word basis $\calb$. For each Lyndon word $b$, let $f_b$ be the eigenvector of $\Psi^{*a}$ of eigenvalue $a$ such that $f_b (g_b)=1$ and $f_b(g_{b'})=0$ for all other Lyndon $b'$. In particular, $f_c=c^*$ and is primitive. For each basis element $b$ with Lyndon factorization $b=b_{1} \dots b_{k}$, let
\begin{equation*}
f_{b}:=\frac{1}{A'(b)}f_{b_{1}} \dots f_{b_{k}},
\end{equation*}
where the normalizing constant $A'(b)$ is calculated as follows: for each Lyndon basis element $b'$, let $a'_{b'}(b)$ be the number of times $b'$ occurs in the Lyndon factorization of $b$, and set  $A'(b)=\prod_{b'}a'_{b'}(b)!$. Then $f_b$ is an eigenvector of $\Psi^{*a}$ of eigenvalue $a^{k}$, and $\{f_{b}\}$ is the dual basis to $\{g_{b}\}$. If $b=c_1c_2 \dots c_l$ with $c_1 \geq c_2 \geq \dots \geq c_l$ in the ordering of generators, then
\begin{equation*}
f_b(b') = \frac{1}{A'(b)} \eta_{b'}^{c_1, \dots, c_l}.
\end{equation*}
\label{fefnthm2}
\end{thm}

\begin{proof}[Proof of \tref{fefnthm1} (polynomial algebra)]
Suppose $b^* \otimes b'^*$ is a term in $m^*(c^*)$, where $c$ is a generator. This means $m^*(c^*)(b \otimes b')$ is non-zero. Since comultiplication
in $\mathcal{H}^{*}$ is dual to multiplication in $\mathcal{H}$, $m^*(c^*)(b \otimes b')=c^*(bb')$, which is only nonzero if $bb'$ is a (real number) multiple of $c$. Since $c$ is a generator, this can only happen if one of $b,b'$ is $c$. Hence $c^*$ is primitive. Apply the symmetrization lemma (\lref{symlemma}) to the primitives $c_1^*, \dots , c_l^*$ to deduce that $f_b$ as defined above is an eigenvector of eigenvalue $a^l$.

Since multiplication
in $\mathcal{H}^{*}$ is dual to the coproduct in $\mathcal{H}$, $c_{\sigma(1)}^* c_{\sigma(2)}^*\dots c_{\sigma(l)}^*(b')=c_{\sigma(1)}^* \otimes c_{\sigma(2)}^* \otimes \dots \otimes c_{\sigma(l)}^* \left(\Delta^{[l]}(b')\right)$, from which the first expression for $f_b (b')$ is immediate. To deduce the second expression, note that the size of the stabilizer of $\left(c_{\sigma(1)}, c_{\sigma(2)}, \dots, c_{\sigma(l)}\right)$ under $S_l$ is precisely $A(b)$.

It is apparent from the formula that $b'^*$ appears in $f_b$ only if $c_{\sigma(1)} \dots c_{\sigma(l)}=b$ appears in $\Psi^l(b')$, hence $b' \rightarrow b$ is necessary. To calculate the leading coefficient $f_b(b)$, note that this is the sum over $S_l$ of the coefficients of $c_{\sigma(1)} \otimes \dots \otimes c_{\sigma(l)}$ in $\Delta^{[l]}(b)=\Delta^{[l]}(c_1) \dots \Delta^{[l]}(c_l)$. Each term in $\Delta^{[l]}(c_i)$ contributes at least one generator to at least one tensor-factor, and each tensor-factor of $c_{\sigma(1)} \otimes \dots \otimes c_{\sigma(l)}$ is a single generator, so each occurrence of $c_{\sigma(1)} \otimes \dots \otimes c_{\sigma(l)}$ is a product of terms from $\Delta^{[l]}(c_i)$ where one tensor-factor is $c_i$ and all other tensor-factors are 1. Such products are all $l!$ permutations of the $c_i$ in the tensor-factors, so, for each fixed $\sigma$, the coefficient of $c_{\sigma(1)} \otimes \dots \otimes c_{\sigma(l)}$ in $\Delta^{[l]}(b)$ is $A(b)$. This proves the first equality in the triangularity statement. Triangularity of $f_b$ with respect to length follows, as ordering by length refines the relation $\rightarrow$ (\pref{acyclicity}).

To see duality, first note that, since $\Psi^{*a}$ is the linear algebra dual to $\Psi^a$, $f_{b}\left(\Psi^{a}g_{b'}\right)=\left(\Psi^{*a}f_{b}\right)\left(g_{b'}\right)$.
Now, using that $f_{b}$ and $g_{b}$ are eigenvectors, it follows that
\begin{equation*}
a^{l(b')}f_{b}\left(g_{b'}\right)=f_{b}\left(\Psi^{a}g_{b'}\right)=\Psi^{*a}f_{b}\left(g_{b'}\right)=a^{l(b)}f_{b}\left(g_{b'}\right),
\end{equation*}
so $f_{b}\left(g_{b'}\right)=0$ if $l(b')\neq l(b)$. 

Now suppose $l(b')=l(b)=l$. Then 
\begin{equation*}
f_b(g_{b'})=  \left(b^* +\sum_{l(b_1) < l} f_b (b_1) b^{*}_1 \right) \left(b +\sum_{l(b_2) > l} g_{b'} (b_2) b_2 \right)=b^*(b'),
\end{equation*}
which is 0 when $b\neq b'$,
and $1$ when $b=b'$.
\end{proof}

\begin{proof}[Proof of \tref{fefnthm2} (cocommutative and free associative algebra)]
$\mathcal{H}^{*}$ is commutative, so the power map is an algebra
homomorphism. Then, since $f_b$ is defined as the product of $k$ eigenvectors each of eigenvalue $a$, $f_b$ is an eigenvector
of eigenvalue $a^{k}$. 

For any generator $c$, $c^*$ is primitive by the same reasoning as in \tref{fefnthm1}, the case of a polynomial algebra. To check that $c^*(g_c)=1$ and $c^*(g_b)=0$ for all other Lyndon $b$, use the triangularity of $g_b$:
\begin{equation*}
c^*(g_b)=c^*(\sym (b))+\sum_{l(b')>l(b)} g_b (b') c^*(b').
\end{equation*}
Each summand $c^*(b')$ in the second term is 0 as $l(c)=1\leq l(b)<l(b')$. As $\sym(b)$ consists of terms of length $l(b)$, $c^*(\sym(b))$ is 0 unless $l(b)=1$, in which case $\sym(b)=b$. Hence $c^*(g_b)=c^*(\sym(b))$ is non-zero only if $c=b$, and $c^*(g_c)=c^*(\sym(c))=c^*(c)=1$.

Turn now to duality. An analogous argument to the polynomial algebra case shows that
$f_{b}\left(g_{b'}\right)\neq0$ only when they have the same eigenvalue, which happens precisely when $b$ and $b'$ have the same number of Lyndon factors. So let $b_{1}...b_{k}$
be the decreasing Lyndon factorization of $b$, and let $b'_{1}...b'_{k}$
be the decreasing Lyndon factorization of $b'$. To evaluate
\begin{equation*}
f_b(g_{b'})=\frac{1}{A'(b)}f_{b_{1}}...f_{b_{k}}\left(\sum_{\sigma\in S_{k}}g_{b'_{\sigma(1)}}...g_{b'_{\sigma(k)}}\right),
\end{equation*}
observe that
\begin{eqnarray*}
f_{b_{1}}...f_{b_{k}}\left(g_{b'_{\sigma(1)}}...g_{b'_{\sigma(k)}}\right) & = & \left(f_{b_{1}}\otimes...\otimes f_{b_{k}}\right)\Delta^{[k]}\left(g_{b'_{\sigma(1)}}...g_{b'_{\sigma(k)}}\right)\\
 & = & \left(f_{b_{1}}\otimes...\otimes f_{b_{k}}\right)\left(\Delta^{[k]}\left(g_{b'_{\sigma(1)}}\right)...\Delta^{[k]}\left(g_{b'_{\sigma(k)}}\right)\right).
\end{eqnarray*}
As $g_{b'_{\sigma(i)}}$ is primitive, each term in $\Delta^{[k]}\left(g_{b'_{\sigma(i)}}\right)$
has $g_{b'_{\sigma(i)}}$ in one tensor-factor and $1$ in the $k-1$
others. Hence the only terms of $\Delta^{[k]}\left(g_{b'_{\sigma(1)}}\right)...\Delta^{[k]}\left(g_{b'_{\sigma(k)}}\right)$
without $1$s in any tensor factors are those of the form $g_{b'_{\tau\sigma(1)}}\otimes...\otimes g_{b'_{\tau\sigma(k)}}$
for some $\tau\in S_{k}$. Now $f_{b_{i}}(1)=0$ for all $i$, so
$f_{b_{1}}\otimes...\otimes f_{b_{k}}$ annihilates any term with
$1$ in some tensor-factor. Hence 
\begin{eqnarray*}
f_{b_{1}}...f_{b_{k}}\left(\sum_{\sigma\in S_{k}}g_{b'_{\sigma(1)}}...g_{b'_{\sigma(k)}}\right) & = & f_{b_{1}}\otimes...\otimes f_{b_{k}}\left(\sum_{\sigma,\tau\in S_{k}}g_{b'_{\tau\sigma(1)}}\otimes...\otimes g_{b'_{\tau\sigma(k)}}\right)\\
 & = & \sum_{\sigma,\tau\in S_{k}}f_{b_{1}}\left(g_{b'_{\tau\sigma(1)}}\right)..f_{b_{k}}\left(g_{b'_{\tau\sigma(k)}}\right).
\end{eqnarray*}
As $f_{b}$ is dual to $g_{b}$ for Lyndon $b$, the only summands which contribute are when $b_{i}=b'_{\sigma\tau(i)}$
for all $i$. In other words, this is zero unless the $b_{i}$ are some permutation
of the $b'_{i}$. But both sets are ordered decreasingly, so this can only happen if $b_{i}=b'_{i}$ for all $i$, hence $b=b'$. In that case, for each fixed $\sigma \in S_k$, the number of $\tau \in S_k$ with $b_i =b_{\sigma \tau (i)}$ for all $i$ is precisely $A'(b)$, so $f_b(g_b)=1$.

The final statement is proved in the same way as in \tref{fefnthm1}, for a polynomial algebra, since, when $b=c_1c_2...c_l$ with $c_1 \geq c_2 \geq ... \geq c_l$ in the ordering of generators, $f_b=\frac{1}{A'(b)}c_1^{*}c_2^{*}...c_l^{*}$.
\end{proof}

\subsection{Stationary distributions, generalized chromatic polynomials, and absorption times}\label{sec35}

This section returns to probabilistic considerations, showing how the left eigenvectors of \ref{sec33} determine the stationary distribution of the associated Markov chain. In the absorbing case, ``generalized chromatic symmetric functions'', based on the universality theorem in \cite{aguiar06}, determine rates of absorption. Again, these general theorems are illustrated in the three sections that follow.

\subsubsection{Stationary distributions}\label{sec351}

The first proposition identifies all the absorbing states when $\calh$ is a polynomial algebra:

\begin{prop}
Suppose $\calh$ is a polynomial algebra where $K_a$, defined by $a^{-n}\Psi^a(b)=\sum_{b'} K_a(b,b')b'$, is a transition matrix. Then the absorbing states are the basis elements $b\in \calb_n$ which are products of $n$ (possibly repeated) degree one elements, and these give a basis of the 1-eigenspace of $K_a$.
\label{absorbing}
\end{prop}

\begin{example}
In the commutative Hopf algebra of graphs in \exrefs{ex21}{ex31}, there is a unique basis element of degree 1 - the graph with a single vertex. Hence the product of $n$ such, which is the empty graph, is the unique absorbing state. Similarly for the rock-breaking example (symmetric functions) on partitions of $n$, the only basis element of degree 1 is $e_1$ and the stationary distribution is absorbing at $1^n$ (or $e_{1^n}$).
\label{ex34}
\end{example}

The parallel result for a cocommutative free associative algebra picks out the stationary distributions:

\begin{prop}
Suppose $\calh$ is a cocommutative and free associative algebra where $K_a$, defined by $a^{-n}\Psi^a(b)=\sum_{b'} K_a(b,b')b'$, is a transition matrix. Then, for each unordered $n$-tuple $\left\{c_1,c_2,\dots,c_n\right\}$ of degree 1 elements (some $c_i$s may be identical), the uniform distribution on $\left\{c_{\sigma(1)}c_{\sigma(2)}\dots c_{\sigma(n)}|\sigma \in S_n \right\}$ is a stationary distribution for the associated chain. In particular, all absorbing states have the form $\bullet ^n$, where $\bullet\in\calb_1$.
\label{stationary}
\end{prop}

\begin{example}
In the free associative algebra $\mathbb{R}\langle x_1,x_2,\dots,x_n\rangle$, each $x_i$ is a degree 1 element. So the uniform distribution on $x_{\sigma(1)}\dots x_{\sigma(n)}$ $(\sigma \in S_n)$ is a stationary distribution, as evident from considering inverse shuffles. 
\label{ex35}
\end{example}

\begin{proof}[Proof of \pref{absorbing}]
From \tref{gefnthm1}, a basis for the 1-eigenspace is $\{g _b | l(b)=n\}$. This forces each factor of $b$ to have degree 1, so $b=c_1 c_2 \dots c_n$ and $g_b=e\left(c_1\right)\dots e\left(c_n\right)$. Now $e(c)=\sum_{a\geq1}\frac{(-1)^{a-1}}{a}m^{[a]}\bar{\Delta}^{[a]}(c)$, and, when $\deg(c)=1$, $m^{[a]}\bar{\Delta}^{[a]}(c)=0$ for all $a\geq 2$. So $e(c)=c$, and hence $g_b=c_1 c_2 \dots c_n=b$, which is a point mass on $b$, so $b=c_1 c_2 \dots c_n$ is an absorbing state.
\end{proof}

\begin{proof}[Proof of \pref{stationary}]
From \tref{gefnthm2}, a basis for the 1-eigenspace is $\{g _b | b\in \calb_n, b$ has $n$ Lyndon factors$\}$. This forces each Lyndon factor of $b$ to have degree 1, so each of these must in fact be a single letter of degree 1. Thus $b=c_1 c_2 \dots c_n$ and $g_b=\sum_{\sigma \in S_n} g_{c_{\sigma(1)}}\dots g_{c_{\sigma(n)}}=\sum_{\sigma \in S_n} c_{\sigma(1)}\dots c_{\sigma(n)}$, as $g_c=c$ for a generator $c$. An absorbing state is a stationary distribution which is a point mass. This requires $c_{\sigma(1)}\dots c_{\sigma(n)}$ to be independent of $\sigma$. As $\calh$ is a free associative algebra, this only holds when $c_1 =\dots = c_n=:\bullet $, in which case $g_b=n!\bullet^n$, so $\bullet^n$ is an absorbing state. 
\end{proof}

\subsubsection{Absorption and chromatic polynomials}\label{sec352}

Consider the case where there is a single basis element of degree 1; call this element $\bullet$ as in \ref{sec32}. Then, by \pref{absorbing} and \pref{stationary}, the $K_a$ chain has a unique absorbing basis vector $\bullet^n\in\calh_n$. The chance of absorption after $k$ steps can be rephrased in terms of an analog of the chromatic polynomial. Note first that the property $K_a\ast K_{a'}=K_{aa'}$ implies it is enough to calculate $K_a(b,\bullet^n)$ for general $a$ and starting state $b\in\calh_n$. To do this, make $\calh$ into a combinatorial Hopf algebra in the sense of \cite{aguiar06} by defining a character $\zeta$ that takes value 1 on $\bullet$ and value 0 on all other generators, and extend multiplicatively and linearly. In other words, $\zeta$ is an indicator function of absorption, taking value 1 on all absorbing states and 0 on all other states. By \cite[Th.~4.1]{aguiar06} there is a unique character-preserving Hopf algebra map from $\calh$ to the algebra of quasisymmetric functions. Define $\chi_b$ to be the quasisymmetric function that is the image of the basis element $b$ under this map. (If $\calh$ is cocommutative, $\chi_b$ will be a symmetric function.) Call this the \textit{generalized chromatic quasisymmetric function} of $b$ since it is the Stanley chromatic symmetric function for the Hopf algebra of graphs \cite{stanley95}. We do not know how difficult it is to determine or evaluate $\chi_b$.
\begin{prop}
With notation as above, the probability of being absorbed in one step of $K_a$ starting from $b$ (that is, $K_a(b,\bullet^n)$) equals
\begin{equation*}
\chi_b\left(\tfrac1{a},\tfrac1{a},\dots,\tfrac1{a},0,0,\dots\right)\qquad\text{(first $a$ arguments are non-zero).}
\end{equation*}
\label{prop32}
\end{prop}
\begin{proof}
By definition of $K_a$, the desired probability $K_a (b,\bullet^n)$ is $a^{-n}$ times the coefficient of $\bullet ^n$ in $\Psi^a(b)$. Every occurrence of $\bullet^n$ in $\Psi^a(b)=m^{[a]}\Delta^{[a]}(b)$ must be due to a term of the form $\bullet^{\alpha_1} \otimes \bullet^{\alpha_2} \otimes \dots \otimes \bullet^{\alpha_a}$ in $\Delta^{[a]}(b)$, for some composition $\alpha=(\alpha_1, \dots, \alpha_n)$ of $n$ (some $\alpha_i$ may be 0). So, letting $\eta_{b}^{b_{1},\dots,b_{a}}$ denote the coefficient of $b_1 \otimes \dots \otimes b_a$ in $\Delta^{[a]}(b)$, 
\begin{equation*}
K_a (b,\bullet^n)=a^{-n}\sum_\alpha \eta_b ^{\bullet^{\alpha_1}, \bullet^{\alpha_2}, \dots, \bullet^{\alpha_a}},
\end{equation*}
where the sum runs over all $\alpha$ with $a$ parts. To re-express this in terms of compositions with no parts of size zero, observe that 
\begin{equation*}
\eta_{b}^{\bullet^{\alpha_1},\dots, \bullet^{\alpha_{a-1}},1}=\eta_{b}^{\bullet^{\alpha_1},\dots, \bullet^{\alpha_{a-1}}},
\end{equation*}
because $\Delta^{[a]}=(\iota\otimes\dots\otimes\iota\otimes\Delta)\Delta^{[a-1]}$ implies $\eta_{b}^{\bullet^{\alpha_1},\dots, \bullet^{\alpha_{a-1}},1}=\sum_{b'}\eta_{b}^{\bullet^{\alpha_1},\dots, \bullet^{\alpha_{a-2}}, b'}\eta_{b'}^{\bullet^{\alpha_{a-1}},1}$, but $\eta_{b'}^{\bullet^{\alpha_{a-1}},1}$ is zero unless $b=\bullet^{\alpha_{a-1}}$. Similar arguments show that $\eta_{b}^{\bullet^{\alpha_1},\dots, \bullet^{\alpha_a}}=\eta_{b}^{\bullet^{\baralpha_1},\dots, \bullet^{\baralpha_{l(\baralpha)}}}$, where $\baralpha$ is $\alpha$ with all zero parts removed. So \begin{equation*}
K_a(b,\bullet^n)=a^{-n}\sum_{\baralpha}\binom{a}{l(\baralpha)}\eta_{b}^{\bullet^{\baralpha_1},\dots, \bullet^{\baralpha_{l(\baralpha)}}},
\end{equation*}
summing over all $\baralpha$ with at most $a$ parts, and no parts of size zero. 

Now, for all compositions $\baralpha$ of $n$ with no zero parts, the coefficient of the monomial quasisymmetric function $M_{\baralpha}$ in $\chi_b$ is defined to be the image of $b$ under the composite 
\begin{equation*}
\calh\xrightarrow{\Delta^{[l(\baralpha)]}}\calh^{\otimes l(\baralpha)}\xrightarrow{\pi_{\baralpha_{1}}\otimes \dots \otimes\pi_{\baralpha_{l(\baralpha)}}}\calh_{\baralpha_{1}}\otimes\dots \otimes\calh_{\baralpha_{l(\baralpha)}}\xrightarrow{\zeta^{l(\baralpha)}}\mathbb{R}
\end{equation*}
where, in the middle map, $\pi_{\baralpha_i}$ denotes the projection to the subspace of degree $\baralpha_i$. As $\zeta$ takes value 1 on powers of $\bullet$ and 0 on other basis elements, it transpires that 
\begin{equation*}
\chi_b=\sum_\alpha \eta_b ^{\bullet^{\baralpha_1}, \bullet^{\baralpha_2}, \dots, \bullet^{\baralpha_{l(\baralpha)}}} M_{\baralpha},
\end{equation*}
summing over all compositions of $n$ regardless of their number of parts. Since $M_{\baralpha} (1,1,\dots 1,0, \dots)=\binom{a}{l(\baralpha)}$, where $a$  is the number of non-zero arguments, it follows that
\begin{equation*}
K_a (b,c_\bullet^n)=a^{-n} \chi_b (1,1, \dots, 1, 0, \dots)=\chi_b\left(\tfrac1{a},\tfrac1{a},\dots,\tfrac1{a},0,0,\dots\right),
\end{equation*}
with the first $a$ arguments non-zero.
\end{proof}

Using a different character $\zeta$, this same argument gives the probability of reaching certain sets of states in one step of the $K_a(b,-)$ chain. This does not require $\calh$ to have a single basis element of degree 1.

\begin{prop}
Let $\calc$ be a subset of generators, and $\zeta^{\calc}$ be the character taking value 1 on $\calc$ and value 0 on all other generators (extended linearly and multiplicatively). Let $\chi^{\calc}_b$ be the image of $b$ under the unique character-preserving Hopf map from $\calh$ to the algebra of quasisymmetric functions. Then the probability of being at a state which is a product of elements in $\calc$, after one step of the $K_a$ chain, starting from $b$, is \begin{equation*}
\chi_b^\calc \left(\tfrac1{a},\tfrac1{a},\dots,\tfrac1{a},0,0,\dots\right)\qquad\text{(first $a$ arguments are non-zero).}
\end{equation*}
\end{prop}

\begin{example}[Rock-breaking]
Recall the rock-breaking chain of \exref{ex32}. Let $\calc=\left\{\hate_1, \hate_2\right\}$. Then $\chi_{\hate_n}^\calc \left(2^{-k},2^{-k},\dots,2^{-k},0,0,\dots\right)$ measures the probability that a rock of size $n$ becomes rocks of size 1 or 2, after $k$ binomial breaks.
\end{example}

\section{Symmetric functions and breaking rocks}\label{sec4}

This section studies the Markov chain induced by the Hopf algebra of symmetric functions. \ref{sec41} presents it as a rock-breaking process with background and references from the applied probability literature. \ref{sec42} gives formulae for the right eigenfunctions by specializing from \tref{fefnthm1} and uses these to bound absorption time and related probabilistic observables. \ref{sec43} gives formulae for the left eigenfunctions by specializing from \tref{gefnthm1} and uses these to derive quasi-stationary distributions.

\subsection{Rock-breaking}\label{sec41}

As in \exrefs{ex12}{ex32}, the Markov chain corresponding to the Hopf algebra of symmetric functions may be described as a rock-breaking process on partitions of $n$: at each step, break each part independently with a symmetric binomial distribution. The chain is absorbed when each part is of size one. Let $P_n(\lambda,\mu)$ be the transition matrix or chance of moving from $\lambda$ to $\mu$ in one step for $\lambda,\mu$ partitions of $n$. For $n=2,3,4$, these matrices are:
\begin{align*}
n=2 & &  n=3 & & n=4 \\
\begin{array}{c|cc}&1^2&2\\\hline
1^2&1&0\\
2&\tfrac12&\tfrac12\end{array} & &
\begin{array}{c|ccc}
&1^3&1 2&3\\\hline
1^3&1&0&0\\
1 2&\tfrac12&\tfrac12&0\\
3&0&\tfrac34&\tfrac14\end{array} & &
\begin{array}{c|ccccc}
&1^4&1^22&2^2&13&4\\\hline
1^4&1&0&0&0&0\\
1^22&\tfrac12&\tfrac12&0&0&0\\
2^2&\tfrac14&\tfrac12&\tfrac14&0&0\\
13&0&\tfrac34&0&\tfrac14&0\\
4&0&0&\tfrac38&\tfrac12&\tfrac18\end{array}
\end{align*}
The $a$th power map $\Psi^a$ yields rock-breaking into $a$ pieces each time, according to a symmetric multinomial distribution. To simplify things, we focus on $a=2$.

The mathematical study of rock-breaking was developed by Kolmogoroff \cite{kolmogorov} who proved a log normal limit for the distribution of pieces of size at most $x$. A literature review and classical applied probability treatment in the language of branching processes is in \cite{athreya}. They allow more general distributions for the size of pieces in each break. A modern manifestation with links to many areas of probability is the study of fragmentation processes. Extensive mathematical development and good pointers to a large physics and engineering literature are in \cite{bertoin03,bertoin06}. Most of the probabilistic development is in continuous time and has pieces breaking one at a time. We have not seen previous study of the natural model of simultaneous breaking developed here.

The rock-breaking Markov chain has the following alternative ``balls in boxes'' description:
\begin{prop}
The distribution on partitions of $n$ induced by taking $k$ steps of the chain $P_n$ starting at the one-part partition $(n)$ is the same as the measure induced by dropping $n$ balls into $2^k$ boxes (uniform multinomial allocation) and considering the partition induced by the box counts in the non-empty cells. For $\lambda$ a partition of $n$, written as $1^{a_1(\lambda)}\cdots n^{a_n(\lambda)}$ with $\sum_{i=1}^nia_i(\lambda)=n$, set $P_n^k((n),\lambda)=\frac{n!2^k}{2^{nk}\prod_{i=0}^n(a_i!)^{a_i}}\frac1{\prod_{i=1}^n(i!)^{a_i}}$. (Here, $a_0(\lambda)$ is the number of empty cells.) Then, for any partition $\mu$
\begin{equation*}
P_n^k(\mu,\lambda)=P_{\mu_1}^k\ast P_{\mu_2}^k\ast\dots\ast P_{\mu_l}^k\qquad\text{(usual convolutions of measures on $\real^n$)}.
\end{equation*}
\label{prop41}
\end{prop}
\begin{rems}
Because of \pref{prop41}, many natural functions of the chain can be understood, as functions of $k$ and $n$, from known properties of multinomial allocation. This includes the distribution of the size of the largest and smallest piece, the joint distribution of the number of pieces of size $i$, and total number of pieces. See \cite{kolchin,barbour}.
\end{rems}

Observe from the matrices for $n=2,3,4$ that, if the partitions are written in reverse-lexicographic order, the transition matrices are lower-triangular. This is because lexicographic order refines the ordering of partitions by refinement. Furthermore, the diagonal entries are the eigenvalues $1/2^i$, as predicted in \pref{acyclicity}. Specializing the counting formula in \tref{gefnthm1} to this case proves the following proposition.
\begin{prop}
The rock-breaking chain $P_n(\lambda,\mu)$ on partitions of size $n$ has eigenvalues $1/2^i,\ 0\leq i\leq n-1$, with $1/2^i$ having multiplicity $p(n,i)$, the number of partitions of $n$ into $i$ parts.
\label{prop42}
\end{prop}

\begin{rem}
As noted by the reviewer, the same rock-breaking chain can be derived using the basis $\{h_\lambda\}$ of complete homogeneous symmetric functions instead of $\{e_\lambda\}$, since the two bases have the same product and coproduct structures. The fact that the dual basis to $\{h_\lambda\}$ is the more recognizable monomial symmetric functions $\{m_\lambda\}$, compared to $\{e_\lambda^*\}$, is unimportant here as the calculations of eigenfunctions do not explicitly require the dual basis.
\end{rem} 

\subsection{Right eigenfunctions}\label{sec42}

Because the operator $P_n(\lambda,\mu)$ is not self-adjoint in any reasonable sense, the left and right eigenfunctions must be developed distinctly. The following description, a specialization of \tref{fefnthm1}, may be supplemented by the examples and corollaries that follow. A proof is at the end of this subsection.
\begin{prop}
Let $\mu,\lambda$ be partitions of $n$. Let $a_i(\lambda)$ denote the number of parts of size $i$ in $\lambda$ and $l(\lambda)$ the total number of parts. Then the $\mu$th right eigenfunction for $P_n$, evaluated at $\lambda$, is
\begin{equation*}
f_{\mu}(\lambda)=\frac{1}{\prod_{i}\mu_{i}!}\sum\prod_{j}\frac{\lambda_{j}!}{a_{1}\left(\mu^{j}\right)!a_{2}\left(\mu^{j}\right)! \dots a_{\lambda_{j}}\left(\mu^{j}\right)!}
\end{equation*}
where the sum is over all sets $\{\mu^j\}$ such that $\mu^j$ is a partition of $\lambda_j$ and the disjoint union $\amalg_j\mu^j=\mu$. The corresponding eigenvalue is $2^{l(\mu)-n}$. $f_\mu(\lambda)$ is always non-negative, and is non-zero if and only if $\mu$ is a refinement of $\lambda$. If $\tilde{\lambda}$ is any set partition with underlying partition $\lambda$, then $f_\mu(\lambda)$ is the number of refinements of $\tilde{\lambda}$ with underlying partition $\mu$.
\label{prop43}
\end{prop}

Here is an illustration of how to compute with this formula:
\begin{example}
For $n=5$, $\mu=(2,1,1,1),\ \lambda=(3,2)$, the possible $\{\mu^j\}$ are
\begin{alignat*}3
\mu^1&=(2,1),&\qquad\mu^2&=(1,1)\\
\mu^1&=(1,1,1),&\qquad\mu^2&=(2).
\end{alignat*}
Then
\begin{equation*}
f_{\mu}(\lambda)=\frac{1}{2!1!1!1!}\left(\frac{3!}{1!1!}\frac{2!}{2!}+\frac{3!}{3!}\frac{2!}{1!}\right)=4
\end{equation*}
\label{ex42}
\end{example}
\begin{example}
For $n=2,3,4$, the right eigenfunctions $f_{\mu}$ are the columns of the matrices:
\begin{align*}
n=2 & & n=3 & & n=4 \\
\begin{array}{cc}1&\tfrac12\\\hline
1&0\\
1&1\end{array} & &
\begin{array}{ccc}
1&\tfrac12&\tfrac14\\\hline
1&0&0\\
1&1&0\\
1&3&1\end{array} & &
\begin{array}{ccccc}
1&\tfrac12&\tfrac14&\tfrac14&\tfrac18\\\hline
1&0&0&0&0\\
1&1&0&0&0\\
1&2&1&0&0\\
1&3&0&1&0\\
1&6&3&4&1\end{array}
\end{align*}
\label{ex41}
\end{example}

For some $\lambda, \mu$ pairs, the formula for $f_\mu (\lambda)$ simplifies.
\begin{example}
When $\lambda=(n)$,
\begin{equation*}
f_{\mu}((n))=\frac{1}{\prod_{i}\mu_{i}!}\frac{n!}{a_{1}\left(\mu\right)!a_{2}\left(\mu\right)! \dots a_{n}\left(\mu\right)!}
\end{equation*}
\label{ex43}
\end{example}

\begin{example}
$f_{1^n} \equiv 1$ has eigenvalue 1.
\end{example} 

\begin{example}
As $(n)$ is not a refinement of any other partition, $f_{(n)}$ is non-zero only at $(n)$. Hence $f_{(n)}(\lambda)=\delta_{(n)}(\lambda)$.
\end{example}

\begin{example}
When $\mu=1^{n-r}r \ (r\neq1)$,
\begin{equation*}
f_\mu(\lambda)=\sum_j\binom{\lambda_j}{r}
\end{equation*}
with eigenvalue $1/2^{r-1}$. Thus $f_{1^{n-2}2}(\lambda)=\sum\binom{\lambda_j}{2}$ with eigenvalue $1/2$ is the unique second-largest eigenfunction.
\label{ex44}
\end{example}

\exref{ex44} can be applied to give bounds on the chance of absorption. The following corollary shows that absorption is likely after $k=2\log_2n+c$ steps.
\begin{cor}
For the rock-breaking chain $X_0=(n),X_1,X_2,\dots$,
\begin{equation*}
P_{(n)}\{X_k\neq1^n\}\leq\frac{\binom{n}2}{2^k}.
\end{equation*}
\label{cor44}
\end{cor}
\begin{proof}
By \exref{ex44}, if $\mu=1^{n-2}2$, then $f_\mu(\lambda)=\sum\binom{\lambda_i}2$ is an eigenfunction with eigenvalue 1/2. Further, $f_\mu(\lambda)$ is zero if and only if $\lambda=1^n$; otherwise $f_\mu(\lambda)\geq1$. Hence
\begin{equation*}
P_{(n)}\{X_k\neq1^n\}=P_{(n)}\{f_\mu(X_k)\geq1\}\leq E_{(n)}\{f_\mu(X_k)\}=\frac{\binom{n}2}{2^k},
\end{equation*}
where the last equality is a simple application of Use A in \ref{sec21}.
\end{proof}
\begin{rems}
\ 
\begin{enumerate}
\item From \pref{prop41} and the classical birthday problem,
\begin{align*}
P_{(n)}\{X_k\neq1^n\}= \prod_{i=1}^{n-1}\left(1-\frac{i}{2^k}\right) &= e^{\sum_{i=1}^{n-1}\log\left(1-\frac{i}{2^k}\right)}
=e^{-\sum_{i=1}^{n-1}\frac{i}{2^k}+O\left(\frac{i^2}{2^{2k}}\right)}
=e^{-\frac{\binom{n}2}{2^k}+O\left(\frac{n^3}{2^{2k}}\right)}.
\end{align*}
It follows that, for $k=2\log_2n+c$ (or $2^k=2^cn^2$) the inequality in the corollary is essentially an equality. 
\item Essentially the same calculations go through for any starting state $\lambda$, since by Use A
\begin{equation*}
E_\lambda\left\{f_\mu(X_k)\right\}=\frac{f_\mu(\lambda)}{2^k}.
\end{equation*}
\item Other eigenfunctions can be similarly used. For example, when $\mu=1^{n-r}r$, $f_\mu(\lambda)=\sum_j\binom{\lambda_j}{r}>0$ if and only if $\max_j\lambda_j\geq r$, and $f_\mu(\lambda)\geq1$ otherwise. It follows as above that
\begin{equation*}
P_{(n)}\{\max_i(X_k)_i\geq r\}=P_{(n)}\{f_\mu(X_k)\geq1\}\leq E_{(n)}\{f_\mu(X_k)\}=\frac{\binom{n}{r}}{2^{(r-1)k}}.
\end{equation*}
\item The right eigenfunctions with $\mu=1^{n-r}r$ can be derived by a direct probabilistic argument. Drop $n$ balls into $2^k$ boxes. Let $N_i$ be the number of balls in box $i$. Then
\begin{equation*}
E_{(n)}\left\{\sum_{i=1}^{2^k}\binom{N_i}{r}\right\}=2^kE_{(n)}\left\{\binom{N_1}{r}\right\}=\frac{\binom{n}{r}}{2^{(r-1)k}}.
\end{equation*}
The last equality follows because $N_1$ is binomial $(n,1/2^k)$ and, if $X$ is binomial $(n,p)$, $E\{X(X-1)\dots(X-r+1)\}=n(n-1)\dots(n-r+1)p^r$. The other eigenvectors can be derived using more complicated multinomial moments.
\end{enumerate}
\end{rems}
\begin{proof}[Proof of \pref{prop43}]
For concreteness, take $l(\lambda)=2$ and $l(\mu)=3$. Then \tref{fefnthm1} states that $f_\mu(\lambda)$ is the coefficient of $\mu_1\otimes\mu_2\otimes\mu_3$ in $\Delta^{[3]}(\lambda)$ (viewing $\mu_i$ as a partition of single part), divided by $\prod_{i}a_{i}(\mu)!$. Recall that 
\begin{equation*}
\Delta^{[3]}(\lambda)=\Delta^{[3]}\left(\lambda_1 \right) \Delta^{[3]}\left(\lambda_2 \right)= \sum_{\substack{i_{1}+j_{1}+k_{1}=\lambda_{1} \\ i_{2}+j_{2}+k_{2}=\lambda_{2}}}\binom{\lambda_{1}}{i_{1}j_{1}k_{1}}\binom{\lambda_{2}}{i_{2}j_{2}k_{2}}i_{1}\amalg i_{2}\otimes j_{1}\amalg j_{2}\otimes k_{1}\amalg k_{2}.
\end{equation*}
So calculating $f_\mu(\lambda)$ requires summing the coefficients of the terms where $i_{1}\amalg i_{2}=\mu_1, \ j_{1}\amalg j_{2}=\mu_2, \ k_{1}\amalg k_{2}=\mu_3$. As $\mu_1$ only has one part, it must be the case that either $i_1=\mu_1$ and $i_2=0$, or $i_1=0$ and $i_2=\mu_1$, and similarly for $\mu_2, \mu_3$. Thus, removing the parts of size 0 from $\left(i_1, j_1, k_1, i_2, j_2, k_2\right)$ and reordering gives $\mu$. So $\binom{\lambda_{1}}{i_{1}j_{1}k_{1}}\binom{\lambda_{2}}{i_{2}j_{2}k_{2}}=\frac{\lambda_{1}!\lambda_{2}!}{\mu_{1}!\mu_{2}!\mu_{3}!}$. Also, if $\mu ^1$ denotes the partition obtained by removing 0s and reordering $\left(i_1, j_1, k_1 \right)$, and $\mu^2$ from $\left(i_2, j_2, k_2 \right)$, then the disjoint union $\mu^1 \amalg \mu^2=\mu$. Given $\mu^1$ and $\mu^2$, the number of different sextuples $\left(i_1, j_1, k_1, i_2, j_2, k_2\right)$ it could have come from is 
\begin{equation*}
\prod_{i}\binom{a_{i}(\mu)}{a_{i}\left(\mu^{1}\right)a_{i}\left(\mu^{2}\right)}.
\end{equation*}
Hence
\begin{equation*}
f_{\mu}(\lambda)=\left(\prod_{i}a_{i}(\mu)!\right)^{-1}\frac{\lambda_{1}!\lambda_{2}!}{\mu_{1}!\mu_{2}!\mu_{3}!}\prod_{i}\binom{a_{i}(\mu)}{a_{i}\left(\mu^{1}\right)a_{i}\left(\mu^{2}\right)},
\end{equation*}
which simplifies as desired. It is then an easy exercise to check that this is the number of refinements of the set partition $\tilde{\lambda}$ of underlying partition $\mu$.
\end{proof}

\subsection{Left eigenfunctions and quasi-stationary distributions}\label{sec43}

This subsection gives two descriptions of the left eigenfunctions: one in parallel with \pref{prop43} and the other using symmetric function theory. Again, examples follow the statement with proofs at the end.
\begin{prop}
For the rock-breaking Markov chain $P_n$ on partitions of $n$, for each partition $\lambda$ of $n$, there is a left eigenfunction $g_\lambda(\mu)$ with eigenvalue $1/2^{n-l(\lambda)}$,
\begin{equation*}
g_{\lambda}(\mu)=\lambda_{1}!\lambda_{2}! \dots \lambda_{l(\lambda)}!\sum\frac{(-1)^{l\left(\mu\right)-l(\lambda)}}{\mu_{1}!\mu_{2}!\dots \mu_{l(\mu)}!}\prod_{j}\frac{\left(l\left(\mu^{j}\right)-1\right)!}{a_{1}\left(\mu^{j}\right)!a_{2}\left(\mu^{j}\right)! \dots a_{\lambda_{j}}\left(\mu^{j}\right)!}
\end{equation*}
where the sum is over sets $\{\mu^j\}$ such that $\mu^j$ is a partition of $\lambda_j$ and $\amalg\mu^j=\mu$. $g_\lambda(\mu)$ is non-zero only if $\mu$ is a refinement of $\lambda$.
\label{prop45}
\end{prop}
As previously, here is a calculational example.
\begin{example}
When $\mu=(2,1,1,1),\ \lambda=(3,2)$,
\begin{equation*}
g_{\lambda}(\mu)=3!2!\frac{(-1)^{2}}{2!1!1!1!}\left(\frac{1!1!}{2!}+\frac{2!0!}{3!}\right)=5
\end{equation*}
\label{ex46}
\end{example}
\begin{example}
For $n=2,3,4$, the left eigenfunctions $g_\lambda$ are the rows of the matrices:
\begin{align*}
n=2 & & n=3 & & n=4 \\
\begin{array}{cc}1&\tfrac12\\\hline
1&0\\
-1&1\end{array} & &
\begin{array}{ccc}
1&\tfrac12&\tfrac14\\\hline
1&0&0\\
-1&1&0\\
2&-3&1\end{array} & &
\begin{array}{ccccc}
1&\tfrac12&\tfrac14&\tfrac14&\tfrac18\\\hline
1&0&0&0&0\\
-1&1&0&0&0\\
1&-2&1&0&0\\
2&-3&0&1&0\\
-6&12&-3&-4&1\end{array}
\end{align*}

Observe that these matrices are the inverses of those in \exref{ex41}, as claimed in \tref{fefnthm1}.
\label{ex45}
\end{example}

The next three examples give some partitions $\lambda$ for which the expression for $g_\lambda (\mu)$ condenses greatly:
\begin{example}
If $\lambda=(n)$ then $g_{(n)}$ is primitive and
\begin{equation*}
g_{(n)}(\mu)=n!\frac{(-1)^{l\left(\mu\right)-1}}{\mu_{1}!\mu_{2}! \dots \mu_{l(\mu)}!}\frac{\left(l\left(\mu\right)-1\right)!}{a_{1}\left(\mu\right)!a_{2}\left(\mu\right)! \dots a_{\lambda_{j}}\left(\mu\right)!}
\end{equation*}
with eigenvalue $1/2^{n-1}$.
\label{ex47}
\end{example}
\begin{example}
If $\lambda=1^{n-r}r \ (r\neq 1)$, $g_\lambda$ has eigenvalue $1/2^{r-1}$ and
\begin{equation*}
g_{\lambda}(\mu)=r!\frac{(-1)^{l\left(\mu\right)-n+r-1}}{\mu_{1}!\mu_{2}! \dots \mu_{l(\mu)}!}\frac{\left(l\left(\mu\right)-n+r-1\right)!}{a_{1}\left(\mu-n+r\right)!a_{2}\left(\mu\right)! \dots a_{\lambda_{j}}\left(\mu\right)!}
\end{equation*}
if $a_1(\mu)\geq a_1(\lambda)$, and 0 otherwise. In particular, $g_{1^{n-2}2}$ puts equal mass at $\mu=1^n$ and $\mu=1^{n-2}2$ with mass 0 for other $\mu$.
\label{ex48a}
\end{example}
\begin{example}
Take $\lambda=1^n$. As no other partition refines $\lambda$, $g_{1^n}(\mu)=\delta_{1^n}(\mu)$, and this is the stationary distribution.
\label{ex48b}
\end{example}

The left eigenfunctions can be used to determine the quasi-stationary distributions $\pi^1,\ \pi^2$ described in \ref{sec21}, Use G.
\begin{cor}
For the rock-breaking Markov chain $P_n$ on partitions of $n$,
\begin{equation*}
\pi^1(\mu)=\pi^2(\mu)=\delta_{1^{n-2}2}(\mu)\qquad\text{for $\mu\neq1^n$}.
\end{equation*}
\label{cor46}
\end{cor}
\begin{proof}
From \eqref{24}, $\pi^1$ is proportional to $g_{1^{n-2}2}$ on the non-absorbing states. The Perron--Frobenius theorem ensures that $\pi^1$ is non-negative. From \exrefs{ex48a}{ex48b}, $\pi^1 (\mu)=\delta_{1^{n-2}2}(\mu)$ for $\mu\neq1^n$. Similarly, $\pi^2$ is proportional to the pointwise product $g_{1^{n-2}2}(\mu)f_{1^{n-2}2}(\mu)=\delta_{1^{n-2}2}(\mu)$ for $\mu\neq1^n$.
\end{proof}

From \cite{geissinger}, the power sum symmetric functions $p_n$ are the primitive elements of the ring of symmetric functions and their products $p_\lambda$  give the left eigenfunctions of the Hopf-power chains. Up to scaling, $p_n$ is the only primitive element of degree $n$, so $p_n$ must be a scalar multiple of $g_{(n)}$. By  \tref{gefnthm1}, $g_{(n)}$ is normalized so that the coefficient of $\hate_n$ is 1 (recall $\hate_n=n!e_n$), whilst the determinant formula \cite[p.~28]{macdonald}:
\begin{equation*}
p_n=\det\begin{pmatrix}
e_1&1&0&\cdots&0\\
2e_2&e_1&1&\cdots&0\\
\vdots&\vdots&\vdots&\vdots&\vdots\\
ne_n&e_{n-1}&e_{n-2}&\cdots&e_1\end{pmatrix}
\end{equation*}
shows that $e_n$ has coefficient $(-1)^{n-1} n$ in $p_n$. Comparing these shows $g_{(n)}=(-1)^{n-1} (n-1)! p_n$, so the left eigenfunctions $g_\lambda$ are $(-1)^{n-l(\lambda)}\prod_{i}\left(\lambda_{i}-1\right)!p_{\lambda}$ expressed in the $\{\hate_\lambda\}$ basis. Hence \pref{prop45} may be rephrased as
\begin{equation*}
p_{\lambda} =(-1)^{n+l\left(\mu\right)} \lambda_{1}\lambda_{2} \dots \lambda_{l(\lambda)}\sum_{\substack{\mu^{j}\vDash\lambda_{j}\\
\amalg\mu^{j}=\mu}
}\prod_{j}\frac{\left(l\left(\mu^{j}\right)-1\right)!}{a_{1}\left(\mu^{j}\right)!a_{2}\left(\mu^{j}\right)! \dots a_{\lambda_{j}}\left(\mu^{j}\right)!}e_{\mu}
\end{equation*}

 For example, for $\lambda=1^n,\ p_1=e_1$, and $p_{1^n}=e_{1^n}=\hate_{1^n}$ corresponding to the left eigenvector $(1,0,\dots,0)$ with eigenvalue 1. For $\lambda=21^{n-2}$, $p_2=e_1^2-2e_2=\hate_1^2-\hate_2$, so $p_{1^{n-2}2}=\hate_{1^n}-\hate_{1^{n-2}2}$ corresponding to the eigenvector $(1,-1,0,\dots,0)$ with eigenvalue 1/2. For $\lambda=1^{n-3}3$, $p_3=\det\left(\begin{smallmatrix}e_1&1&0\\2e_2&e_1&1\\3e_3&e_2&e_1\end{smallmatrix}\right)=e_1^3-3e_1e_2+3e_3=\hate_{1^3}-\frac32\hate_1\hate_2+\frac12\hate_3$. Multiplying by 2 gives the left eigenvector $(2,-3,1,0,\dots,0)$ with eigenvalue 1/4.

The duality $\sum_{\mu}f_{\lambda}(\mu)g_{\nu}(\mu)=\delta_{\lambda\nu}$ implies that 
\begin{equation*}
\hate_{\lambda}=\sum_{\mu}f_{\mu}(\lambda)g_{\lambda}=\sum_{\mu}f_{\mu}(\lambda)(-1)^{n-l(\lambda)}\prod_{i}(\lambda_{i}-1)!p_{\lambda},
\end{equation*}
so \pref{prop43} gives
\begin{equation*}
e_{\lambda} =\frac{(-1)^{n+l\left(\mu\right)}}{\mu_1 \dots \mu_{l(\mu)}} \sum_{\substack{\mu^{j}\vDash\lambda_{j}\\
\amalg\mu^{j}=\mu}
}\prod_{j}\frac{1}{a_{1}\left(\mu^{j}\right)!a_{2}\left(\mu^{j}\right)! \dots a_{\lambda_{j}}\left(\mu^{j}\right)!}p_{\mu},
\end{equation*}
which also follows from \cite[Prop. 7.7.1]{stanley99}.
\begin{proof}[Proof of \pref{prop45}]
$g_\lambda(\mu)$ is the coefficient of $\mu$ in $e\left(\lambda_1\right)e\left(\lambda_2\right)\dots e\left(\lambda_{l(\lambda)}\right)$. Every occurrence of $\mu$ in $e\left(\lambda_1\right) \dots e\left(\lambda_{l(\lambda)}\right)$ is a product of a $\mu^1$ term in $e(\lambda_1)$, a $\mu^2$ term in $e(\lambda_2)$, etc., for some choice of partitions $\mu^j$ of $\lambda_j$ with $\amalg_j\mu^j=\mu$. Hence it suffices to show that the coefficient of a fixed $\mu^j$ in $e(\lambda_j)$ is 
\begin{equation*}
\frac{(-1)^{l(\mu^j)-1}\lambda_j!(l(\mu^j)-1)!}{a_1(\mu^j)!\cdots a_{\lambda_j}(\mu^j)!\mu_1^j!\cdots\mu_{l(\mu)}^j!}.
\end{equation*}
 Recall that $e(\lambda_j)=\sum_{a\geq1}\frac{(-1)^{a-1}}{a}m^{[a]}\bard^{[a]}(\lambda_j)$, and observe that all terms of $m^{[a]}\bard^{[a]}(\lambda_j)$ are partitions with $a$ parts. Hence $\mu^j$ only occurs in the summand with $a=l(\mu^j)$. So the number needed is $\frac{(-1)^{l(\mu^j)-1}}{l(\mu^j)}$ multiplied by the coefficient of $\mu^j$ in $m^{[a]}\bard^{[a]}(\lambda_j)$. Each occurrence of $\mu^j$ in $m^{[a]}\bard^{[a]} \left(\lambda_j\right)$ is caused by $\mu_{\sigma(1)}^j \otimes \dots \otimes \mu_{\sigma(a)}^j$ in $\bard^{[a]} \left(\lambda_j\right)$ for some $\sigma \in S_a$. For each fixed $\sigma$, $\mu_{\sigma(1)}^j \otimes \dots \otimes \mu_{\sigma(a)}^j$  has coefficient 
\begin{equation*}
\binom{\lambda_{j}}{\mu_{\sigma(1)}^j \dots \mu_{\sigma(a)}^j} = \binom{\lambda_{j}}{\mu_1^j \dots \mu_a^j}
\end{equation*}
in $\bard^{[a]} \left(\lambda_j\right)$, and the number of $\sigma \in S_a$ leading to distinct $a$-tuples $\mu_{\sigma(1)}^j, \dots, \mu_{\sigma(a)}^j$ is $\frac{a!}{a_1(\mu^j)!\dots a_{\lambda_j}(\mu^j)!}$. Hence the coefficient of $\mu^j$ in $m^{[a]}\bard^{[a]}\left(\lambda_j\right)$ is
$\binom{\lambda_{j}}{\mu_1^j \dots \mu_{l(\mu^j)}^j}\frac{l(\mu_j)!}{a_1(\mu^j)!\dots a_{\lambda_j}(\mu^j)!}$ as desired.
\end{proof}

\begin{rem}
This calculation is greatly simplified for the algebra of symmetric functions, compared to other polynomial algebras. The reason is that, for a generator $c$, it is in general false that all terms of $m^{[a]}\bard^{[a]}(c)$ have length $a$, or equivalently that all tensor-factors of a term of $\bard^{[a]}(c)$ are generators. See the fourth summand of the coproduct calculation in \exref{exncgraphs} for an example. Then terms of length say, three, in $e(c)$ may show up in both $m^{[2]}\bard^{[2]}(c)$ and $m^{[3]}\bard^{[3]}(c)$, so determining the coefficient of this length three term in $e(c)$ is much harder, due to these potential cancelations in $e(c)$. Hence much effort \cite{fisher, aguiarsottile05, aguiarsottile06} has gone into developing cancelation-free expressions for primitives, as alternatives to $e(c)$.
\end{rem}

\section{The free associative algebra and riffle shuffling}\label{sec5}

This section works through the details for the Hopf algebra $k\langle x_1,x_2,\dots,x_N\rangle$ and riffle shuffling (\exrefs{ex11}{ex13}). \ref{sec51} gives background on shuffling, \ref{sec52} develops the Hopf connection, \ref{sec53} gives various descriptions of right eigenfunctions. These are specialized to decks with distinct cards in \ref{sec54} which shows that the number of descents$-\frac{n-1}2$ and the number of peaks$-\frac{n-2}3$ are eigenfunctions. The last section treats decks with general composition showing that all eigenvalues $1/2^i$, $0\leq i\leq n-1$, occur as long as there are at least two types of cards. Inverse riffle shuffling is a special case of walks on the chambers of a hyperplane arrangement and of random walks on a left regular band. \cite{saliola} and \cite{denham} give a description of left eigenfunctions (hence right eigenfunctions for forward shuffling) in this generality.

\subsection{Riffle shuffles}\label{sec51}

Gilbert--Shannon--Reeds introduced a realistic model for riffle shuffling a deck of $n$ cards. It may be described in terms of a parameterized family of probability measures $Q_a(\sigma)$ for $\sigma$ in the symmetric group $S_n$ and $a\in\{1,2,3,\dots\}$ a parameter. A physical description of the $a$-shuffle begins by cutting a deck of $n$ cards into $a$ piles according to the symmetric multinomial distribution, so the probability of pile $i$ receiving $n_i$ cards is $\binom{n}{n_1n_2\cdots n_a}/a^n$. Then, the piles are riffled together by (sequentially) dropping the next card from pile $i$ with probability proportional to pile size; continuing until all cards have been dropped. Usual riffle shuffles are 2-shuffles and \cite{bayer} show that $Q_a\ast Q_b(\sigma)=\sum_\eta Q_a(\eta)Q_b(\sigma\eta^{-1})=Q_{ab}(\sigma)$. Thus to study $Q_2^{\ast k}(\sigma)=Q_{2^k}(\sigma)$ it is enough to understand $Q_a(\sigma)$. They also found the closed formula
\begin{equation}
Q_a(\sigma)=\binom{n+a-(d(\sigma)+1)}{n}\bigg/a^n,\qquad\text{$d(\sigma)=$ \# descents in $\sigma$.}
\label{51}
\end{equation}
Using this they proved that $\frac32\log_2n+c$ 2-shuffles are necessary and suffice to mix $n$ cards.

The study of $Q_a(\sigma)$ has contacts with other areas of mathematics: to Solomon's descent algebra \cite{solomon,diacfillpit}, quasisymmetric functions \cite{stanley99,stanley01,fulman}, hyperplane arrangements \cite{bidigare,brown,athan}, Lie theory \cite{reut93,reut03}, and, as the present paper shows, Hopf algebras. A survey of this and other connections is in \cite{diaconis} with \cite{assaf,conger10,diacfulmanholmes} bringing this up to date. A good elementary textbook treatment is in \cite{grinstead}.

Of course, shuffling can be treated as a Markov chain on $S_n$ with transition matrix $K_a(\sigma,\pi)=Q_a(\pi\sigma^{-1})$, the chance of moving from $\sigma$ to $\pi$ after one $a$-shuffle. To check later calculations, when $n=3$, the transition matrix is $1/a^3$ times
\begin{equation*}\begin{array}{c|cccccc}
&123&132&213&231&312&321\\\hline
123&\binom{a+2}3&\binom{a+1}3&\binom{a+1}3&\binom{a+1}3&\binom{a+1}3&\binom{a}3\\
132&\binom{a+1}3&\binom{a+2}3&\binom{a+1}3&\binom{a}3&\binom{a+1}3&\binom{a+1}3\\
213&\binom{a+1}3&\binom{a+1}3&\binom{a+2}3&\binom{a+1}3&\binom{a}3&\binom{a+1}3\\
231&\binom{a+1}3&\binom{a}3&\binom{a+1}3&\binom{a+2}3&\binom{a+1}3&\binom{a+1}3\\
312&\binom{a+1}3&\binom{a+1}3&\binom{a}3&\binom{a+1}3&\binom{a+2}3&\binom{a+1}3\\
321&\binom{a}3&\binom{a+1}3&\binom{a+1}3&\binom{a+1}3&\binom{a+1}3&\binom{a+2}3
\end{array}.
\end{equation*}

It is also of interest to study decks with repeated cards. For example, if suits don't matter, the deck may be regarded as having values $1,2,\dots,13$ with value $i$ repeated four times. Now, mixing requires fewer shuffles; see \cite{conger06,conger07,assaf} for details. The present Hopf analysis works here, too.

\subsection{The Hopf connection}\label{sec52}

Let $\calh=k\langle x_1,x_2,\dots,x_N\rangle$ be the free associative algebra on $N$ generators, with each $x_i$ primitive. As explained in \exrefs{ex11}{ex13}, the map $x\rightarrow \Psi^a(x)/a^{\deg x}$ is exactly inverse $a$-shuffling. Observe that the number of cards having each value is unchanged during shuffling. This naturally leads to the following finer grading on the free associative algebra: for each $\nu=(\nu_1,\nu_2, \dots ,\nu_N)\in \mathbb{N}^N$, define $\calh_{\nu}$ to be the subspace spanned by words where $x_i$ appears $\nu_i$ times. The $a$th power map $\Psi^a=m^{[a]}\Delta^{[a]}:\calh\to\calh$ preserves this finer grading. The subspace $\calh_{1^N}\subseteq k\langle x_1,x_2,\dots,x_N\rangle$ is spanned by words of degree 1 in each variable. A basis is $\{x_\sigma=x_{\sigma^{-1}(1)}\cdots x_{\sigma^{-1}(n)}\}$. The mapping $\Psi^a$ preserves $\calh_{1^n}$ and $\frac1{a^n}\Psi^a(x_\sigma)=\sum_\pi Q_a(\pi\sigma^{-1})x_\pi$. With obvious modification the same result holds for any subspace $\calh_{\nu}$. Working on the dual space $\calh^*$ gives the usual Gilbert--Shannon--Reeds riffle shuffles. Let us record this formally; say that a deck has \textit{composition} $\nu$ if there are $\nu_i$ cards of value $i$.
\begin{prop}
Let $\nu=\left(\nu_1,\nu_2, \dots ,\nu_N\right)$ be a composition of $n$. For any $a\in\{1,2,\dots\}$, the mapping $\frac1{a^n}\Psi^a$ preserves $\calh_\nu$ and the matrix of this map in the monomial basis is the transpose of the transition matrix for the inverse $a$-shuffling Markov chain for a deck of composition $\nu$. The dual mapping is the Gilbert--Shannon--Reeds measure \eqref{51} on decks with this composition.
\label{prop51}
\end{prop}

\begin{rem}
Since the cards behave equally independent of their labels, any particular deck of interest can be relabeled so that $\nu_1 \geq \nu_2 \geq \cdots \geq \nu_N$. In other words, it suffices to work with $\calh _\nu$ for partitions $\nu$.
\end{rem}

\subsection{Right eigenfunctions}\label{sec53}

\tref{gefnthm2} applied to the free associative algebra gives a basis of left eigenfunctions for inverse shuffles, which are right eigenfunctions for the forward GSR riffle shuffles. By Remark 2 after \tref{gefnthm2}, each word $w\in \calh_\nu$ corresponds to a right eigenfunction $f_w$ for the GSR measure on decks with composition $\nu$. As explained in \exref{ex13} and \ref{sec23}, these are formed by factoring $w$ into Lyndon words, standard bracketing each Lyndon factor, then expanding and summing over the symmetrization.  The eigenvalue is $a^{k-n}$ with $k$ the number of Lyndon factors of $w$. The following examples should help understanding.
\begin{example}
For $n=3$, with $\nu=1^3$, $\calh_\nu$ is 6-dimensional with basis $\{x_\sigma\}_{\sigma\in S_3}$. Consider $w=x_1x_2x_3$. This is a Lyndon word so no symmetrization is needed. The standard bracketing $\lambda(x_1x_2x_3)=[\lambda(x_1),\lambda(x_2x_3)]=[x_1,[x_2,x_3]]=x_1x_2x_3-x_1x_3x_2-x_2x_3x_1+x_3x_2x_1$. With the labeling of the transition matrix $K_a$ of \eqref{51}, the associated eigenvector is $(1,1,0,-1,0,-1)^T$ with eigenvalue $1/a^2$.

For $n=3, \nu=(2,1),\ \calh_\nu$ is three-dimensional with basis $\{x_1^2x_2,x_1x_2x_1,x_2x_1^2\}$. Consider $w=x_2x_1^2$. This factors into Lyndon words as $x_2\cdot x_1\cdot x_1$; symmetrizing gives the eigenvector $x_1^2x_2+x_1x_2x_1+x_2x_1^2$ or $(1,1,1)^T$ with eigenvalue 1.
\label{ex51}
\end{example}

The description of right eigenvectors can be made more explicit. This is carried foward in the following two sections.

\subsubsection{Right eigenfunctions for decks with distinct values}\label{sec54}

Recall from \ref{sec23} that the values of an eigenfunction $f_w$ at $w'$ can be calculated graphically from the decreasing Lyndon hedgerow $T_w$ of $w$. When $w$ has all letters distinct, this calculation simplifies neatly. To state this, extend the definition of $f_l$ ($l$ a Lyndon word with distinct letters) to words longer than $l$, also with distinct letters: $f_l(w)$ is $f_l$ evaluated on the subword of $w$ whose letters are those of $l$, if such a subword exists, and $0$ otherwise. (Here, a subword always consist of consecutive letters of the original word.) Because $w$ has distinct letters, there is at most one such subword. For example, $f_{35}(14253)=f_{35}(53)=-1$.

\begin{prop}
Let $w$ be a word with distinct letters and Lyndon factorization $l_1l_2\dots l_k$. Then, for all $w'$ with distinct letters and of the same length as $w$, $f_w(w')=f_{l_1}(w')f_{l_2}(w')...f_{l_k}(w')$ and $f_w$ takes only values 1,-1 and 0.
\end{prop}

\begin{example}
$f_{35142}(14253)=f_{35}(14253)f_{142}(14253)=-1\cdot 1=-1$ as calculated in \ref{sec23}.
\label{ex52}
\end{example}

\begin{proof}
Recall from \ref{sec23} that $f_w(w')$ is the signed number of ways to permute the branches and trees of $T_w$ to spell $w'$. When $w$ and $w'$ each consist of distinct letters, such a permutation, if it exists, is unique. This gives the second assertion of the proposition. This permutation is precisely given by permuting the branches of each $T_{l_i}$ so they spell subwords of $w'$. The total number of branch permutations is the sum of the number of branch permutations of each $T_{l_i}$. Taking parity of this statement gives the first assertion of the proposition.
\end{proof}

The example above employed a further shortcut that is worth pointing out: the Lyndon factors of a word with distinct letters start precisely at the \textit{record minima} (working left to right, thus \b35\b142 has minima $3,1$ in positions $1,3$), since a word with distinct letters is Lyndon if and only if its first letter is minimal. This leads to

\begin{prop}
The multiplicity of the eigenvalue $a^{n-k}$ on $\calh_{1,1,...1}$ is $c(n,k)$, the signless Stirling number of the first kind.
\end{prop}

\begin{proof}
By the above observation, the multiplicity of the eigenvalue $a^{n-k}$ on $\calh_{1,1, \dots, 1}$ is the number of permutations with $k$ record minima, which is also the number of permutations with $k$ cycles, by \cite[~Prop. 1.3.1]{stanley97}. 
\end{proof}

\begin{example}[Invariance under reversal]
Let $\bar{w}$ denote the reverse of $w$, then, for any $\sigma$, switching branches at every node shows that $f_\sigma(\bar{w})=\pm f_\sigma(w)$ where the parity is $n-$\# Lyndon factors in $\sigma$. Thus, each eigenspace of $\Psi^a$ is invariant under the map $w\to \bar{w}$. For example, $f_{35142}(35241)=f_{35}(35241)f_{142}(35241)=1\cdot 1 = -f_{35142}(14253)$ when compared with \exref{ex52} above. The quantity ($n-$\# Lyndon factors in $35142$) is $5-2=3$, hence the change in sign. 
\label{ex53}
\end{example}
\begin{example}[Eigenvalue 1]
$\sigma=n,n-1,\dots,1$ is the only word with $n$ Lyndon factors, so $f_\sigma(w)=1$ spans the 1-eigenspace.
\label{ex54}
\end{example}
\begin{example}[Eigenvalue $1/a$ and descents]
There are $\binom{n}2$ permutations $\sigma$ which have $n-1$ Lyndon factors. They may be realized by choosing $i<j$ and taking $\sigma=n,n-1,\dots,j+1,j-1,\dots,i+1,i,j,i-1,\dots,1$. Then, all but $j$ are record minima and the corresponding eigenfunctions are (in the notation at the start of this subsection),
\begin{equation*}
f_{ij}(w)=\begin{cases}1,&\text{if $ij$ occurs as a subword of $w$,}\\-1,&\text{if $ji$ occurs as a subword of $w$,}\\0,&\text{otherwise.}\end{cases}
\end{equation*}
Their sum is $f(w):=\sum_{i<j}f_{ij}(w)=n-1-2d(w)$ with $d(w)$ the number of descents in $w$. This is thus an eigenfunction with eigenvalue $1/a$, as claimed in \exref{ex14}.  This eigenfunction appears in the Markov chain recording the number of descents in successive $a$-shuffles, which is the same as the Markov chain of carries when $n$ numbers are added in base $a$. The transition matrix of this Markov chain is Holte's \cite{holte} ``amazing matrix.'' See \cite{diacfulman09a,diacfulman09b,diacfulman11}. The theory developed there shows that, with $s(n,k)$ the Stirling numbers ($x(x-1)\cdots(x-n+1)=\sum_{k\geq0}s(n,k)x^k$),
\begin{equation}
h_j(w)=n!\sum_{k\geq0}\frac{s(k,n-j)}{k!}\binom{n-d(w)-1}{n-k}
\label{52}
\end{equation}
is a right eigenfunction with eigenvalue $1/a^j,\ 0\leq j\leq n-1$, and the eigenfunction $f$ above is $\frac{2}{n}h_1$.
\label{ex55}
\end{example}
\begin{example}[Eigenvalue $1/a^2$ and peaks]
Recall that a permutation $w$ has a \textit{peak} at position $i,\ 1<i<n$, if $w(i-1)<w(i)>w(i+1)$, and a \textit{trough} at position $i,\ 1<i<n$, if $w(i-1)>w(i)<w(i+1)$. Call the remaining case a \textit{straight}: $w(i-1)<w(i)<w(i+1)$ or $w(i-1)>w(i)>w(i+1)$. The number of peaks and the peak set have been intensively investigated \cite{stembridge,warren}. The following development shows that \# peaks$(w)-\frac{n-2}3$ is an eigenfunction with eigenvalue $1/a^2$, as is \# troughs$(w)-\frac{n-2}3$. Indeed, a basis of this eigenspace is $\{f_\sigma\}$ where $\sigma$ is obtained from $n,n-1,\dots,1$ by removing $j$ and inserting it after $i$ for $i<j$, and then removing $k>j$ and inserting it further down also. Then, all but $j,k$ are record minima. There are three places to insert $k$ (in the examples, $i=1,j=3,k=5$):
\begin{description}
\item [{Case 1}] after $l$ with $l\neq i,\ l\neq j,l<k$ (e.g., 42513);
\item [{Case 2}] after $j$, i.e., $\sigma=n,n-1,\dots,k+1,k-1,\dots,j+1,j-1,\dots,i+1,i,j,k,i-1,\dots,1$ (e.g., 42135);
\item [{Case 3}] after $i$, i.e., $\sigma=n,n-1,\dots,k+1,k-1,\dots,j+1,j-1,\dots,i+1,i,k,j,i-1,\dots,1$ (e.g., 42153).
\end{description}
Then $f_\sigma(w)$ is, respectively
\begin{description}
\item [{Case 1}] 1 if $ij$ and $kl$ both occur as subwords of $w$, or if $ji$ and $lk$ both occur; $-1$ if $ji$ and $kl$ both occur, or if $ij$ and $lk$ both occur; 0 if $i$ is not adjacent to $j$ in $w$ or if $k$ is not adjacent to $l$ (this is $f_{ij}f_{lk}$);
\item [{Case 2}] 1 if $ijk$ or $kji$ occur as subwords of $w$; $-1$ if $ikj$ or $jki$ occur; 0 otherwise (this is $f_{ij}f_{jk}+f_{ik}f_{jk}$);
\item [{Case 3}] 1 if $ikj$ or $jki$ occur as subwords of $w$; $-1$ if $kij$ or $jik$ occur; 0 otherwise (this is $-f_{ik}f_{jk}+f_{ij}f_{ik}$).
\end{description}
\label{ex56}
\end{example}
\begin{prop}
$f_{\wedge}(w):=$ \# peaks in $w-\frac{n-2}3$, $f_{\vee}(w):=$ \# troughs in $w-\frac{n-2}3$ and $f_{-}(w):=$ \# straights in $w-\frac{n-2}3$ are right eigenfunctions with eigenvalue $1/a^2$.
\label{prop53}
\end{prop}
\begin{proof}
Let $s_{r}$ denote the sum of all eigenfunctions arising from case $r$, as
defined in \exref{ex56} above. Note that 
\begin{equation*}
s_{2}=\sum_{i<j<k}f_{ij}f_{jk}+f_{ik}f_{jk}=\mbox{\# straights}-\mbox{\# peaks}
\end{equation*}
\begin{equation*}
s_{3}=\sum_{i<j<k}-f_{ik}f_{jk}+f_{ik}f_{ij}=\mbox{\# peaks}-\mbox{\# troughs}.
\end{equation*}
Since each successive triple in $w$ either forms a straight,
a peak or a trough,
\begin{equation*}
\mbox{\# straights}+\mbox{\# peaks}+\mbox{\# troughs}=n-2.
\end{equation*}
Hence $f_{\wedge}=\frac{1}{3}\left(s_{3}-s_{2}\right)$, $f_{\vee}=\frac{-1}{3}\left(s_{2}+2s_{3}\right)$, $f_{-} = \frac{1}{3}\left(2s_2 +s_3\right)$
are all in the $1/a^{2}$-eigenspace.
\end{proof}

It may be possible to continue the analysis to patterns of longer length; in particular, one interesting open question is which linear combinations of patterns and constant functions give eigenfunctions. 

\subsubsection{Right eigenfunctions for decks with general composition}\label{sec55}

Recall from \pref{prop51} that, for a composition $\nu=(\nu_1,\nu_2,\cdots,\nu_N)$ of $n$, the map $\frac1{a^n}\Psi^a$ describes inverse $a$-shuffling for a deck of composition $\nu$, i.e., a deck of $n$ cards where $\nu_i$ cards have value $i$. \tref{gefnthm2} applies here to determine a full left eigenbasis (i.e., a right eigenbasis for forward shuffles). The special case of $\nu=(n-1,1)$ (follow one labeled card) is worked out in \cite{ciucu} and used to bound the expected number of correct guesses in feedback experiments. His work shows that the same set of eigenvalues $\{1,1/a,1/a^2,\dots,1/a^{n-1}\}$ occur.

This section shows that this is true for all deck compositions (provided $N>1$). It also determines a basis of eigenfunctions with eigenvalue $1/a$ and constructs an eigenfunction which depends only on an appropriately defined number of descents, akin to \exref{ex55}. 

The following proposition finds one ``easy'' eigenfunction for each eigenvalue of $1/{a^k}$. The examples that follow the proof show again that eigenfunctions can correspond to natural observables.
\begin{prop}
Fix a composition $\nu$ of $n$. The dimension of the $1/{a^k}$-eigenspace for the $a$-shuffles of a deck of composition $\nu$ is bounded below by the number of Lyndon words in the alphabet $\{1,2,\dots,N\}$ of length $k+1$ in which letter $i$ occurs at most $\nu_i$ times. In particular, $1/a^k$ does occur as an eigenvalue for each $k,\ 0\leq k\leq n-1$.
\label{prop54}
\end{prop}
\begin{proof}
By remark 2 after \tref{gefnthm2}, the dimension of the $1/{a^k}$-eigenspace is the number of monomials in $\calh_{\nu}$ with $n-k$ Lyndon factors. One way of constructing such monomials is to choose a Lyndon word of length $k+1$ in which letter $i$ occurs at most $\nu_i$ times, and leave the remaining $n-k$ letters of $\nu$ as singleton Lyndon factors. The monomial is then obtained by putting these factors in decreasing order. This shows the lower bound.

To see that $1/a^k$ is an eigenvalue for all $k$, it suffices to construct, for each $k$, a Lyndon word of length $k+1$ in which letter $i$ occurs at most $\nu_i$ times. For $k > \nu_1$, this may be achieved by placing the smallest $k+1$ values in increasing order. For $k\leq \nu_1$, take the word with $k$ 1s followed by a 2.
\end{proof}

\begin{example}
For $\nu=(3,2,1,2)$, the eight eigenfunctions constructed in the last step of the proof correspond to the words shown below in order $1/a^k,\ 0\leq k\leq7$. The bracketed term is the sole non-singleton Lyndon factor:
\begin{equation*}
44322111,\ 443(12)11,\ 4432(112)1,\ 4432(1112),\ 443(11122),\ 44(111223),\ 4(1112234),\ (11122344).
\end{equation*}
\label{ex57}
\end{example}

\begin{example}
For an $n$-card deck of composition $\nu$, the second largest eigenvalue is $1/a$. Our choice of eigenvectors correspond to words with $n-1$ Lyndon factors. Each such word must have $n-2$ singleton Lyndon factors and a Lyndon factor of length 2. Hence the bound in \pref{prop54} is attained; furthermore it can be explicitly calculated: the Lyndon words of length 2 are precisely a lower value followed by a higher value, so the multiplicity of eigenvalue $1/a$ is $\binom{N}2$. This doesn't depend on $\nu$, only on the number $N$ of distinct values. 

Summing these eigenfunctions and arguing as in \exref{ex55} gives
\label{ex59}
\end{example}

\begin{prop}
For any $n$-card deck of composition $\nu$, let $a(w),\ d(w)$ be the number of strict ascents, descents in $w$ respectively. Then $a(w)-d(w)$ is an eigenfunction of $\Psi^a$ with eigenvalue $1/a$.
\label{prop55}
\end{prop}

\begin{proof}
Fix two values $i<j$. Order the deck in decreasing order, then take a card of value $j$ and put it after the first card of value $i$. In other words, let $ij$ be the only non-singleton Lyndon factor. By inspection, the corresponding eigenvector is (up to scaling)
\begin{equation*}
f_{ij}(w)=\{\text{\# subwords $ij$ in $w$}\}-\{\text{\# subwords $ji$ in $w$}\};
\end{equation*}
summing $f_{ij}$ over $1\leq i<j\leq N$ shows that \{\# ascents in $w$\} $-$ \{\# descents in $w$\} is an eigenfunction with eigenvalue $1/a$.
\end{proof}
\begin{rems}
Under the uniform distribution, the expectation of $a(w)-d(w)$ is zero. If initially the deck is arranged in increasing order $w^0,\ a(w^0)-d(w^0)=N-1$. If $w^k$ is the permutation after $k$ $a$-shuffles, the proposition gives $E\{a(w^k)-d(w^k)\}=\frac1{a^k}(N-1)$. Thus for $a=2,\ k=\log_2(N-1)+\theta$ shuffles suffice to make this expected value $2^{-\theta}$. On the other hand, consider a deck with $n$ cards labeled 1 and $n$ cards labeled 2. If the initial order is $w^0=11\cdots12\cdots21,\ a(w^0)-d(w^0)=0$ and so $E\{a(w^k)-d(w^k)\}=0$ for all $k$.

Central limit theorems for the distribution of descents in permutations of multi-sets are developed in \cite{conger07}.
\end{rems}

\begin{example}
Specialize \exref{ex59} to $\nu=(1,n-1)$, so there is one exceptional card of value 1 in a deck of otherwise identical cards of value 2. Then there is a unique eigenfunction $f_{12}$ of eigenvalue $1/a$:  
\begin{equation*}
f_{12}(w)=\begin{cases}1,&\text{if 1 is the top card,}\\-1,&\text{if 1 is the bottom card,}\\0,&\text{otherwise.}\end{cases}
\end{equation*}
\label{ex58}
\end{example}

\section{Examples and counter-examples}\label{sec6}

This section contains a collection of examples where either the Hopf-square map leads to a Markov chain with a reasonable ``real world'' interpretation --- Markov chains on simplicial complexes and quantum groups --- or the constructions do not work out to give Markov chains --- a quotient of the symmetric functions algebra, Sweedler's Hopf algebra and the Steenrod algebra. Further examples will be developed in depth in future work.
\begin{example}[A Markov chain on simplicial complexes]
Let $\calx$ be a finite set and $\calc$ a simplicial complex of subsets of $\calx$. Recall that this means that $\calc$ is a collection of non-empty subsets of $\calx$ such that $c\in\calc$ implies that all non-empty subsets of $c$ are in $\calc$. As an example, consider the standard triangulation of the torus into 18 triangles:
\begin{center}
\includegraphics[scale=1.0, clip]{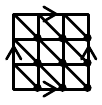}
\end{center}
\noindent
Here the top and bottom edges are identified as are the left and right sides. This identifies several vertices and edges and $\calx$ consists of nine distinct vertices. The complex $\calc$ contains these nine vertices, the 24 distinct edges, and the 18 triangles.

The set of all simplicial complexes (on all finite sets $\calx$) is a basis for a Hopf algebra, under disjoint union as product and coproduct
\begin{equation*}
\Delta(\calc_{\calx})=\sum_{S\subseteq\calx}\calc_S\otimes\calc_{S^\calc}
\end{equation*}
with the sum over subsets $S\subseteq\calx$, and $\calc_S=\{A\subseteq S:A\in\calc\}$. By convention $\calc_\emptyset=1$ in this Hopf algebra so $S=\emptyset$ is allowed in the sum. Graded by $|\calx|$, this gives a commutative, cocommutative Hopf algebra with basis given by all complexes $\calc$; the generators are the connected complexes.

The associated Markov chain, restricted to complexes on $n$ vertices, is simple to describe: from a given complex $\calc$, color the vertices red or blue, independently, with probability 1/2. Take the disjoint union of the complex induced by the red vertices with the complex induced by the blue vertices. As usual, the process terminates at the trivial complex consisting of $n$ isolated vertices.
\label{ex61}
\end{example}

This Markov chain is of interest in quantifying the results of a ``topological statistics'' analysis. There, a data set ($n$ points in a metric space) gives rise to a family of complexes $\calc_\epsilon,\ 0\leq\epsilon<\infty$, where the vertices of each $\calc_\epsilon$ are the data points, and $k$ points form a simplex if the intersection of the $\epsilon$ balls around each point is non-empty in the ambient space. For $\epsilon$ small, the complex is trivial. For $\epsilon$ sufficiently large, the complex is the $n$-simplex. In topological statistics \cite{carlsson06,carlsson09} one studies things like the Betti numbers of $\calc_\epsilon$ as a function of $\epsilon$. If these are stable for a range of $\epsilon$ this indicates interpretable structure in the data.

Consider now a data set with large $n$ and $\epsilon$ fixed. If a random subset of $k$ points is considered (a frequent computational ploy) the induced sub-complex essentially has the distribution of the ``painted red'' sub-complex (if the painting is done with probability $k/n$). Iterating the Markov chain corresponds to taking smaller samples.

If the Markov chain starts out at the $n$-simplex, every connected subset of the resulting Markov chain is a simplex. Thus, at each stage, all of the higher Betti numbers are zero and $\beta_0$ after $k$ steps is $2^k-X_k$, where $X_k$ is distributed as the number of empty cells if $n$ balls are dropped into $2^k$ boxes. This is a thoroughly studied problem \cite{kolchin}. The distribution of the Betti numbers for more interesting starting complexes is a novel, challenging problem. Indeed, consider the triangulation of the torus with $2n^2$ initial triangles. Coloring the vertices red or blue with probability 1/2, the edges with red/red vertices are distributed in the same way as the ``open sites'' in site percolation on a triangular lattice. Computing the Betti number $\beta_0$ amounts to computing the number of connected components in site percolation. In the infinite triangular lattice, it is known that $p=1/2$ is the critical threshold and at criticality, the chance that the component containing the origin has size greater than $k$ falls off as $k^{-5/48}$. These and related facts about site percolation are among the deepest results in modern probability. See \cite{grimmett,werner,schramm} for background and recent results. Iterates of the Markov chain result in site percolation with $p$ below the critical value but estimating $\beta_0$ is still challenging.

It is natural to study the absorption of this chain started at the initial complex $\calc$. This can be studied using the results of \ref{sec35}.
\begin{prop}
Let the simplicial complex Markov chain start at the complex $\calc$. Let $G$ be the graph of the 1-skeleton of $\calc$. Suppose that the chromatic polynomial of $G$ is $p(x)$. Then the probability of absorption after $k$ steps is $p(2^k)/2^{nk}$ (with $n=|\calx|$).
\label{prop61}
\end{prop}

For example, if $\calc$ is the $n$-simplex, $p(x)=x(x-1)\cdots(x-n+1)$ and $P$\{absorption after $k$ steps\}$=\prod_{i=1}^{n-1}(1-i/2^k)\sim e^{-2^{-(2c-1)}}$ if $k=2(\log_2n+c)$ for $n$ large. If $\calc$ is a tree, $p(x)=x(x-1)^{n-1}$ and $P$\{absorption after $k$ steps\}$=(1-1/2^k)^{n-1}\sim e^{-2^{-c}}$ if $k=\log_2n+c$ for $n$ large.

Using results on the birthday problem in non-standard situations \cite{barbour, chatterjee} it is possible to do similar asymptotics for variables such as the number of $l$-simplices remaining after $k$ steps for more interesting starting $\calc$ such as a triangulation of the torus into $2(n-1)^2$ triangles.

As a final remark, note that the simplicial complex Markov chain induces a Markov chain on the successive 1-skeletons. The eigenvectors of this Markov chain are beautifully developed in \cite{fisher}. These all lift to eigenvectors of the complex chain, so much is known.
\begin{example}[Quantized shuffle algebras]
It is natural to seek useful deformations of processes like riffle shuffling. One route is via the quantized shuffle algebras of \cite{green95,green97} and \cite{rosso95,rosso97,rosso98}. These have become a basic object of study \cite{leclerc,kleshchev}. Consider the vector space $k\langle x_1,\dots,x_n\rangle$, and equip its degree 1 subspace with a symmetric $\mathbb{Z}$ form $x_i\cdot x_j$. Turn this into an algebra with the product of concatentation. Take as coproduct $\Delta(x_i)=1\otimes x_i+x_i\otimes1$. However, $\Delta$ is to be multiplicative with respect to the twisted tensor product $(x_1\otimes x_2)(y_1\otimes y_2)=q^{x_2\cdot y_1}(x_1x_2\otimes y_1y_2)$. Green translates this into shuffling language. The upshot is if $w=x_{i_1}x_{i_2}\cdots x_{i_k}$ is a word in $k\langle x_1,\dots,x_n\rangle$ then
\begin{equation*}
m\Delta(w)=\sum_{S\subseteq\{1,2,\dots,k\}}q^{\wt(S,w)}w_Sw_{S^{\calc}}.
\end{equation*}
Here the sum is over all subsets (including the empty set), $w_Sw_{S^{\calc}}$ is the inverse shuffle moving the letters in the positions marked by $S$ to the front (keeping them in the same relative order). The weight $\wt(S,w)$ is the sum of $x_{j'}\cdot x_j$ for $j'\in S^{\calc},\ j\in S,\ j'<j$. Thus if $w=ijklm$ and $S=\{2,4\}$, we have $w_Sw_{S^{\calc}}=jlikm$, and $\wt(S,w)=i\cdot j+i\cdot l+k\cdot l$.

When $q=1$ this shuffle product gives Ree's shuffle algebra and general values of $q$ lead to elegant combinatorial formulations of quantum groups. For general $q>0$, there is also a naturally associated Markov chain. Work on the piece with multi-grading $1^n$, so each variable appears once and we may work with permutations in $S_n$. For a starting permutation $\pi$, and $0\leq j\leq n$, let $\theta(j)=\sum_{|S|=j}q^{\wt(S,\pi)}$. Set $\theta=\sum_{j=0}^n\theta(j)$. Choose $j$ with probability $\theta(j)/\theta$ and then $S$ (with $|S|=j$) with probability $q^{\wt(S,\pi)}/\theta(j)$. Move to $\pi_S\pi_{S^\calc}$. This defines a Markov transition matrix $K_q(\pi,\pi')$ via $\frac1{\theta}m\Delta(\pi)$. Note that the normalization $\theta$ depends on $\pi$.
\label{ex62}
\end{example}

We have not seen our way through this to nice mathematics. There is one case where progress can be made: suppose $x_i\cdot x_j\equiv1$. Then $\wt(S,\pi)=\inv(S)$, the minimum number of pairwise adjacent transpositions needed to move $S$ to the left. (When $n=5$ and $S=\{2,4\},\ \inv(S)=3$.) Since $\wt(S,\pi)$ doesn't depend on $\pi$, neither does $\theta$ and the Markov chain becomes a random walk on $S_n$ driven by the measure
\begin{equation*}
\mu(\sigma)=\begin{cases}q^{\inv(\sigma)}/z_n,&\text{if $\sigma$ has a single descent,}\\0,&\text{otherwise,}\end{cases}
\end{equation*}
where $z_n$ is a normalizing constant.

The preceding description gives inverse riffle shuffles. It is straightforward to describe the $q$-analog of forward riffle shuffles by taking inverses. Let $[j]_q=1+q+\dots+q^{j-1}$, $[j]_q !=[j]_q [j-1]_q \dots [1]_q$, and $\nk= \frac{[n]_q !}{[k]_q ![n-k]_q !}$, the usual $q$-binomial coefficient. We write $I(w)$ for the number of inversions of the permutation $w$ and $R(w)=d(w^{-1})+1$ for the number of rising sequences.
\begin{prop}
For $q>0$, the $q$-riffle shuffle measure has the following description on $S_n$:
\begin{equation}
Q_q(w)=\begin{cases}q^{I(w)}/z_n,&\text{if $R(w)\leq2$,}\\
0,&\text{otherwise,}\end{cases}
\label{61}
\end{equation}
where $z_n$ is the normalizing constant $\sum_{w:R(w)\leq2}q^{I(w)}$. To generate $w$ from $Q_q$, cut off $j$ cards with probability $\nj/z_n$ and drop cards sequentially according to the following rule: if at some stage there are $A$ cards in the left pile and $B$ cards in the right pile, drop the next card
\begin{equation}
\text{from left with probability}\quad\frac{q^B[A]_q}{[A+B]_q}\qquad
\text{and from right with probability}\quad\frac{[B]_q}{[A+B]_q}.
\label{62}
\end{equation}
Continue, until all cards have been dropped.
\label{prop6x}
\end{prop}
\begin{proof}
Equation \eqref{61} follows from the inverse description because $I(w)=I(w^{-1})$. For the sequential description, it is classical that $\nj$ is the generating function for multi-sets containing $j$ ones and $n-j$ twos by number of inversion \cite[Sect.~1.7]{stanley97}. For two piles with say, $1,2,\dots,j$ in the left and $j+1,\dots,k$ in the right, in order, dropping $j$ induces $k-j$ inversions. Multiplying the factors in \eqref{62} results in a permutation with $R(w)\leq2$, with probability $q^{I(w)}/\nj$. Since the cut is made with probability $\nj/z_n$, the two-stage procedure gives \eqref{61}.
\end{proof}
\begin{rems}
When $q=1$, this becomes the usual Gilbert--Shannon--Reeds measure described in the introduction and in \ref{sec5}. In particular, for the sequential version, the cut is $j$ with probability $\binom{n}{j}/2^n$ and with $A$ in the left and $B$ in the right, drop from left or right with probability $\frac{A}{A+B},\ \frac{B}{A+B}$ respectively. For general $q$, as far as we know, there is no closed form for $z_n$. As $q\to\infty$, the cutting distribution is peaked at $n/2$ and most cards are dropped from the left pile. The most likely permutation arises from cutting off $n/2$ cards and placing them at the bottom. As $q\to0$, the cutting distribution tends to uniform on $\{0,1,\dots,n\}$ and most cards are dropped from the right pile. The most likely permutation is the identity. There is a natural extension to a $q$-$a$-shuffle with cards cut into $a$ piles according to the $q$-multinomial distribution and cards dropped sequentially with probability ``$q$-proportional'' to packet size. 
\end{rems}

We hope to analyze this Markov chain in future work. See \cite{diacram} for related $q$-deformations of a familiar random walk. 

\begin{example}[A quotient of the algebra of symmetric functions]
Consider $\bar{\Lambda} = \Lambda / e_1 =k[e_2, e_3, \dots]$, with 
\begin{equation*}
\Delta(e_n)= 1 \otimes e_n + \sum_{j=2}^{n-2} e_j \otimes e_{n-j} + e_n \otimes 1.
\end{equation*}
This is the Weyl group invariants of type A. It is a polynomial algebra, so the theory in \ref{sec3} generates an eigenbasis of $\Psi^a$ and $\Psi^{*a}$. One hopes this will induce a rock-breaking process where pieces of size one are not allowed; however, we cannot rescale the basis via \tref{thm31} to obtain a transition matrix, as both $e_2$ and $e_3$ are primitive basis elements of degree greater than one. Hence $e_2^3,\ e_3^2$ have the same degree, but $m\Delta(e_2^3)=8e_2^3$ and $m\Delta(e_3^2)=4e_3^2$, so no rescaling can make the sum of the cofficients of $m\Delta(e_2^3)$ equal to that of $m\Delta(e_3^2)$.
\label{quotsym}
\end{example}

\begin{example}[Sweedler's example]
The four-dimensional algebra
\begin{equation*}
H_4=k(1,g,x,gx:g^2=1,\ x^2=0,\ xg=-gx)
\end{equation*}
becomes a Hopf algebra with $\Delta(g)=g\otimes g,\ \Delta(x)=x\otimes 1+g\otimes x,\ \epsilon(g)=1,\ \epsilon(x)=0$. The antipode is $s(g)=g^{-1},\ s(x)=-gx$. It is discussed in \cite{montgomery} as an example of a Hopf algebra that is neither commutative nor cocommutative. With the given basis $\{1,g,x,gx\},\ m\Delta(1)=1,\ m\Delta(g)=1,\ m\Delta(x)=x-gx,\ m\Delta(gx)=-x+gx$. The negative coefficients forstall our efforts to find a probabilistic interpretation. The element $v=x-gx$ is an eigenvector for $m\Delta$ with eigenvalue 2 and high powers of $\frac12 m\Delta$ applied to a general element $a+bg+cx+dgx$ converge to $(c-d)v$. Of course, this example violates many of our underlying assumptions. It is not graded and is neither a polynomial algebra nor a free associative algebra.
\label{ex63}
\end{example}
\begin{example}[The Steenrod algebra]
Steenrod squares (and higher powers) are a basic tool of algebraic topology \cite{hatcher}. They give rise to a Hopf algebra $A_2$ over $\mathbb{F}_2$. Its dual $A_2^*$ is a commutative, noncocommutative Hopf algebra over $\mathbb{F}_2$ with a simple description. As an algebra, $A_2^*=\mathbb{F}_2[x_1,x_2,\dots]$ (polynomial algebra in countably many variables) graded with $x_i$ of degree $2^i-1$. The coproduct is $\Delta(x_n)=\sum_{i=0}^nx_{n-i}^{2^i}\otimes x_i$ ($x_0=1$). Alas, because the coefficients are mod 2, we have been unable to find a probabilistic interpretation of $m\Delta$. For example, $(A_2^*)_3$ has basis $\{x_1^3,x_2\}$ and $m\Delta(x_1^3)=0,\ m\Delta(x_2)=x_1^3$ so $(m\Delta)^2\equiv0$. Of course, high powers of operators can be of interest without positivity \cite{pd84,guralnick}.
\label{ex64}
\end{example}

\newcommand{\etalchar}[1]{$^{#1}$}
\def\cprime{$'$} \def\cprime{$'$} \def\cprime{$'$}

\end{document}